\documentclass[a4paper,12pt]{article}
\usepackage{amssymb,amsmath,amsthm,times,mathrsfs,fullpage}
\usepackage{tikz,extarrows,pifont,yfonts,shadow}
\usetikzlibrary{positioning,shapes,mindmap}
\usetikzlibrary{scopes,snakes,arrows}
\usetikzlibrary{mindmap,backgrounds}
\usepackage[pdftex]{hyperref}
\hypersetup{plainpages=True, pdfstartview=FitH, bookmarksopen=true,
pdfpagemode=none, colorlinks=true,linkcolor=blue,citecolor=blue}

\def\le{\leqslant}
\def\ge{\geqslant}
\def\dd#1{{\,\mathrm d}#1}
\def\ve{\varepsilon}
\def\tr#1{\left\lfloor #1\right\rfloor}
\def\pf{\noindent {\em Proof.}\ }
\def\qed{{\quad\rule{1mm}{3mm}\,}}
\def\JS{\mathscr{J\!\!S}}

\newtheorem{thm}{Theorem}[section]

\newtheorem{prop}[thm]{Proposition}
\newtheorem{Def}{Definition}

\title{An analytic approach to the asymptotic variance
of trie statistics and related structures}

\author{Michael Fuchs\\
    Department of Applied Mathematics\\
    National Chiao Tung University\\
    Hsinchu, 300\\ Taiwan
\and Hsien-Kuei Hwang \\
    Institute of Statistical Science\\
    Institute of Information Science\\
    Academia Sinica\\
    Taipei 115\\
    Taiwan
\and Vytas Zacharovas \\
    Dept. Mathematics \&\ Informatics\\
    Vilnius University\\
    Lithuania}
\date{\today}

\begin{document}
\maketitle

\centerline{\emph{Dedicated to the memory of Philippe Flajolet}}

\begin{abstract}
We develop analytic tools for the asymptotics of general trie
statistics, which are particularly advantageous for clarifying the
asymptotic variance. Many concrete examples are discussed for which
new Fourier expansions are given. The tools are also useful for
other splitting processes with an underlying binomial distribution.
We specially highlight Philippe Flajolet's contribution in the
analysis of these random structures.
\end{abstract}

\section{Introduction}

Coin-flipping is one of the simplest ways of resolving a conflict,
deciding between two alternatives, and generating random phenomena.
It has been widely adopted in many daily-life situations and
scientific disciplines. There exists even a term ``flippism.'' The
curiosity of understanding the randomness behind throwing coins or
dices was one of the motivating origins of early probability theory,
culminating in the classical book ``\emph{Ars Conjectandi}'' by
Jacob Bernoulli, which was published exactly three hundred years ago
in 1713 (many years after its completion; see
\cite{polasek00a,shafer96a}). When flipped successively, one
naturally encounters the binomial distribution, which is pervasive
in many splitting processes and branching algorithms whose analysis
was largely developed and clarified through Philippe Flajolet's
works, notably in the early 1980s, an important period marking the
upsurgence of the use of complex-analytic tools in the Analysis of
Algorithms.

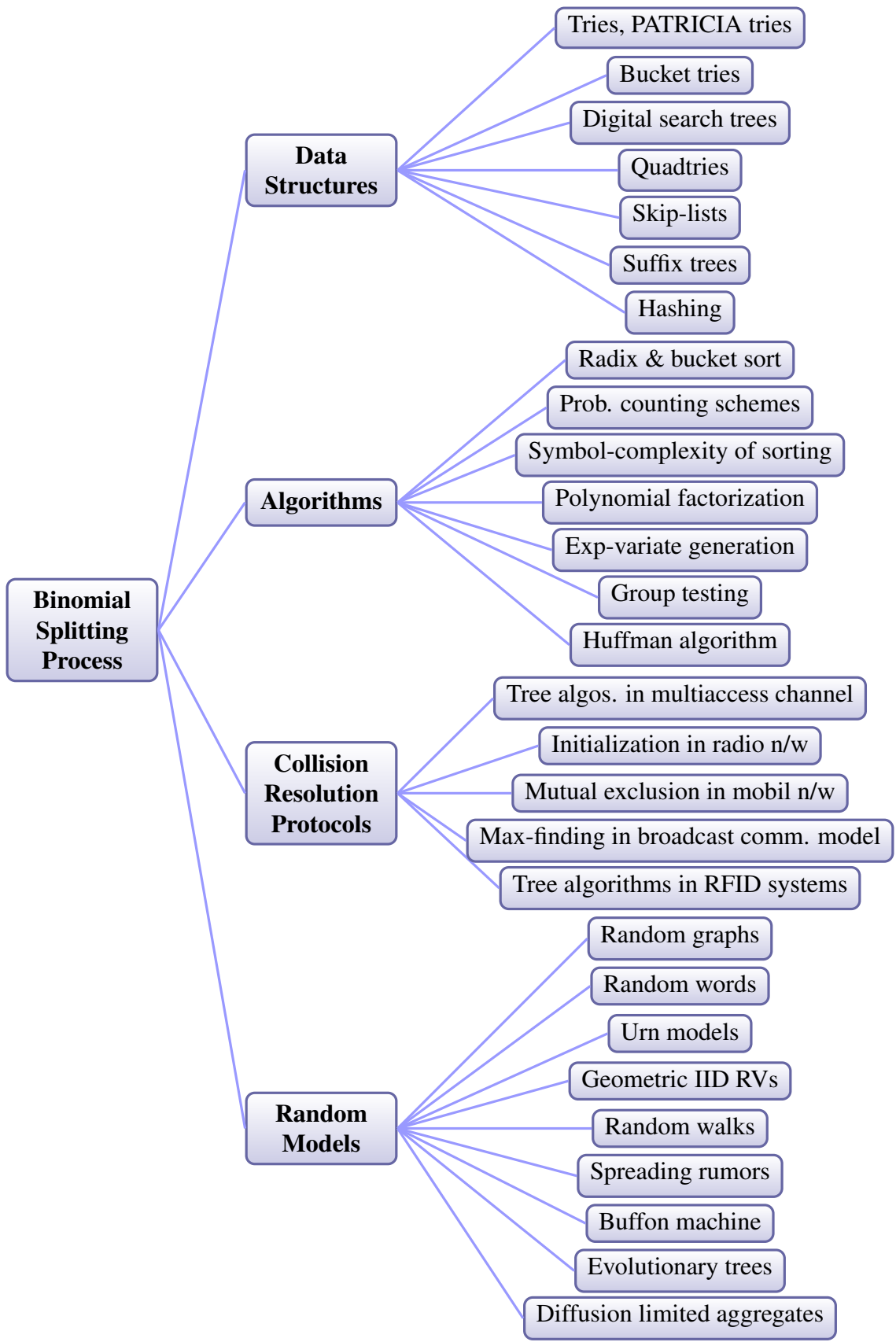
\begin{figure}
\begin{center}
\begin{tikzpicture}[
    every tree node/.style={anchor=base west},
    grow=right,
    level 1/.style={sibling distance=4.5cm,level distance=4cm},
    level 2/.style={sibling distance=0.8cm, level distance=6cm},
    edge from parent/.style={very thick,draw=blue!40!white,
        shorten >=0pt, shorten <=0pt},
    edge from parent path={(\tikzparentnode.east) --
    (\tikzchildnode.west)},
    every node/.style={text centered, inner sep=2mm},
    punkt/.style={rectangle, rounded corners, shade, top color=white,
    bottom color=blue!50!black!20, draw=blue!40!black!60, very
    thick }
    ]
\node[punkt, text width=5em] {\bf Binomial Splitting Process}
    child[sibling distance=5.6cm] {
        node[punkt, rectangle,text width=5em] {\bf Random Models}
        child {
            node[punkt, text height=0.5em]
            {Diffusion limited aggregates}
        }
        child {
            node[punkt, text height=0.5em] {Evolutionary trees}
        }
        child {
            node[punkt, text height=0.5em] {Buffon machine}
        }
        child {
            node[punkt, text height=0.5em] {Spreading rumors}
        }
        child {
            node[punkt, text height=0.5em] {Random walks}
        }
        child {
            node[punkt, text height=0.5em] {Geometric IID RVs}
        }
        child {
            node[punkt, text height=0.5em] {Urn models}
        }
        child {
            node[punkt, text height=0.5em] {Random words}
        }
        child {
            node[punkt, text height=0.5em] {Random graphs}
        }
    }
    child[sibling distance=5.5cm] {
            node[punkt, text width=5em]
            {\bf Collision Resolution Protocols}
        child {
            node[punkt, text height=0.5em]
            {Tree algorithms in RFID systems}
        }
        child {
            node[punkt, text height=0.5em]
            {Max-finding in broadcast comm. model}
        }
        child {
            node[punkt, text height=0.5em]
            {Mutual exclusion in mobil n/w}
        }
        child {
            node[punkt, text height=0.5em]
            {Initialization in radio n/w}
        }
        child {
            node[punkt, text height=0.5em]
            {Tree algos. in multiaccess channel}
        }
        }
    child[sibling distance=4.5cm] {
        node[punkt, rectangle,text width=5em,yshift=-.1cm]
        {\bf Algorithms}
        child {
            node[punkt, text height=0.5em] {Huffman algorithm}
        }
        child {
            node[punkt, text height=0.5em] {Group testing}
        }
        child {
            node[punkt, text height=0.5em] {Exp-variate generation}
        }
        child {
            node[punkt, text height=0.5em] {Polynomial factorization}
        }
        child {
            node[punkt, text height=0.5em]
            {Symbol-complexity of sorting}
        }
        child {
            node[punkt, text height=0.5em]
            {Prob. counting schemes}
        }
        child {
            node[punkt, text height=0.5em] {Radix \& bucket sort}
        }
    }
    child[sibling distance=4.7cm] {
        node[punkt, rectangle,text width=5em,yshift=.7cm]
        {\bf Data Structures}
        child {
            node[punkt, text height=0.5em] {Hashing}
        }
        child {
            node[punkt, text height=0.5em] {Suffix trees}
        }
        child {
            node[punkt, text height=0.5em] {Skip-lists}
        }
        child {
            node[punkt, text height=0.5em] {Quadtries}
        }
        child {
            node[punkt, text height=0.5em] {Digital search trees}
        }
        child {
            node[punkt, text height=0.5em] {Bucket tries}
        }
        child{
            node[punkt,text height=0.5em,xshift=0]
            {Tries, PATRICIA tries}
        }
    };
\end{tikzpicture}
\end{center}
\vspace*{-.3cm} \caption{\emph{A tree rendering of the diverse
themes pertinent to binomial splitting processes.}}\label{fig-bsps}
\end{figure}

\paragraph{Technical content of this paper.}
This paper is a sequel to \cite{hwang10a} and we will develop an
analytic approach that is especially useful for characterizing the
asymptotics of the mean and the variance of additive statistics of
random tries under the Bernoulli model; such statistics can often
be computed recursively by
\begin{align} \label{Xn-rr}
    X_n \stackrel{d}{=} X_{I_n}+ X_{n-I_n}^* + T_n,
\end{align}
with suitable initial conditions, where $T_n$ is known, $X_n^*$ is
an independent copy of $X_n$ and $I_n$ is the binomial distribution
with mean $pn$, $0<p<1$.

Many asymptotic approximations are known in the literature for the
variance of $X_n$, which has in many cases of interest the pattern
\begin{align} \label{var-Xn}
    \frac{\mathbb{V}(X_n)}{n} = c\log n +c'
    + \left\{\begin{array}{ll}
        P(\varpi\log n),& \text{if }
        \frac{\log p}{\log q}\in\mathbb{Q}\\
        0, & \text{if }\frac{\log p}{\log q}\not\in\mathbb{Q},
    \end{array}\right\} +o(1),
\end{align}
where $c$ may be zero, $\varpi$ depends on the ratio $\frac{\log
p}{\log q}$ and $P(x)=P(x+1)$ is a bounded periodic function.
However, known expressions in the literature for the
periodic function $P$ are rare due to the complexity of the problem,
and are often either less transparent, or less explicit, or too messy
to be stated. In many situations they are given in the form of
one periodic function minus the square of the other.
Our approach, in contrast, provides not only a systematic derivation
of the asymptotic approximation \eqref{var-Xn} but also a simpler,
explicit, independent expression for $P$, notably in the symmetric
case ($p=q$). Further refinement of the
$o(1)$-term lies outside the scope of this paper and can be dealt
with by the approach developed by Flajolet et al.\ in \cite{PF208}.

\paragraph{Binomial splitting processes.}
In general, the simple splitting idea behind the recursive random
variable \eqref{Xn-rr} ($0$ goes to the left and $1$ goes to the
right) has also been widely adopted in many different modeling
processes, which, for simplicity, will be vaguely referred to as
``\emph{binomial splitting processes}'' (BSPs), where binomial
distribution and some of its extensions are naturally involved in
the analysis; see Figure~\ref{fig-bsps} for concrete examples of
BSPs that are related to our analysis here. For convenience of
presentation, we roughly group these structures in four categories:
Data Structures, Algorithms, Collision Resolution Protocols, and
Random Models.

To see the popularity of BSPs in different areas, we start from the
recurrence ($q=1-p$)
\begin{align} \label{an-rr}
    a_n = \sum_{0\le k\le n}\pi_{n,k}\left(
    a_k + a_{n-k}\right) + b_n, \quad\text{where}
    \quad\pi_{n,k} := \binom{n}{k}p^k q^{n-k},
\end{align}
which results, for example, from \eqref{Xn-rr} by taking
expectation. Here the ``toll-function'' $b_n$ may itself involve
$a_j$ ($j=0,1,\dots$) but with multipliers that are exponentially
small.

From an analytic point of view, the trie recurrence \eqref{an-rr}
translates for the Poisson generating function
\begin{align}\label{pgf}
    \tilde{f}(z) := e^{-z}\sum_n \frac{a_n}{n!}\, z^n,
\end{align}
into the \emph{trie functional equation}
\begin{align} \label{an-pgf}
    \tilde{f}(z)=\tilde{f}(pz)+\tilde{f}(qz)+\tilde{g}(z),
\end{align}
with suitable initial conditions. Such a functional equation is a
special case of the more general pattern
\begin{align}\label{gfe}
    \sum_{0\le j\le b} \binom{b}{j}
    \tilde{f}^{(j)}(z)
    = \alpha \tilde{f}(pz+\lambda)
    +\beta \tilde{f}(qz+\lambda)+\tilde{g}(z),
\end{align}
where $b=0,1,\dots$, and $\tilde{g}$ itself may involve $\tilde{f}$
but with exponentially small factors. When $b=0$, one has a pure
functional equation,
\begin{align} \label{ab-pgf}
    \tilde{f}(z)
    = \alpha \tilde{f}(pz+\lambda)
    +\beta \tilde{f}(qz+\lambda)+\tilde{g}(z),
\end{align}
while when $b\ge1$, one has a differential-functional equation.

It turns out that the equation \eqref{gfe} covers almost all cases
we collected (a few hundred of publications) in the analysis of BSPs
the majority of which correspond to the case $b=\lambda=0$. The
cases when $b=0$ and $\lambda>0$ are thoroughly treated in
\cite{PF049,PF054,PF055,jacquet90b}, and the cases when
$b\ge1$ are discussed in detail in \cite{hwang10a} (see also the
references cited there). We focus on
$b=\lambda=0$ in this paper. Since the literature abounds with
equation \eqref{an-pgf} or the corresponding recurrence
\eqref{an-rr}, we contend ourselves with listing below some
references that are either standard, representative or more closely
connected to our study here. See also
\cite{drmota11a,erdos87a,fredman74a} for some non-random contexts
where \eqref{an-pgf} appeared.

\begin{description}
\item[\emph{Data Structures.}]
Tries: \cite{knuth98a,mahmoud92a,szpankowski01a};
PATRICIA tries: \cite{jacquet98a,knuth98a,szpankowski90a};
Quadtries and $k$-d tries: \cite{PF058,fuchs11a};
Hashing: \cite{fagin79a,mendelson82a,PF036,PF037};
Suffix trees: \cite{jacquet95a,szpankowski01a}.

\item[\emph{Algorithms.}]
Radix-exchange sort: \cite{knuth98a}; Bucket selection and bucket
sort: \cite{PF159,christophi01a}; Probabilistic counting schemes:
\cite{PF048,PF050,PF174,PF180,PF193};
Polynomial factorization: \cite{PF036}; Exponential variate
generation: \cite{PF060}; Group testing: \cite{goodrich08a}; Random
generation: \cite{PF200,ressler92a}.

\item[\emph{Collision resolution protocols.}]
Tree algorithms in multiaccess channel: \cite{biglieri07a,
massey81a,molle93a,PF049,PF054,PF055,PF071,wagner09a};
Initialization in radio networks: \cite{myoupo03a,shiau05a}; Mutual
exclusion in mobil networks: \cite{mellier04a}; Broadcast
communication model: \cite{yang91a,grabner02a,chen03a}; Leader
election: \cite{fill96a,janson97a,prodinger93a}; Tree algorithms in
RFID systems: \cite{hush98a, namboodiri10a};

\item[\emph{Random models.}]
Random graphs: \cite{banderier13a,gelenbe86a,simon88a};
Geometric IID RVs (or order statistics):
\cite{eisenberg08a,fuchs12b,louchard12a}; Cantor distributions:
\cite{cristea07a,grabner96a}; 
Evolutionary trees: \cite{aldous96a, maddison91a}; Diffusion limited
aggregates: \cite{bradley85a,majumdar03a}; Generalized Eden model on
trees: \cite{dean06a}.

\end{description}

Asymptotics of most of the BSPs can nowadays be handled by standard
\emph{analytic} techniques, which we owe largely to Flajolet for
initiating and laying down the major groundwork. We focus in this
paper on analytic methods. Many elementary and probabilistic methods
have also been proposed in the literature with success;
see, for example, \cite{devroye86a,devroye05a,janson12a,
neininger04a,szpankowski01a} for more information.

\paragraph{Flajolet's works on BSPs.}
We begin with a brief summary of Flajolet's works in the
analysis of BSPs. For more information, see the two chapter
introductions on \emph{Digital Trees} (by Cl\'ement and Ward) and on
\emph{Communication Protocols} (by Jacquet) in \emph{Philippe
Flajolet's Collected Papers, Volume III} (edited by Szpankowski).

Flajolet published his first paper related to BSP in June 1982 in a
paper jointly written with Dominique Sotteau entitled\footnote{Note
that the word ``partitioning'' is spelled as ``partionning'' in the
title of \cite{PF034}, and as ``partitionning'' in the paper.} ``A
recursive partitioning process of computer science'' (see
\cite{PF034}). This first paper is indeed a review paper and starts
with the sentence:
\begin{quote}
    \textsl{We informally review some of the algebraic and analytic
    techniques involved in investigating the properties of a
    combinatorial process that appears in very diverse contexts in
    computer science including digital sorting and searching,
    dynamic hashing methods, communication protocols in local
    networks and some polynomial factorization algorithms.}
\end{quote}
They first brought the attention of the generality of the same
splitting principle in diverse contexts in their Introduction,
followed by a systematic development of generating functions under
different models (\cite[Sec.\ 2: \emph{Algebraic methods}]{PF034}).
Then a general introduction was given of the saddle-point method and
Mellin transform to the Analysis of Algorithms (\cite[Sec.\ 3:
\emph{Analytic methods}]{PF034}). They concluded in the last section
by giving applications of these techniques to one instance in each
of the four areas mentioned above.

Such a synergistic germination of diverse research ideas
\begin{center}
\begin{tikzpicture}[scale=0.85]
\hspace*{-.3cm}
\tikzset{
   s1/.style={
   rectangle,
   minimum size=5mm,
   minimum height=10mm,
   rounded corners=4mm,
   very thick,
   text centered,
   draw=white!10!black!90,
   top color=white,
   bottom color=red!5,
   text width=2cm,
   text centered,
}, s2/.style={
   rectangle,
   minimum size=5mm,
   minimum height=13mm,
   rounded corners=6mm,
   very thick,
   text centered,
   draw=white!10!black!90,
   top color=white,
   bottom color=white!50!blue!10,
   text width=3cm,}
   ,}
\node [s2] (m1) at (0,0) {\small Algorithms \& Applications};%
\node [s1] (m2) at (-3,-2) {\small Algebraic Methods};%
\node [s1] (m3) at (3,-2) {\small Analytic Methods};%
\path[line width=1pt,draw=white!10!black!90,->] (m1) edge (m2);%
\path[line width=1pt,draw=white!10!black!90,->] (m2) edge (m3);%
\path[line width=1pt,draw=white!10!black!90,->] (m3) edge (m1);%
\end{tikzpicture}
\end{center}
later expanded into a wide spectrum of applications and research
networks (see Figure~\ref{fig-pf} for a plot of BSP-related themes).
It was also fully developed and explored, and evolved into his
theory of \emph{Analytic Combinatorics}. Many of these
\emph{objects} become in his hands a \emph{subject} of interest, and
many follow-up papers continued and extended with much ease.

\begin{figure}[!h]
\begin{center}
\begin{tikzpicture}[scale=1]
\path[mindmap,concept color=brown!15,text=black]
node[concept] {\large \bf Binomial Splitting Processes}
    child[concept color=red!15,grow=45,xshift=.2cm] {
        node[concept,minimum size=2.5cm] {\bf Methodology}
        child[grow=135] { node[concept,yshift=0cm,minimum size=2cm]
        {Singularity analysis} }
        child[grow=90] { node[concept,yshift=0cm,minimum size=2cm]
        {Saddle-point method} }
        child[grow=45] { node[concept,yshift=0cm,minimum size=2cm]
        {Complex analysis} }
        child[grow=0] { node[concept,xshift=0cm,minimum size=2cm]
        {Mellin transform} }
        child[grow=-45] { node[concept,minimum size=2cm]
        {Bernoulli sums} }
    }
    child[concept color=blue!20,grow=-45,xshift=0.2cm] {
        node[concept,minimum size=2.5cm] {\bf Applications}
        child[grow=45] { node[concept,yshift=0cm,minimum size=2cm]
        {Digital sorting \& searching} }
        child[grow=0] { node[concept,yshift=0cm,minimum size=2cm]
        {Hashing} }
        child[grow=-45] { node[concept,yshift=0cm,minimum size=2cm]
        {Collision resolution protocols} }
        child[grow=-90] { node[concept,xshift=0cm,minimum size=2cm]
        {Group testing} }
        child[grow=-135] { node[concept,minimum size=2cm]
        {Multi-\\processor systems} }
    }
    child[concept color=red!15,grow=-135,xshift=0cm] {
        node[concept,minimum size=2.5cm] {\bf Theory}
        child[grow=-45] { node[concept,yshift=0cm,minimum size=2cm]
        {Random walks} }
        child[grow=-90] { node[concept,yshift=0cm,minimum size=2cm]
        {Buffon machines} }
        child[grow=-135] { node[concept,yshift=0cm,minimum size=2cm]
        {Dynamical systems} }
        child[grow=180] { node[concept,xshift=0cm,minimum size=2cm]
        {Urn models} }
        child[grow=135] { node[concept,minimum size=2cm]
        {Random graphs} }
        child[grow=0] { node[concept,minimum size=2cm]
        {Random words} }
    }
    child[concept color=blue!20,grow=135,xshift=0cm] {
        node[concept,minimum size=2.5cm] {\bf Algorithms}
        child[grow=225] { node[concept,yshift=-.5cm,minimum size=2cm,]
        {Bit complexity} }
        child[grow=180] { node[concept,yshift=-.5cm,minimum size=2cm]
        {Bucket sort} }
        child[grow=45]
        {node[concept,xshift=.7cm,yshift=-.5cm,minimum size=2.75cm]
        {Probabilistic counting schemes} }
        child[grow=90]
        {node[concept,xshift=0cm,yshift=.3cm,minimum size=2.75cm]
        {Polynomial factorization} }
        child[grow=135] { node[concept,xshift=-.5cm,minimum size=2cm]
        {Random variable generation} }
    } ;
\end{tikzpicture}\vspace*{-.3cm}
\end{center}
\caption{\emph{The diverse themes and methodology developed (or
mentioned) in Flajolet's works that are connected to BSPs.}}
\label{fig-pf}
\end{figure}
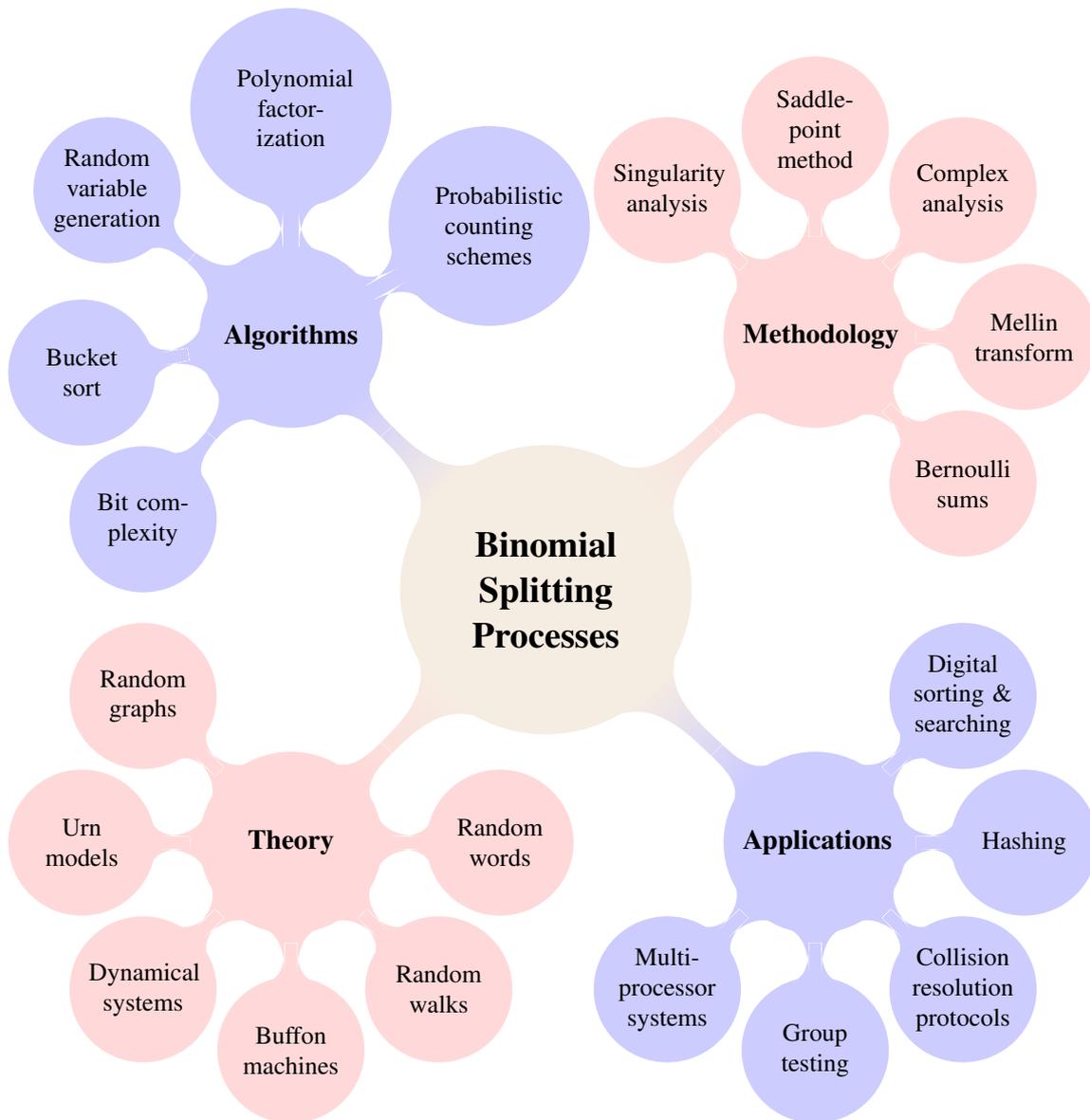

Analysis of algorithms (and particularly BSPs) in the pre-Flajolet
era relied mostly on more elementary approaches (including Tauberian
theorems; see \cite{fredman74a}), with some sporadic exceptions in
the use of the ``Gamma-function method'' (a particular case of
Mellin transform): the height of random trees \cite{de-bruijn72a},
the analysis of radix-exchange sort (essentially the external path
length of random tries) given in \cite[\S 5.2.2]{knuth98a}, PATRICIA
tries and digital search trees \cite[\S 6.3]{knuth98a}, odd-even
merging \cite{sedgewick78a}, register function of random trees
\cite{kemp78a}, analysis of carry propagation \cite{knuth78a}, and
extendible hashing \cite{fagin79a}. See Dumas's chapter introduction
(Chapter 4, Volume III) for a more detailed account.

Many asymptotic patterns such as \eqref{var-Xn}, which most of us take
for granted today, were far from being clear in the 1980's, notably in
engineering contexts. For example, the minute periodic fluctuations
when $\log p/\log q$ is rational are often invisible in numerical
calculations, leading possibly to wrong conclusions. Flajolet pioneered
and developed systematic analytic tools to fully characterize such
tiny perturbations, which he called ``wobbles.'' See \cite{PF120} for
more information.

Amazingly, most of the items in the big picture of
Figure~\ref{fig-pf} were already discovered or clarified in the
1980's in Flajolet's published works with a few later themes aiming
at finer improvements in results or more general stochastic models.
Among these, the ``digital process" and ``probabilistic counting
algorithms'' became two of his favorite subjects of presentation, as
can be seen from his webpage of lectures where about one third of
talks are related to these two subjects.

This paper is organized as follows. We briefly introduce tries,
functional equations and the analytic tools in the next section.
We then develop more analytic tools we need in
Section~\ref{sec-hada}, the most difficult part being the proof
of admissibility under Hadamard product. Then we focus on the
characterization of the asymptotic variance of general trie
statistics in the following sections. We also include PATRICIA tries
in Section~\ref{sec-patricia} and conclude this paper with a few
remarks.

\section{Random tries, functional equations and
asymptotic analysis}

\label{sec-prelim}

The design of an ordinary dictionary according to the alphabetical
(or lexicographical) order induces itself a tree structure, which is
also the splitting procedure used in many digital tree structures
and bucketing algorithms such as tries and radix sort. Tries (coined
by Fredkin \cite{fredkin60a}, which is a mixture of ``tree" and
``retrieval") were first introduced in computer algorithms by de la
Briandais \cite{de-la-briandais59a} in 1959, the same year when the
radix-exchange sort (a digital realization of Quicksort) was
proposed by Hildebrandt and Isbitz \cite{hildebrandt59a}; see
\cite[\S 6.3]{knuth98a} for more information. Tries are one of
the most widely adopted prototype data structures for words and
strings, and admit a large number of extensions and variants.

Given a set of $n$ random binary strings (each being a sequence of
Bernoulli random variables with parameter $p$), we can recursively
define the random trie associated with this set as follows. If
$n=0$, then the trie is empty; if $n=1$, then the trie is composed
of a single (external) node holding the input-string; if $n>1$, then
the trie contains three parts: a root (internal) node used to direct
keys to the left (when the first bit of the string is $0$) or to
the right (when the first bit of the string is $1$), a left sub-trie
of the root for keys whose first bits are $0$ and a right sub-trie
for keys whose first bits are $1$; strings directing to each of the
two subtrees are constructed recursively as tries (but using
subsequent bits successively). Thus tries are ordered, prefix trees.
See Figure~\ref{fg-trie} for a trie of $7$ keys.

\begin{figure}
\begin{center}
\begin{tikzpicture}[line width=1.5pt]
\tikzstyle{every node}=[draw ,circle, level distance=0.2cm]
\tikzstyle{level 1}=[sibling distance=8cm]
\tikzstyle{level 2}=[sibling distance=3cm]
\node {}
    child{ node{}
        child{node[rectangle]{$00$}
        edge from parent node[draw=none, left] {$0$}}
        child{node{}
            child{node{}
                child{node[rectangle]{$0100$}
                edge from parent node[draw=none, left] {$0$}}
                child{node{}
                    child{node[rectangle]{$01010$}
                    edge from parent node[draw=none, left] {$0$}}
                    child{node[rectangle]{$01011$}
                    edge from parent node[draw=none, right] {$1$}}
                edge from parent node[draw=none, right] {$1$}}
            edge from parent node[draw=none, left] {$0$}}
            child{node[rectangle]{$011$}
            edge from parent node[draw=none, right] {$1$}}
        edge from parent node[draw=none, right] {$1$}}
    edge from parent node[draw=none, left] {$0$}}
    child{ node{}
        child{node{}
            child{
            edge from parent[draw=none]}
            child{node{}
                child{node[rectangle]{$1010$}
                edge from parent node[draw=none, left] {$0$}}
                child{node[rectangle]{$1011$}
                edge from parent node[draw=none, right] {$1$}}
            edge from parent node[draw=none, right] {$1$}}
        edge from parent node[draw=none, left] {$0$}}
        child{
        edge from parent[draw=none]}
    edge from parent node[draw=none, right] {$1$}};
\end{tikzpicture}
\caption{\emph{A trie of $n=7$ records: the circles
represent internal nodes and rectangles holding the records are
external nodes.}} \label{fg-trie}
\end{center}
\end{figure}

Asymptotic analysis of the trie recurrence \eqref{an-rr} is nowadays
not difficult and a typical way of deriving asymptotic estimates
starts with the Poisson generating function \eqref{pgf}, which satisfies
the functional equation \eqref{an-pgf} (when $a_n$ does not grow faster
than, say exponential), where $\tilde{g}(z)$ there depends on $b_n$ and
the initial conditions. From this, one sees that the Mellin transform of
$\tilde{f}(z)$
\[
    \mathscr{M}[\tilde{f};s]
    :=\int_0^\infty \tilde{f}(z)z^{s-1}\dd z
\]
satisfies \emph{formally}
\[
    \mathscr{M}[\tilde{f};s]
    = \frac{\mathscr{M}[\tilde{g};s]}{1-p^{-s}-q^{-s}}.
\]

Then the asymptotics of $a_n$ can be manipulated by a two-stage
analytic approach: first derive asymptotics of $\tilde{f}(z)$ for
large $|z|$ by the inverse Mellin integral
\begin{align}\label{fz-inv-mellin}
    \tilde{f}(z) = \frac1{2\pi i}\int_\uparrow
    \frac{\mathscr{M}[\tilde{g};s]\,z^{-s}}{1-p^{-s}-q^{-s}}\,\dd s,
\end{align}
where the integration path is some vertical line, and then apply the
saddle-point method to Cauchy's integral formula (called
\emph{analytic de-Poissonization} \cite{jacquet98a})
\begin{align}\label{an-cauchy}
    a_n = \frac{n!}{2\pi i}\oint_{|z|=r} z^{-n-1}
    e^{z}\tilde{f}(z) \dd{z}\qquad(r>0).
\end{align}
This two-stage Mellin-saddle approach has been largely developed by
Jacquet and Szpankowski (see \cite{jacquet98a}) and can in many real
applications be encapsulated into one, called the
Poisson-Mellin-Newton cycle in Flajolet's papers (see
\cite{PF052,PF124})
\begin{align}\label{an-pmn}
    a_n = \frac{n!}{2\pi i}\int_{\uparrow}
    \frac{\mathscr{M}[\tilde{g};s]}
    {(1-p^{-s}-q^{-s})\Gamma(n+1-s)}\,\dd s,
\end{align}
which is \emph{formally} obtained by substituting
\eqref{fz-inv-mellin} into \eqref{an-cauchy} and by interchanging
the order of integration.

Note that such a formal representation may be meaningless due to the
divergence of the integral. One of the most useful tools in
justifying the \emph{exponential smallness} of $\mathscr{M}
[\tilde{g};s]$ at $c\pm \infty$ is Proposition 5 of Flajolet et
al.'s survey paper \cite{PF120} on Mellin transforms. For ease of
reference, we call it the \emph{Exponential Smallness Lemma} in this
paper.
\begin{quote}
    \textbf{Exponential Smallness Lemma.} \cite[Prop.\ 5]{PF120}
    \emph{If, inside the sector $|\arg(z)|\le \theta$ ($\theta>0$),
    $f(z) = O(|z|^{-\alpha})$, as $z\to0$, and $f(z) = O(|z|^{-\beta})$
    as $|z|\to\infty$, then $\mathscr{M}[f;s]= O(e^{-\theta |\Im(s)|})$
    holds uniformly for $\Re(s)\in \langle\alpha,\beta\rangle$.}
\end{quote}
This simple Lemma is crucial in the development of our approach.

In various practical cases, the use of the Poisson-Mellin-Newton approach
relies mostly on the so-called Rice's integral formula (or integral
representation for finite differences) when the integral converges; see
Figure~\ref{fig-pmr} for a diagrammatic illustration.

\tikzstyle{block} = [rectangle, draw, fill=none, line width = 0.8pt,
text width=5em, rounded corners, minimum height=4em]
\tikzstyle{line} = [draw, -latex',line width = 1pt]
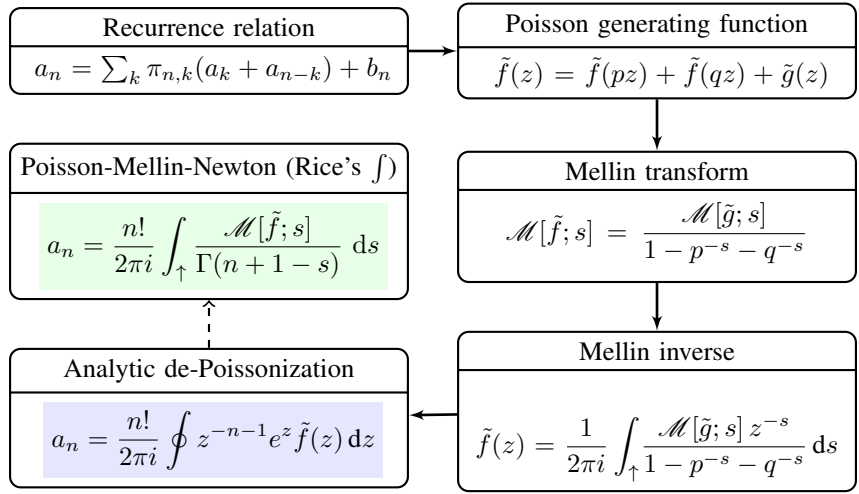
\begin{figure}[!h]
\begin{center}{\footnotesize
\begin{tikzpicture}[auto]
\node [block, text width=5cm, rectangle split, rectangle split
parts=2, text centered] (b1) { Recurrence relation

\nodepart{second} { $ a_n = \sum\nolimits_k \pi_{n,k}
(a_k+a_{n-k})+b_n$ } };

\node [block, text width=5cm, right=4ex of b1, node distance=2.5cm,
rectangle split, rectangle split parts=2, text centered] (b2)
{
Poisson generating function
\nodepart{second}
{
{$\tilde{f}(z)=\tilde{f}(pz)+\tilde{f}(qz)+\tilde{g}(z)$}
}
};

\node [block, text width=5cm, below=4ex of b1, node distance=2.5cm,
rectangle split, rectangle split parts=2, text centered] (b3)
{
Poisson-Mellin-Newton (Rice's $\int$)
\nodepart{second}
{
\colorbox{green!10}{$\displaystyle
a_n = \frac{n!}{2\pi i}\int_{\uparrow}
\frac{\mathscr{M}[\tilde{f};s]}
{\Gamma(n+1-s)}\,\dd s$}
}
};

\node [block, text width=5cm, below=4.3ex of b2, node distance =3cm,
rectangle split, rectangle split parts=2, text centered] (b4)
{
Mellin transform
\nodepart{second}
{ \quad \\
$\displaystyle \mathscr{M}[\tilde{f};s]
=\frac{\mathscr{M}[\tilde{g};s]}{1-p^{-s}-q^{-s}}$ \\ \quad
}
};

\node [block, text width=5cm, below =4ex of b3, node distance = 2.5cm,
rectangle split, rectangle split parts=2, text centered] (b5)
{
Analytic de-Poissonization
\nodepart{second}
{
\colorbox{blue!10}{$\displaystyle
    a_{n} =\frac{n!}{2\pi i}\oint z^{-n-1}e^{z}
    \tilde{f}(z)\,\text{d} z$}
}
};

\node [block, text width=5cm, below =4 ex of b4, node distance = 3cm,
rectangle split, rectangle split parts=2, text centered] (b6)
{
Mellin inverse
\nodepart{second}
{\quad \\ \quad \\
$\displaystyle\tilde{f}(z)=\frac{1}{2\pi i}
    \int_{\uparrow}\!
    \frac{\mathscr{M}[\tilde{g};s]\,z^{-s}}
    {1-p^{-s}-q^{-s}}\,\text{d} s$
}
};
\draw [line] (b1) -- (b2);
\draw [line] (b2) -- (b4);
\draw [line] (b4) -- (b6);
\draw [line] (b6) -- (b5);
\draw [thick,dashed,->] (b5) -- (b3);
\end{tikzpicture}}
\end{center}\label{fig-diag}\vspace*{-.4cm}
\caption{\emph{The two analytic approaches to the asymptotics of
$a_n$. Here $\pi_{n,k} := \binom{n}{k}p^k q^{n-k}$.}}
\label{fig-pmr}
\end{figure}

Asymptotics of either of the two integrals \eqref{fz-inv-mellin} and
\eqref{an-pmn} rely heavily on the singularities of the integrand,
which in turn depends on the location of the zeros of the equation
$1-p^{-s}-q^{-s}=0$. A detailed study of the zeros can be found in
\cite{PF055}, and later in \cite{drmota11a,schachinger00a}. While
the dominant asymptotic terms are often easy to characterize when
analytic properties of $\mathscr{M}[\tilde{g};s]$ are known (owing
largely to the systematic tools Flajolet and his coauthors
developed), error analysis turned out to be highly challenging when
$\log p/\log q$ is irrational; see \cite{PF208}.

These analytic tools are well-suited for computing the asymptotics
of the mean, but soon become very messy when adopted for higher moments,
which satisfy the same type of recurrences but with convolution terms
that are often difficult to manipulate analytically. The situation
becomes even worse when dealing with the variance or higher central
moments because the high concentration of binomial distribution results
in smaller variance, meaning more complicated cancelations in the
desired asymptotic approximations have to be properly taken into account.

The key, crucial step of our approach to the asymptotic variance of
trie statistics is to introduce, as in \cite{hwang10a}, the corrected
Poissonized variance of the form
\begin{align}\label{Poi-Var}
    \tilde{V}(z) := \tilde{f}_2(z) - \tilde{f}_1(z)^2
    - z\tilde{f}_1'(z)^2,
\end{align}
where $\tilde{f}_1$ and $\tilde{f}_2$ denote the Poisson generating
functions of the first and the second moments, respectively.
Such a consideration results in simpler Fourier series expansions
for the periodic functions that appear in the asymptotic
approximations due to the no-cancelation character, especially in
the symmetric case. We will enhance this approach by introducing the
class of \emph{JS-admissible functions} as in our previous paper
\cite{hwang10a},
a notion formulated from Jacquet and Szpankowski's works on analytic
de-Poissonization (see \cite{jacquet98a}) and mostly inspired from
Hayman's classical work \cite{hayman56a} on saddle-point method
(see also \cite[\S VIII.5]{PF201}), via which many
asymptotic approximations can be derived by checking only simple
criteria of admissibility. The combined use leads to a very
effective, systematic approach that can be easily adapted for
diverse contexts where a similar type of analytic problems is
encountered; see Sections~\ref{sec-apps} and \ref{sec-more} for some
examples.

In general, polynomial growth rate for $\tilde{g}(z)$ for large
$|z|$ implies the same for $\tilde{f}$ in a small sector containing
the real axis. The only exception is the functional-differential
equation \cite{banderier13a}
\[
    \tilde{f}'(z) = \tilde{f}(qz) + \tilde{g}(z),
\]
for which the growth is of order $z^{c\log z}$ when
$\tilde{g}$ grows polynomially for large $|z|$. Note that this
equation is a special case of the so-called ``pantograph equations";
see \cite{banderier13a} for more information.

\medskip

\noindent \textbf{Notations.} Throughout this paper, $q=1-p$ and
$0<p<1$. Also $h:=-p\log p-q\log q$ denotes the entropy of the
Bernoulli distribution. The splitting distribution $I_n$ is a
binomial distribution with mean $pn$. For brevity, we introduce the
generic symbol $\mathscr{F}[G](x)$ to denote a bounded periodic
function of period $1$ of the form
\begin{align} \label{FGx}
    \mathscr{F}[G](x) = \left\{\begin{array}{ll}
    {\displaystyle h^{-1}
    \sum_{k\in\mathbb{Z}\setminus\{0\}}
    G(-1+\chi_k)e^{2k\pi i x}},
    &\text{if}\ \frac{\log p}{\log q}\in\mathbb{Q}\\
    0, &\text{if}\ \frac{\log p}{\log q}\not\in\mathbb{Q}
    \end{array}\right\},
\end{align}
where $\chi_k = \frac{2rk\pi i}{\log p}$ when $\frac{\log p} {\log
q} = \frac{r}{\ell}$ with $(r,\ell)=1$. The average value of
$\mathscr{F}[G]$ is zero and the Fourier series is always absolutely
convergent (it is indeed infinitely differentiable for all cases we
study).

\section{JS-admissibility, Hadamard product and asymptotic
transfer}
\label{sec-hada}

We collect and develop in this section some technical preliminaries,
which are needed later for our asymptotic analysis.

\subsection{JS-admissible functions}
We begin with recalling the definition and a few fundamental
properties from \cite{hwang10a} of JS-admissibility (a framework
combining ideas from \cite{hayman56a,jacquet98a}).
\begin{Def} \label{def-js}
Let $\tilde{f}(z)$ be an entire function. Then we say that
$\tilde{f}(z)$ is JS-admissible and write $\tilde{f}\in\JS$ (or more
precisely, $\tilde{f}\in\JS_{\!\!\!\alpha,\beta},
\alpha,\beta\in\mathbb{R})$ if for some $0<\theta<\pi/2$ and
$\vert z\vert\ge 1$ the following two conditions hold.
\begin{itemize}
\item[{\bf (I)}] (Polynomial growth inside a sector) Uniformly for
$\vert\arg(z)\vert\le\theta$,
\[
    \tilde{f}(z)=O\left(\vert z\vert^{\alpha}(\log_{+}\vert
    z\vert)^{\beta}\right),
\]
where $\log_{+}x:=\log(1+x)$.

\item[{\bf (O)}] (Exponential bound) Uniformly for
$\theta\le\vert\arg(z)\vert \le\pi$,
\[
    f(z):=e^{z}\tilde{f}(z)=O\left(e^{(1-\ve)\vert
    z\vert}\right),
\]
for some $\ve>0$.
\end{itemize}
\end{Def}

The major reason of introducing JS-admissible functions is to
provide a systematic analytic justification of the \emph{Poisson
heuristic} $a_n \sim \tilde{f}(n)$, where $\tilde{f}$ is the Poisson
generating function of $a_n$. We do not however pursue optimum
conditions here for simplify and easy applications. On the other
hand, since the conditions of admissibility we impose are strong, we
can indeed provide a very precise asymptotic characterization of
$a_n$.

\begin{prop}[\cite{hwang10a}] \label{prop-PC}
If $\tilde{f}\in \JS_{\!\!\!\alpha,\beta}$, then $a_n $ satisfies
the asymptotic expansion
\begin{align} \label{PC-exp}
    a_n = \sum_{0\le j< 2k}
    \frac{\tilde{f}^{(j)}(n)}{j!}\,\tau_j(n)
    + O\left(n^{\alpha-k}\log^\beta n \right),
\end{align}
for $k=0,1,\dots$, where the $\tau_j$'s are polynomials of $n$ of
degree $\tr{j/2}$ given by
\[
    \tau_j(n) = \sum_{0\le l\le j}
    \binom{j}{l}(-n)^l \frac{n!}{(n-j+l)!}
    \qquad(j=0,1,\dots).
\]
\end{prop}
Note that the Poisson-Charlier expansion
\[
    a_n = \sum_{j\ge0} \frac{\tilde{f}^{(j)}(n)}{j!}\,
    \tau_j(n)
\]
converges as long as $\tilde{f}$ is an entire function; see
\cite{hwang10a}.
It is the asymptotic nature \eqref{PC-exp} that requires more
regularity conditions, which can intuitively be seen by observing
that $\tilde{f}^{j}(n) \asymp n^{-j}\tilde{f}(n)$ when $\tilde{f}
\in \JS$; see \cite{hwang10a}.

The $\tau_j$'s are closely connected to Charlier and Laguerre
polynomials; see \cite{hwang10a} for a more detailed discussion.
The first few terms are given as follows.
\begin{center}
\begin{tabular}{|c|c|c|c|c|c|c|}\hline
$\tau_0(n)$ & $\tau_1(n)$ & $\tau_2(n)$ & $\tau_3(n)$ & $\tau_4(n)$
&$\tau_5(n)$ & $\tau_6(n)$ \\ \hline $1$ &  $0$ & $-n$  & $2n$ &
$3n(n-2)$ & $-4n(5n-6)$ & $-5n(3n^2-26n+24)$ \\ \hline
\end{tabular}
\end{center}
The fact that $\tau_1=0$ indicates that much information is
condensed in the dominant term $\tilde{f}(n)$.

At the generating function level, the usefulness of JS-admissible
functions lies in the closure properties under several elementary
operations.
\begin{prop}[\cite{hwang10a}]\label{prop-closure}
Let $m$ be a non-negative integer and $\alpha\in(0,1)$.
\begin{itemize}

\item[(i)] $z^m, e^{-\alpha z}\in\JS$.

\item[(ii)] If $\tilde{f}\in\JS$, then $\tilde{P}\tilde{f}\in\JS$
for any polynomial $\tilde{P}(z)$.

\item[(iii)] If $\tilde{f}\in\JS$, then $\tilde{f}(\alpha z)\in\JS$.

\item[(iv)] If $\tilde{f},\tilde{g}\in\JS$, then $\tilde{f}
+\tilde{g}\in\JS$.

\item[(v)] If $\tilde{f},\tilde{g}\in\JS$, then $\tilde{f}(\alpha
z)\tilde{g}((1-\alpha)z)\in\JS$.

\item[(vi)] If $\tilde{f}\in\JS$, then $\tilde{f}^{(m)}\in\JS$.
\end{itemize}
\end{prop}

We will enhance these closure properties by proving that
JS-admissibility is also closed under Hadamard product.

\subsection{Asymptotic transfer}

For our purposes, we need also a transfer theorem for entire
functions satisfying the functional equation \eqref{an-pgf}.

\begin{prop}\label{prop-at}
Let $\tilde{f}(z)$ and $\tilde{g}(z)$ be entire functions satisfying
\[
    \tilde{f}(z)=\tilde{f}(pz)+\tilde{f}(qz)+\tilde{g}(z),
\]
with $f(0)$ given. Then
\[
    \tilde{f}\in\JS\quad\text{if and only
    if}\quad\tilde{g}\in\JS.
\]
\end{prop}
\pf The proof is similar to and simpler than that of Proposition 2.4
in \cite{hwang10a}. Thus we only give the proof for ({\bf I}). Define
\[
    \tilde{B}(r):=\max_{\substack{\vert z\vert\le
    r\\ \vert\arg(z)\vert\le\ve}}\vert\tilde{f}(z)\vert.
\]
Then
\[
    \tilde{B}(r)\le\tilde{B}(pr)+\tilde{B}(qr)+
    O\left(r^{\alpha}(\log_{+}r)^{\beta}+1\right).
\]
Now define a majorant function $\tilde{K}(r)$ by
\[
    \tilde{K}(r)=\tilde{K}(pr)+\tilde{K}(qr)+
    C\left(r^{\alpha}(\log_{+}r)^{\beta}+1\right),
\]
where $C>0$. Then $\tilde{B}(r)\le\tilde{K}(r)$ for a sufficiently
large $C>0$, and by standard Mellin argument \cite{PF120} or by the
proof used in \cite{schachinger95b}
\[
    \tilde{K}(r)=\left\{\begin{array}{ll} O(r),&\text{if}\
    \alpha<1;\\ O(r^{\alpha}(\log_{+}r)^{\beta+1}),
    &\text{if}\ \alpha=1;\\
    O(r^{\alpha}(\log_{+}r)^{\beta}),
    &\text{if}\ \alpha>1.\end{array}\right\}
\]
This completes the proof.\qed

We now refine the asymptotic transfer and focus on asymptotically
linear functions.

\begin{prop}\label{prop-at2} Let $\tilde{f}$ and $\tilde{g}$ be
entire functions related to each other by the functional equation
$\tilde{f}(z)=\tilde{f}(pz)+\tilde{f}(qz)+\tilde{g}(z)$ with $f(0)$
given. Assume $0<\theta<\pi/2,\alpha<1$ and $\beta\in\mathbb{R}$.
\begin{itemize}
\item[(a)] If $\tilde{g}(z)=O(\vert z\vert^{\alpha}(\log_{+}\vert
z\vert)^{\beta})$, where the $O$-term holds uniformly for $|z|\ge1$
and $\vert\arg(z)\vert\le\theta$, then, as $|z|\rightarrow\infty$ in
the same sector,
\[
    \frac{\tilde{f}(z)}{z}= \frac{G(-1)}{h}
    +\mathscr{F}[G](r\log_{1/p}z)
    +o(1),
\]
where the notations $\mathscr{F}[G](x)$ and $r$ are defined in
\eqref{FGx}.

\item[(b)] If $\tilde{g}(z)=cz+O(\vert z\vert^{\alpha}(\log_{+}\vert
z\vert)^{\beta})$ uniformly for $|z|\ge1$ and $\vert\arg(z)\vert
\le\theta$, then, as $|z|\rightarrow\infty$ in the same sector,
\[
    \frac{\tilde{f}(z)}{z} =
    \frac{c}{h}\log z+h_0
    +\mathscr{F}[G](r\log_{1/p}z)+o(1),
\]
where 
\begin{align}\label{h0}
	h_0 := \frac{c_0}{h} +\frac{c(p\log^2p+q\log^2 q)}{2h^2},
\end{align}	
$G(s)$ is the meromorphic continuation of $\mathscr{M}[\tilde{g};s]$, 
and
\[
    c_0 : =\lim_{s\rightarrow-1}
    \left(G(s)+\frac{c}{s+1}\right) .
\]
\end{itemize}
\end{prop}
\pf Without loss of generality, we may assume that $\tilde{f}(0)
=\tilde{f}'(0)=\tilde{g}(0)=\tilde{g}'(0)=0$. Then both Mellin
transforms exist in the strip $\langle -2,-1\rangle$ and
\[
    \mathscr{M}[\tilde{f};s]
    =\frac{G(s)}
    {1-p^{-s}-q^{-s}}.
\]
Note that $G(s)$ can be extended to a meromorphic function in the
strip $\langle-2,-\alpha-\ve\rangle$. In the case of (\emph{a}),
$G(s)$ is analytic on the line $\Re(s)=-1$ while in the case of
(\emph{b}) $G(s)$ has a unique simple pole on $\Re(s)=-1$ at $s=-1$
with the local expansion $G(s)=-c/(s+1)+ c_0+\cdots$. Note that by
applying the Exponential Smallness Lemma (\cite[Prop. 5]{PF120}), we
have the estimate
\[
    |G(\sigma+it)| =O\left(e^{-\theta|t|}\right),
\]
uniformly for large $|t|$ and $\sigma \in \langle-2, -\alpha-\ve
\rangle$. Thus the Proposition follows from standard Mellin analysis
(see \cite{PF120}) and known properties of the zeros of
$1-p^{-s}-q^{-s}$ (see \cite{PF055}). \qed

In the symmetric case when $p=q=1/2$, both error terms $o(1)$ in the
Proposition can be improved to $O(\max\{1,|z|^{\alpha-1} (\log
|z|)^\beta)\}$. Indeed, all error terms in such a case in this paper
can be improved by standard argument; we focus instead on the
Fourier series expansion in this paper.

\subsection{A Hadamard product for Poisson generating functions}

We need a new closure property for the analysis of the variance.
Given two Poisson generating functions
\[
    \tilde{f}(z)=e^{-z}\sum_{n\ge 0}\frac{a_n}{n!}\,z^n
    \qquad\text{and}\qquad
    \tilde{g}(z)=e^{-z}\sum_{n\ge 0}\frac{b_n}{n!}\,z^n,
\]
we define the \emph{Hadamard product} of these two functions as
\[
    \tilde{h}(z):= \tilde{f}(z) \odot \tilde{g}(z)=
    e^{-z}\sum_{n\ge 0}\frac{a_nb_n}{n!}\,z^n.
\]
This definition differs from but coincides with the usual one
if we consider the corresponding exponential generating functions.
We show that JS-admissibility is closed under the Hadamard product.
The proof is subtle and delicate.

\begin{prop}\label{prop-hada}
If $\tilde{f}\in\JS_{\!\!\!\alpha_1,\beta_1}$ and
$\tilde{g}\in\JS_{\!\!\!\alpha_2,\beta_1}$, then
$\tilde{h}\in\JS_{\!\!\!\alpha_1+\alpha_2,\beta_1+\beta_2}$.
More precisely, we have
\begin{align}\label{hfg-hada}
    \tilde{h}(z)=\tilde{f}(z)\tilde{g}(z)
    +z\tilde{f}'(z)\tilde{g}'(z)+
    O\left(\vert z\vert^{\alpha_1+\alpha_2-2}(\log_{+}\vert
    z\vert)^{\beta_1+\beta_2}\right),
\end{align}
uniformly as $|z|\to\infty$ and $\vert \arg(z)\vert\le\theta$, where
$0<\theta<\pi/2$.
\end{prop}
Our proof indeed gives an asymptotic expansion for $\tilde{h}$;
we content ourselves with the statement of \eqref{hfg-hada}, which is
sufficient for our purposes.

\begin{proof}
Let $f(z):=e^{z}\tilde{f}(z), g(z):=e^{z}\tilde{g}(z)$ and
$h(z):=e^{z}\tilde{h}(z)$. Also let $0<\theta_0<\pi/2$ be an angle
where \textbf{(I)} holds for both $\tilde{f}(z)$ and $\tilde{g}(z)$.
Note that conditions \textbf{(I)} and \textbf{(O)} remain true
if $\theta_0$ is replaced by an arbitrarily small but
fixed angle $0<\theta\le \theta_0$ with a suitable choice of $\ve
=\ve(\theta)$.

We prove the proposition in the special case when $\beta_1
=\beta_2=0$ since the proof in the general case remains the same
with only additional logarithmic terms in the corresponding error
estimates.

Define
\[
    J(z):=\sum_{n\ge0}\frac{a_nb_n}{(n!)^2}\,z^{2n}
    =\frac{1}{2\pi}\int_{-\pi}^\pi f(ze^{it})g(ze^{-it})\dd t.
\]
Substituting here $z\to zu$, multiplying both sides by $ue^{-u^2}$,
integrating from $0$ to infinity and multiplying the result by
$e^{-z^2}/2$, we obtain
\[
    \tilde{h}(z^2)=e^{-z^2}\sum_{n\ge0}\frac{a_nb_n}{n!}\,z^{2n}
    =\frac{e^{-z^2}}{2}\int_0^\infty ue^{-u^2}J(zu)\dd u.
\]
We now fix a $0<\theta<\theta_0$. We first show that $h(z)$
satisfies condition \textbf{(O)} for $z$ lying outside the sector
$|\arg(z)|\le \theta$. Assume $\theta/2 \le |y|\le \pi/2$. Then
\[
    J(re^{iy})=\frac{1}{2\pi}\int_{-\pi}^\pi
    f(re^{i(t+y)})g(re^{i(y-t)})\dd t.
\]
Note that \textbf{(I)} and \textbf{(O)} imply that
\begin{align}\label{fg-O}
    f(z)=O\left((|z|^{\alpha_1}+1) e^{|z|}\right),
    \quad\text{and}\quad
    g(z)=O\left((|z|^{\alpha_2}+1) e^{|z|}\right),
\end{align}
uniformly for $z\in \mathbb{C}$. Now making the change of variables
$t\mapsto t-y$ and taking into account that the function under the
integral sign is periodic, we see that
\[
\begin{split}
    J(re^{iy})&=\frac{1}{2\pi}\int_{-\pi}^\pi
    f(re^{it})g(re^{i(2y-t)})\dd t \\
    &=\frac{1}{2\pi}\int_{-y}^{y}
    f(re^{it})g(re^{i(2y-t)})\dd t +O\left(
    (r^{\alpha_2}+1) e^{(2-\ve)r}\right).
\end{split}
\]
Here we evaluated the integral over the region $|y|\le |t|\le \pi$
by the estimate $|f(re^{it})|=O\bigl( e^{(1-\ve)r}\bigr)$, which
follows from \textbf{(O)}, and used the upper bound of (\ref{fg-O})
for $|g(re^{i(2y-t)})|$.
In a similar way, we note that $0<\theta/2\le |2y-t|\le 3\pi/2$
whenever $|t|\le |y|$ and $\theta\le|y|\le\pi/2$. This means that
$z=re^{i(2y-t)}$ lies inside the sector $|\arg(z)|\ge \theta/2$ and
as a consequence we can use estimates $|g(re^{i(2y-t)})|
=O\bigl(e^{(1-\ve)r}\bigr)$ and $|f(re^{it})|=O\bigl(
(r^{\alpha_1}+1) e^{r}\bigr)$ to evaluate the integral over the
range $|t|\le |y|$. Combining these estimates, we get
\[
\begin{split}
    J(re^{iy})&= O\left((r^{\alpha_1}+r^{\alpha_2}+1)
    e^{(2-\ve)r}\right)\\
    &=O\left(e^{2\ve_1 r}\right),
\end{split}
\]
where $\ve_1$ is chosen such that $(2-\ve)/2<\ve_1<1$. This estimate
yields
\[
    e^{r^2e^{2iy}}\tilde{h}(r^2e^{2iy})
    =O\left(\int_0^\infty ue^{2\ve_1 ru} e^{-u^2}\dd u\right)
    =O\left(r e^{\ve_1^2 r^2}\right),
\]
which implies, by replacing $r\to \sqrt{r}$ and $y\to y/2$, the
estimate
\[
    e^{re^{iy}}\tilde{h}(re^{iy})
    =O\left(r e^{\ve_1^2 r}\right),
\]
in the region $\theta\le |y|\le \pi$ for any fixed $\theta>0$.
Thus condition \textbf{(O)} holds.

We now prove that $\tilde{h}(z)$ grows polynomially in the sector
$|\arg z|\le \theta$ with some sufficiently small $\theta>0$.

Note that $|\arg(ze^{\pm it})|\ge \theta_0/4$ for all values of $t$
and $z$ such that $|\arg z|\le \theta_0/4$ and $\pi\ge|t|\ge
\theta_0/2$, which, by \textbf{(I)} and \textbf{(O)}, implies that
$f(ze^{it})=O(e^{\ve_2|z|})$ and $g(ze^{it})=O(e^{\ve_2|z|})$ with a
suitable choice of $\varepsilon_2<1$ for all $z$ and $t$ satisfying
such restrictions. It follows that
\[
\begin{split}
    J(z)=\frac{1}{2\pi}\int_{-\theta_0/2}^{\theta_0/2}
    e^{2z\cos t}\tilde{f}(ze^{it})\tilde{g}(ze^{-it})
    \dd t +O\left(e^{2\ve_2|z|}\right),
\end{split}
\]
when $|\arg z|\le \theta_0/4$. Thus
\begin{equation}\label{estim_h_z_2}
\begin{split}
    \tilde{h}(z^2) =\frac{e^{-z^2}}{4\pi}
    \int_0^\infty ue^{-u^2}\!\!\int_{-\theta_0/2}^{\theta_0/2}
    e^{2zu\cos t}\tilde{f}(zue^{it})
    \tilde{g}(zue^{-it})\,\dd t \dd u
    +O\left(|z|e^{\ve_2^2|z|^2-\Re(z^2)}\right).
\end{split}
\end{equation}
Noting that $\Re (z^2) =\cos(2\arg(z))|z|^2\ge
(1-\arg(z)^2/2) |z|^2$, we then have
\begin{equation}\label{hz2}
    \tilde{h}(z^2)=I(z)+O\left(|z|
    e^{-\left(1-\ve_2^2-\arg(z)^2/2)\right)|z|^2}\right),
\end{equation}
where $I(z)$ denotes the double integral
\[
    I(z):=\frac{1}{4\pi}\int_0^\infty u\int_{|t|\le\theta_0/2}
    e^{-(u-z\cos t)^2-z^2\sin^2 t}\tilde{f}(zue^{it})
    \tilde{g}(zue^{-it})\dd t \dd u.
\]
Since $|\arg(z)|\le\theta_0/4$, the arguments
$zue^{it}$ and $zue^{-it}$ of the functions $\tilde{f}$ and
$\tilde{g}$ lie inside the sector $|\arg(z)|\le\theta_0$, which
means that $\tilde{f}(zue^{it}) =O(|z|^{\alpha_1})$ and $\tilde{g}
(zue^{-it}) =O(|z|^{\alpha_2})$.

Changing the order of integration and making a change of integration
path from the interval $u\in(0,\infty)$ to the line $(0,z\infty)$ by
mapping $u\mapsto uz\cos t$, we get
\[
\begin{split}
    I(z)&=\frac{z^2}{4\pi}\int_{-\theta_0/2}^{\theta_0/2} \!
    \int_0^\infty e^{-z^2(u-1)^2\cos^2 t-z^2\sin^2 t}
    \tilde{f}(z^2ue^{it}\cos t)
    \tilde{g}(z^2ue^{-it}\cos t)u\dd u
    \cos^2 t \dd t\\
    &=O\left({|z|^2}\int_{-\theta_0/2}^{\theta_0/2} \!
    \int_0^\infty \!\!
    e^{-\Re(z^2)(u-1)^2\cos^2 (\theta_0/2)-\Re(z^2)\sin^2 t}
    (|z^2|u+1)^{\alpha_1}(|z^2|u+1)^{\alpha_2}
    u\dd u \dd t\right)\\
    &=O\left(|z|^{2} \left(\int_0^\infty
    e^{-\Re(z^2)(u-1)^2\cos^2 (\theta_0/2)}
    (|z|^2u+1)^{\alpha_1+\alpha_2}u\dd u\right)
    \int_{|t|\le \theta_0/2} e^{-\Re(z^2)\sin^2 t}\dd t\right)\\
    &=O\left(\frac{(|z|+1)^{2+2\alpha_1+2\alpha_2}}
    {\Re(z^2)}\right)\\
    &= O\left(|z|^{2\alpha_1+2\alpha_2}\right),
\end{split}
\]
uniformly for large $|z|$ in the sector $|\arg(z)|\le\theta_0/4$.

Applying this estimate to the expression (\ref{estim_h_z_2}) of
$\tilde{h}(z^2)$, we obtain
\[
    \tilde{h}(z^2)=O\left(|z|^{2\alpha_1+2\alpha_2}\right).
\]
Thus
\[
    \tilde{h}(z)=O\left(|z|^{\alpha_1+\alpha_2}\right).
\]
for all $|\arg z|\le \theta \le\theta_0/2$, where $\theta$ is chosen
to be small enough to ensure that the error term in (\ref{hz2})
decreases exponentially fast. This proves the first part of the
proposition.

To prove the finer estimate \eqref{hfg-hada}, we use the Taylor
expansion
\[
   \tilde{f}(z)=\sum_{0\le j<N}
   \frac{\tilde{f}^{(j)}(w)}{j!}\,(w-z)^j
   +O\left(\max\{(|z|+1)^{\alpha_1-N},
   (|w|+1)^{\alpha_1-N}\}|z-w|^{N}\right),
\]
for any fixed $N\ge1$. Applying this formula with $z\to z^2ue^{it}
\cos t$ and $w\to z^2$, we get
\[
\begin{split}
    \tilde{f}(z^2ue^{it}\cos t)
    &=\sum_{0\le j<N}
    \frac{\tilde{f}^{(j)}(z^2)}{j!}\,z^{2j}
    (ue^{it}\cos t-1)^j \\
    &\qquad +O\left(\bigl(|z|^2(u+1)+1\bigr)^{\alpha_1}
    |ue^{it}\cos t-1|^N\right),
\end{split}
\]
and a similar expression for $\tilde{g}$. Substituting
these expressions with $N=4$ into $I(z)$, we get
\begin{equation}\label{I(z)}
\begin{split}
    I(z)&=\frac{z^2}{4\pi}\int_{-\theta_0/2}^{\theta_0/2}
    \!\int_0^\infty \!e^{-z^2(u-1)^2\cos^2 t-z^2\sin^2 t}
    \tilde{f}(z^2ue^{it}\cos t)\tilde{g}(z^2ue^{-it}\cos t)u
    \dd u \cos^2 t\dd t\\
    &=\sum_{k,l\le3}\frac{\tilde{f}^{(k)}(z^2)
    \tilde{g}^{(l)}(z^2)}{k!l!}\cdot \frac{z^{2(1+k+l)}}{4\pi}\\
    &\quad\times\int_{-\theta_0/2}^{\theta_0/2}\! \int_0^\infty
    \!e^{-z^2(u-1)^2\cos^2 t-z^2\sin^2 t}(ue^{it}\cos t-1)^k
    (ue^{-it}\cos t-1)^lu\dd u \cos^2 t \dd t\\
    &\qquad+O\left(|z|^2\sum_{\substack{0\le k,l\le 4\\ k+l>3}}
    I_{\alpha_1+\alpha_2,k+l}\right),
\end{split}
\end{equation}
where
\[
    I_{\rho,\kappa}:=\int_{-\theta_0/2}^{\theta_0/2}
    \!\int_0^\infty e^{-z^2(u-1)^2\cos^2 t-z^2\sin^2 t}
    u\dd u (|z|^2(u+1)+1)^\rho |ue^{it}\cos t-1|^\kappa\dd t.
\]

Applying now the inequality
\[
    |ue^{it}\cos t-1|=|u\cos t-e^{-it}|
    =|u\cos t-\cos t +i\sin t|\le|u-1|+|t|,
\]
we get
\[
    I_{\rho,\kappa}=O\left(\int_{-\theta_0/2}^{\theta_0/2}
    \!\int_0^\infty\!\! e^{-\Re (z^2)(u-1)^2\cos^2 t-\Re(z^2)\sin^2 t}
    u(|z|^2(u+1)+1)^\rho (|u-1|^\kappa+|t|^\kappa )\dd u\dd t\right).
\]
Note that $\Re(z^2)\ge|z|^2\cos 2\theta$ and $\sin^2 t\ge c_1t^2$
when $|\arg (z)|\le\theta$ and $|t|\le\theta_0/2$ for some constant
$c_1>0$. Thus there exists a positive constant $c>0$ such that
\[
   \Re (z^2)(u-1)^2\cos^2 t+\Re(z^2)\sin^2 t
   \ge c (u-1)^2|z|^2+ct^2,
\]
for $|\arg (z)|\le\theta$ and $|t|\le\theta_0/2$. It then follows
that
\[
\begin{split}
    I_{\rho,\kappa}&=O\left(
    \int_{-\theta_0/2}^{\theta_0/2} \int_0^\infty
    e^{-c(u-1)^2|z|^2-c|z|^2 t^2}u (|z|^2(u+1)+1)^\rho
    (|u-1|^\kappa+|t|^\kappa )\dd u\dd t\right)\\
    &=O\Bigg(\left(\int_0^\infty e^{-c(u-1)^2|z|^2}
    (|z|^2(u+1)+1)^\rho |u-1|^\kappa u\dd u\right)
    \int_{-\theta_0/2}^{\theta_0/2}e^{-c|z|^2 t^2}\dd t\\
    &\qquad +\left(\int_0^\infty e^{-c(u-1)^2|z|^2}
    (|z|^2(u+1)+1)^\rho u\dd u\right)
    \int_{-\theta_0/2}^{\theta_0/2}
    e^{-c|z|^2 t^2}|t|^\kappa \dd t\Bigg)\\
    &=O\left((|z|^2+1)^\rho |z|^{-\kappa-2}\right).
\end{split}
\]
Substituting this bound in the error term of \eqref{I(z)}, we obtain
\[
\begin{split}
    I(z) =\sum_{0\le k,l\le3}\frac{\tilde{f}^{(k)}(z^2)
    \tilde{g}^{(l)}(z^2)}{k!l!}\,S_{k,l}
    +O\left(|z|^{2\alpha_1+2\alpha_2-4}\right),
\end{split}
\]
where
\[
\begin{split}
    S_{k,l}&=\frac{z^{2(1+k+l)}}{4\pi} \\
    &\quad\times\int_{-\theta_0/2}^{\theta_0/2} \!\int_0^\infty
    \!\!e^{-z^2(u-1)^2\cos^2 t-z^2\sin^2 t}(ue^{it}\cos t-1)^k
    (ue^{-it}\cos t-1)^lu\dd u \cos^2 t \dd t.
\end{split}
\]
We can approximate the integral $S_{k,l}$ by reversing the order of the
procedure by which we obtained it. First, making the change of
variables $u\mapsto u /(z\cos t)$, we get
\[
\begin{split}
    S_{k,l}&=\frac{1}{4\pi}\int_{-\theta_0/2}^{\theta_0/2}\!
    \int_0^\infty e^{- (u-z\cos t)^2-z^2\sin^2 t}
    (zue^{it}-z^2)^k(zue^{-it}-z^2)^lu\dd u\dd t\\
    &=\frac{e^{-z^2}}{4\pi}\int_{-\pi}^{\pi}
    \int_0^\infty e^{- 2zu\cos t}(zue^{it}-z^2)^k
    (zue^{-it}-z^2)^lue^{-u^2}\dd u\dd t\\
    &\quad+O\left((1+|z|^{k+l})
    e^{\Re^2(z)\cos^2(\theta_0/2)-\Re(z^2)}\right)\\
    &=\frac{e^{-z^2}}{4\pi} \int_0^\infty
    \left(\int_{-\pi}^{\pi}
    \bigl[e^{zue^{it}}(zue^{it}-z^2)^k\bigr]
    \bigl[e^{zue^{-it}}(zue^{-it}-z^2)^l\bigr]
    \dd t\right)ue^{-u^2}\dd u\\
    &\quad+O\left((1+|z|^{k+l})
    e^{\Re^2(z)\cos^2(\theta_0/2)-\Re(z^2)}\right)\\
    &=e^{-z^2}\sum_{n\ge0}\frac{\nu_{n,k}\nu_{n,l}}{n!}\,
    z^{2n} +O\left((1+|z|^{k+l})
    e^{\Re^2(z)\cos^2(\theta_0/2)-\Re(z^2)}\right),
\end{split}
\]
where the $\nu_{n,k}$'s are defined by
\[
    \sum_{n\ge 0}\frac{\nu_{n,k}}{n!}\,w^n
    =e^{w}(w-z^2)^k.
\]
In particular,
\[\begin{split}
    S_{1,1} &=e^{-z^2}\left(z^4+\sum_{n\ge1}
    \frac{\bigl(n-z^2\bigr)^2}{n!}\,z^{2n}\right)+
    O\left((1+|z|^{2})
    e^{\Re^2(z)\cos^2(\theta_0/2)-\Re(z^2)}\right)\\
    &=z^2+O\left((1+|z|^{2})
    e^{\Re^2(z)\cos^2(\theta_0/2)-\Re(z^2)}\right).
\end{split}
\]
Similarly,
\[
    S_{0,0}=1
    +O\left(e^{\Re^2(z)\cos^2(\theta_0/2)-\Re(z^2)}\right),
\]
and
\[
    S_{m,n}=O\left((1+|z|^{m+n})
    e^{\Re^2(z)\cos^2(\theta_0/2)-\Re(z^2)}\right) ,
\]
whenever $m\not=n$. Therefore
\begin{align*}
    \tilde{h}(z^2) &=I(z)
    +O\left((1+|z|^{\alpha_1+\alpha_2})
    e^{\Re^2(z)\cos^2(\theta_0/2)-\Re(z^2)}\right)\\
    &=\tilde{f}(z^2)\tilde{g}(z^2)+z^2\tilde{f}'(z^2)
    \tilde{g}'(z^2)
    +O\left(|z|^{2\alpha_1+2\alpha_2-4}\right).
\end{align*}
Accordingly,
\begin{align*}
    \tilde{h}(z)&=I(\sqrt{z})
    +O\left((1+|z|^{\alpha_1+\alpha_2})
    e^{\Re^2(z)\cos^2(\theta_0/2)-\Re(z^2)}\right) \\
    &=\tilde{f}(z)\tilde{g}(z)
    +z\tilde{f}'(z)\tilde{g}'(z)
    +O\left(|z|^{\alpha_1+\alpha_2-2}\right).
\end{align*}
This proves the second part and completes the proof of the
proposition.
\end{proof}

\section{Asymptotic variance of trie statistics}
\label{sec-tries}

We address in this section the asymptotic variance of general trie
statistics.

Let $X_n$ be an additive shape parameter in a random trie of size
$n$. Then $X_n$ satisfies the following distributional recurrence
\begin{align}\label{Xn-rrr}
    X_n\stackrel{d}{=}X_{I_n}+X_{n-I_n}^{*}+T_n
    \qquad (n\ge 2),
\end{align}
where $I_n=\mathrm{Binom}(n,p)$, $X_n\stackrel{d}{=}X_n^{*}$, and
$X_n,X_n^{*},I_n,T_n$ are independent. Without loss of generality,
we may assume that $X_0=0$ and $X_1=0$ (only minor modifications
needed when under nonzero initial conditions). Changing the value of
$X_0$ and $X_1$ affects only the mean but not the variance.

Consider first the moment-generating functions $M_n(y):={\mathbb
E}(e^{X_n y})$. Then, by \eqref{Xn-rrr},
\[
    M_n(y)= \mathbb{E}(e^{T_n y})\sum_{0\le k\le n}\pi_{n,k}
    M_k(y)M_{n-k}(y)\qquad(n\ge 2),
\]
where $\pi_{n,k}:= \binom{n}{k}p^kq^{n-k}$. By taking derivatives,
we obtain the recurrences for the first two moments ($\mu_n :=
\mathbb{E}(X_n)$ and $s_n := \mathbb{E}(X_n^2)$)
\begin{align}\label{Xn-mu-sigma}
\begin{split}
    \mu_n &=\sum_{0\le k\le n}\pi_{n,k}
    \left(\mu_k+\mu_{n-k}\right) +\mathbb{E}(T_n)\\
    s_n&=\sum_{0\le k\le n}
    \pi_{n,k}\left(s_k+s_{n-k}\right)
    +\mathbb{E}(T_n^2) \\ &\qquad \qquad
    +2\sum_{0\le k\le n}\pi_{n,k}
    \bigl(\mu_k\mu_{n-k}+\mathbb{E}(T_n)
    \left(\mu_k+\mu_{n-k}\right) \bigr).
\end{split}
\end{align}
Our major interest lies in the variance $\sigma_n^2 := 
\mathbb{V}(X_n)$, which also satisfies the same type of recurrence
\[
    \sigma_n^2 = \sum_{0\le k\le n}\pi_{n,k}
    \left(\sigma^2_k+\sigma^2_{n-k}\right) +\mathbb{V}(T_n)
    + \sum_{0\le k\le n} \pi_{n,k} \Delta_{n,k}^2,
\]
where $\Delta_{n,k} := \mu_k + \mu_{n-k} -\mu_n +\mathbb{E}(T_n)$.

Different approaches have been proposed to the asymptotics of
$\sigma_n^2$; these include an elementary induction approach (see
\cite{chen03a,hubalek02a}), the second-moment approach (see
\cite{kirschenhofer88b,kirschenhofer89a,kirschenhofer89b,
kirschenhofer93a}), Poissonized variance (by considering 
$\tilde{f}_2(z) - \tilde{f}_1^2(z)$) approach (see
\cite{jacquet88a,jacquet95a,PF159, regnier89a,mahmoud92a,park08a}),
(bivariate) characteristic function approach
(\cite{jacquet86a,mahmoud92a,jacquet95a}), and Schachinger's
differencing approach \cite{schachinger95b}.

The slight modification of our approach, which relies on
\eqref{Poi-Var}, from the usual Poissonized variance one turns out
to be very helpful and makes a significant difference, notably in
the resulting expressions for the periodic functions, mostly because
the cancelation is avoided (somehow incorporated in the generating
functions).

\subsection{Analytic schemes for the mean and the variance}

The tools we developed in Section~\ref{sec-hada} are useful in
establishing simple, general, analytic frameworks under which
asymptotics of the mean and the variance can be easily derived by
checking only a few sufficient conditions.

\paragraph{Asymptotics of the mean.}
Denote by $\tilde{f}_1(z)$ and $\tilde{g}_1(z)$ the Poisson
generating function of $\mathbb{E}(X_n)$ and $\mathbb{E}(T_n)$,
respectively. Then
\[
    \tilde{f}_1(z)
    =\tilde{f}_1(pz)+\tilde{f}_1(qz)+\tilde{g}_1(z),
\]
with $\tilde{f}_1(0)=\tilde{f}_1'(0)=0$.

\begin{thm}\label{mean}
Let $0<\theta<\pi/2, \alpha<1$ and $\beta\in\mathbb{R}$. If either
$\tilde{g}_1\in\JS_{\!\!\!\alpha,\beta}$ or $\tilde{g}_1(z) \in\JS$,
and $\tilde{g}_1(z)= cz+O\left(\vert z\vert^{\alpha} (\log_{+}\vert
z\vert)^{\beta}\right)$ uniformly as $|z|\to\infty$ and
$\vert\arg(z)\vert\le\theta$, where $c\in\mathbb{R}$, then
\[
    \frac{\mathbb{E}(X_n)}{n} = \frac{c}{h}\log n+ d
    +\mathscr{F}[G_1](r\log_{1/p}n)+o(1),
\]
where $d=G_1(-1)/h$ if $c=0$ and $d=h_0$ (see \eqref{h0}) if $c\ne0$,
$G_1(s) := \mathscr{M}[\tilde{g}_1;s]$ and the other notations are 
described as in Proposition~\ref{prop-at2}.
\end{thm}
\begin{proof}
Combining Propositions \ref{prop-PC} and \ref{prop-at}, we have
\[
    \mathbb{E}(X_n)=\sum_{0\le j<2k}
    \frac{\tilde{f}_1^{(j)}(n)}{j!}\,\tau_j(n)
    +O\left(n^{1-k}\right),
\]
for $k=0,1,\dots$. Then apply Proposition~\ref{prop-at2}.
\end{proof}

If the initial conditions are not zero, say $X_0=a$ and $X_1=b$,
then we consider $\bar{f}_1(z) := \tilde{f}_1(z)-(b-a)z-a$, leading
to the functional equation
\[
    \bar{f}_1(z) = \bar{f}_1(pz)+
    \bar{f}_1(qz) + \tilde{g}_1(z) + a(1-(1+z)e^{-z}),
\]
which results in an additional linear term for $\mathbb{E}(X_n)$
of the form
\[
    \left(b-a+\frac{a}{h}
    +a\mathscr{F}[G_0](c\log_{1/p}n)\right) n,
\]
where $G_0(s) := -(s+1)\Gamma(s)$. Such an additional term is
related to the expected size of tries; see Section \ref{sec-size}.

\paragraph{Functional equations related to the variance.}
For the variance, we begin with the second moment. Let
$\tilde{f}_2(z)$ and $\tilde{g}_2(z)$ be the Poisson generating
function of $\mathbb{E}(X_n^2)$ and $\mathbb{E}(T_n^2)$,
respectively. Then, by \eqref{Xn-mu-sigma},
\[
    \tilde{f}_2(z)=\tilde{f}_2(pz)+\tilde{f}_2(qz)
    +2\tilde{f}_1(pz)\tilde{f}_1(qz)+\tilde{g}_2(z)
    +\tilde{h}_2(z),
\]
where
\begin{align} \label{h2z}
    \tilde{h}_2(z)&=2e^{-z}\sum_{n\ge0}\mathbb{E}(T_n)
    \sum_{0\le k\le n}\pi_{n,k}
    \left(\mu_k+\mu_{n-k}\right)
    \frac{z^{n}}{n!}\\
    &=2e^{-z}\sum_{n\ge 0}\mathbb{E}(T_n)
    \mu_n\frac{z^n}{n!}
    -2e^{-z}\sum_{n\ge 0}(\mathbb{E}(T_n))^2\frac{z^n}{n!},
    \nonumber
\end{align}
the last two terms being Hadamard products.

Now, let
\begin{align*}
    \tilde{V}_X(z)&:=\tilde{f}_2(z)-
    \tilde{f}_1(z)^2-z\tilde{f}_1'(z)^2\\
    \tilde{V}_T(z)&:=\tilde{g}_2(z)-
    \tilde{g}_1(z)^2-z\tilde{g}_1'(z)^2.
\end{align*}
Then by a straightforward computation
\begin{align}\label{VXz}
    \tilde{V}_X(z)=\tilde{V}_X(pz)+\tilde{V}_X(qz)
    +\tilde{V}_T(z)+\tilde{\phi}_1(z)+\tilde{\phi}_2(z),
\end{align}
where
\begin{align}\label{s1-s2}
\begin{split}
    \tilde{\phi}_1(z)&:=\tilde{h}_2(z)-2\tilde{g}_1(z)
    \left(\tilde{f}_1(pz)+\tilde{f}_1(qz)\right)
    -2z\tilde{g}'_1(z)\left(p\tilde{f}'_1(pz)
    +q\tilde{f}'_1(qz)\right)\\
    \tilde{\phi}_2(z)&:=pqz\left(\tilde{f}'_1(pz)
    -\tilde{f}'_1(qz)\right)^2.
\end{split}
\end{align}

\paragraph{Ideas of our approach.} We sketch here the underlying
ideas used in our approach before presenting a simple analytic
scheme for the asymptotics of the variance. We assume first that
$\tilde{g}_1\in\JS$. This implies the JS-admissibility of
$\tilde{f}_1$, and thus, by Proposition~\ref{prop-PC}, we have the
asymptotic expansion for the mean
\[
    \mu_n = \sum_{0\le j<2k}
    \frac{\tilde{f}_1^{(j)}(n)}{j!}\,\tau_j(n)
    +O\left(\tilde{f}_1(n) n^{-k} \right)\qquad(k=1,2,\dots).
\]
If we also assume $\tilde{g}_2\in\JS$, then we have the same
type of expansion for $\mathbb{E}(X_n^2)$ with $\tilde{f}_1$ there
replaced by $\tilde{f}_2$. Thus (dropping error terms for
convenience of presentation)
\[
    \sigma_n^2 \sim \sum_{0\le j<2k}
    \frac{\tilde{f}_2^{(j)}(n)}{j!}\,\tau_j(n)
    - \left(\sum_{0\le j<2k}
    \frac{\tilde{f}_1^{(j)}(n)}{j!}\,\tau_j(n)\right)^2.
\]
Now substituting $\tilde{f}_2 = \tilde{V}_X + \tilde{f}_1^2
+z(\tilde{f}_1')^2$ yields \emph{formally}
\[
    \sigma_n^2 \sim \tilde{V}_X(n)
    - \frac{n}{2}\tilde{V}_X''(n) -\frac{n^2}2\tilde{f}_1''(n)^2
    + \cdots,
\]
under suitable growth conditions and a suitably chosen $k$. Thus the
asymptotics of the variance is reduced to that of $\tilde{V}(n)$ and
its derivatives. Further extensions of this approach are discussed in
detail elsewhere.

\paragraph{Asymptotics of the variance.}

We now show that the variance of $X_n$ can also be handled in a
general way by reducing the required asymptotics to essentially
checking conditions for JS-admissibility.
\begin{thm}\label{var}
Let $0<\theta<\pi/2,\alpha<1$ and $\beta\in\mathbb{R}$. Assume
$\tilde{g}_2\in\JS$ and $\tilde{V}_T(z)=O\left(|z|^{\alpha}
(\log_{+}|z|)^{\beta}\right)$ as $|z|\rightarrow\infty$ in the
sector $|\arg(z)|\le \theta$.
\begin{itemize}
\item[(a)] If $p=q=1/2$, and
$\tilde{g}_1\in\JS_{\!\!\!\alpha,\beta}$ or
$\tilde{g}_1\in\JS_{\!\!\!1,0}$, then
\begin{align}\label{symm-var}
    \frac{\mathbb{V}(X_n)}{n}
    =\frac{1}{\log 2}\sum_{k\in\mathbb{Z}}
    \Phi_1(-1+\chi_k)n^{-\chi_k} +o(1),
\end{align}
where $\chi_k=\frac{2k\pi i}{\log 2}$ and $\Phi_1(s)=
\mathscr{M}[\tilde{V}_T+\tilde{\phi}_1;s]$.

\item[(b)] Assume $p\ne q$.
\begin{itemize}

\item[(i)] If $\tilde{g}_1\in\JS_{\!\!\!\alpha,\beta}$, then
\[
    \frac{\mathbb{V}(X_n)}{n} = \frac{G(-1)}{h}+
    \mathscr{F}[G](r\log_{1/p}n) +o(1).
\]
Here $G(s)=\Phi_1(s)+\Phi_2(s)$, where $\Phi_1(s) =\mathscr{M}
[\tilde{V}_T+\tilde{\phi}_1;s]$ and $\Phi_2(s)$ is the analytic
continuation of $\mathscr{M}[\tilde{\phi}_2;s]$.

\item[(ii)] If $\tilde{g}_1\in\JS$ and $\tilde{g}_1=z +O(\vert
z\vert^{\alpha}(\log_{+}\vert z\vert)^{\beta})$ uniformly as
$|z|\to\infty$ and $\vert\arg(z)\vert\le\theta$, then
\begin{align*}
    \frac{\mathbb{V}(X_n)}{n} &=
    \frac{pq\log^2(p/q)}{h^3}\log n
    +\frac{d}{h}+\frac{pq\log^2(p/q)(p\log^2p+q\log^2 q)}{2h^4} \\
    &\qquad +\mathscr{F}[G](r\log_{1/p}n)
    +o(1),
\end{align*}
Here $G(s) = \Phi_1(s)+\Phi_2(s)$, 
where $\Phi_1(s) = \mathscr{M}[\tilde{V}_T+\tilde{\phi}_1;s]$ and 
$\Phi_2(s)$ is the meromorphic continuation of 
$\mathscr{M}[\tilde{\phi}_2;s]$, and 
$d=\Phi_1(-1)+\lim_{s\rightarrow -1}
(\Phi_2(s)+pq\log^2(p/q)/(h^2(s+1)))$.
\end{itemize}
\end{itemize}
\end{thm}
\pf Since $\tilde{V}_T$ is assumed to be small (less than linear),
we first show that, under the assumptions of the theorem, both
$\tilde{\phi}_1$ and $\tilde{\phi}_2$ are also small; see \eqref{VXz}.

If $\tilde{g}_1\in\JS_{\!\!\!\alpha,\beta}$, then $\tilde{f}_1
\in\JS_{\!\!\!1,0}$ by Theorem \ref{mean}. These imply, by
Proposition \ref{prop-hada}, that $\tilde{h}_2\in\JS$ and
\begin{align*}
    \tilde{h}_2(z)&=2\tilde{g}_1(z)\tilde{f}_1(z)
    +2z\tilde{g}'_1(z)\tilde{f}'_1(z)
    -2\tilde{g}_1(z)^2-2z\tilde{g}'_1(z)^2 + O\left(\vert
    z\vert^{\alpha-1}
    (\log_{+}\vert z\vert)^{\beta}\right)\\
    &=2\tilde{g}_1(z)\left(\tilde{f}_1(pz)
    +\tilde{f}_1(qz)\right)+2z\tilde{g}'_1(z)\left(p\tilde{f}'_1(pz)
    +q\tilde{f}'_1(qz)\right)\\ &\qquad\qquad +
    O\left(\vert z\vert^{\alpha-1}
    \left(\log_{+}\vert z\vert\right)^{\beta}\right),
\end{align*}
uniformly as $|z|\to\infty$ and $\vert\arg(z)\vert\le\theta$. It
follows, by \eqref{s1-s2}, that $\tilde{\phi}_1(z)=O\left(\vert
z\vert^{\alpha-1}\left(\log_{+}\vert z\vert\right)^{\beta}\right)$.
Similarly, if $\tilde{g}_1\in\JS_{\!\!\!1,0}$, then
$\tilde{h}_2\in\JS$ and $\tilde{\phi}_1(z)=O\left(\log_{+}\vert
z\vert\right)$ as $|z|\rightarrow\infty$ in the sector
$\vert\arg(z)\vert\le\theta$.

Without loss of generality, we may assume that all generating 
functions $f$ involved here have the property that $f(0)=f'(0)=0$.

Consider first $\tilde{V}_T+\tilde{\phi}_1$. By assumption and by
the preceding analysis, we see that $\mathscr{M}[\tilde{V}_T
+\tilde{\phi}_1;s]$ exists in the strip $\langle
-2,-\alpha-\ve\rangle$, with $\ve>0$ arbitrarily small. Thus we
argue as in Proposition \ref{prop-at2} and obtain
\eqref{symm-var}. Note that $\theta>0$ is crucial here.

We now turn to $\tilde{\phi}_2$, which is zero when $p= q$. So
assume now $p\ne q$. If $\tilde{g}_1\in\JS_{\!\!\!\alpha,\beta}$,
then its Mellin transform $G_1$ exists in the strip $\langle
-2,-\alpha-\ve \rangle$ and, by applying the Exponential Smallness
Lemma, is exponentially small at $c\pm i\infty$. Thus from the
integral representation
\[
    \tilde{f}_1'(pz)- \tilde{f}_1'(qz)
    = -\frac1{2\pi i}\int_{(-1-\ve)}
    \frac{wG_1(w)(p^{-w-1}-q^{-w-1})
    }{1-p^{-w}-q^{-w}}\,z^{-w-1}\dd w,
\]
it follows that $\tilde{f}_1'(pz)- \tilde{f}_1'(qz) = o(1)$ and
consequently by \eqref{s1-s2} $\tilde{\phi}_2(z)=o(\vert z\vert)$ as
$|z|\to\infty$ and $\vert\arg(z)\vert\le\theta$. Thus the Mellin
transform $\Phi_2(s)$ of $\tilde{\phi}_2(z)$ exists in the strip
$\langle-3,-1\rangle$ and
\begin{align}\label{Phi-2-s}
\begin{split}	
    \Phi_2(s) &= \frac{pq}{2\pi i}\int_{(-1/2)}
    \frac{(p^{-w}-q^{-w})(p^{w-1-s}-q^{w-1-s})}
    {(1-p^{1-w}-q^{1-w})(1-p^{w-s}-q^{w-s})}\\
    &\hspace*{3cm} \times (w-1)G_1(w-1)(s-w)G_1(s-w)\dd w.
\end{split}
\end{align}
Note that
\[
    \frac{p^{-w}-q^{-w}}{1-p^{1-w}-q^{1-w}} =
    \frac{-(1-p^{-w})+(1-q^{-w})}
    {p(1-p^{-w})+q(1-q^{-w})}.
\]
If $\log p/\log q = r/\ell \in\mathbb{Q}$, where $(r,\ell)=1$
are positive integers, then any zero of the form $2rk\pi i/\log p$
of the denominator is also a zero of the numerator. Thus the
integration path can be moved to the imaginary axis.

By summing over all residues of poles with real parts
less than $-\alpha$ (see \cite{PF055} for a detailed study), we
see that $\Phi_2(s)$ can be extended to a
meromorphic function beyond the line $\Re(s)=-1$ which is analytic
on $\Re(s)=-1$. Consequently, the asymptotic estimate in case
(\emph{b})-(\emph{i}) follows as in Proposition \ref{prop-at2}.

The analysis for the last part (\emph{ii}) is  similar with the only
difference that now $\tilde{\phi}_2(z)=pq\log^2(p/q)z/h^2+o(\vert
z\vert)$ uniformly as $|z|\to\infty$ and $\vert\arg(z)\vert
\le\theta$. Hence, one can again extend $\mathscr{M}
[\tilde{\phi}_2; s]$ to a meromorphic function beyond the line
$\Re(s)=-1$, but there is a simple pole on $\Re(s)=-1$ at $s=-1$
with the singular expansion $\mathscr{M}[\tilde{\phi}_2; s]=
-pq\log^2(p/q)/(h^2(s+1))+d+\cdots$. Thus similar arguments used in
Proposition \ref{prop-at2} apply. This completes the proof. \qed

\paragraph{Calculation of the Fourier coefficients
$\Phi_2(-1+\chi_k)$.} We outline here an approach by residue 
calculus to simplify the Fourier coefficients $\Phi_2(-1+\chi_k)$,
which will be applied several times later. 

We begin with the integral representation \eqref{Phi-2-s}, which we 
first shift to the imaginary axis. Then, we use the following 
decomposition
\begin{align*}
    \Phi_2(-1+\chi_k) &= \frac{1}{2\pi i}\int_{(0)^+}
    \left(\frac1{1-p^{1-w}-q^{1-w}}+ \frac{p^{1+w}-q^{1+w}}
    {1-p^{1+w}-q^{1+w}}\right)\\
    &\hspace*{3cm} \times (w-1)G_1(w-1)(-1+\chi_k-w)
    G_1(-1+\chi_k-w)\dd w \\
    &=: J_0+J,
\end{align*}
where the integration contour $\int_{(0)^+}$ is the imaginary axis but
with a sufficiently small indentation to the right of each zero
of the equation $1-p^{1-it}-q^{1-it}=0$ for real $t$ (only one
when $\frac{\log p}{\log q}\ne\mathbb{Q}$, and an infinity number
of equally-spaced ones otherwise).

By the change of variables $w\mapsto \chi_k-w$ and then by moving
the line of integration to the right, we have
\begin{align*}
    J_0 &= \frac1{2\pi i} \int_{(0)^-}
    \frac{(w-1)G_1(w-1)(-1+\chi_k-w)
    G_1(-1+\chi_k-w)}{1-p^{1+w}-q^{1+w}} \dd w\\
    &= -\frac1h\sum_{j\in\mathbb{Z}}
    (\chi_j-1)G_1(\chi_j-1)(-1+\chi_{k-j})
    G_1(-1+\chi_{k-j})\\
     &\qquad +  \frac1{2\pi i} \int_{(0)^+}
    \frac{(w-1)G_1(w-1)(-1+\chi_k-w)
    G_1(-1+\chi_k-w)}{1-p^{1+w}-q^{1+w}} \dd w,
\end{align*}
where the integration contour $\int_{(0)^-}= -\int_{(0)^+}$. The last
integral equals
\[
    \frac1{2\pi i} \int_{(0)^+}
    (w-1)G_1(w-1)(-1+\chi_k-w)
    G_1(-1+\chi_k-w) \dd w + J.
\]
Note that
\[
    \frac1{2\pi i} \int_{(0)^+}
    (w-1)G_1(w-1)(-1+\chi_k-w)
    G_1(-1+\chi_k-w) \dd w = \int_0^\infty
    \tilde{g}_1'(t)^2 t^{-1+\chi_k} \dd t.
\]
Combining these relations, we obtain
\begin{align}\label{Phi-1-chi}
\begin{split}
    \Phi_2(-1+\chi_k) &= 2J-\frac1h \sum_{j\in\mathbb{Z}}
    (\chi_j-1)G_1(\chi_j-1)(-1+\chi_{k-j})
    G_1(-1+\chi_{k-j}) \\
    &\qquad + \int_0^\infty
    \tilde{g}_1'(t)^2 t^{-1+\chi_k} \dd t .
\end{split}	
\end{align}	
To further simplify the integral $J$, we write
\[
    \tilde{g}_1(z) = \sum_{j\ge2}\frac{\tilde{b}_j}{j!}\, z^j.
\]
Then it follows from the Direct Mapping Theorem (see
\cite{PF120}) of Mellin transform that $G_1(s)$ can be extended to
a meromorphic function to the left of $\Re(s)=-2$ with simple poles
at $s=-j$, the residue there being equal to $\tilde{b}_j/j!$.

If we assume that $-(s-1)G_1(s-1)$ has no singularity to the right
of imaginary axis, then we obtain
\begin{align} \label{J}
    J = \sum_{j\ge2}\frac{\tilde{b}_j(p^j+q^j)}{(j-1)!
    (1-p^j-q^j)}(j-2+\chi_k)G_1(j-2+\chi_k).
\end{align}
This and \eqref{Phi-1-chi} will be useful later.

This procedure is very effective in many applications having
linear variance (namely, the situation of Theorem \ref{var}
(\emph{b})-(\emph{i})); similar but slightly more involved arguments 
can be used in more general situations such as $n\log n$-variance. 

\section{Applications}
\label{sec-apps}

We apply or slightly modify the schemes developed in the previous
sections to a few standard examples in the literature for which new
results are proposed for the asymptotics of the variance.

\subsection{Size of random tries}
\label{sec-size}

The size of a trie is defined to be the number of internal nodes
used, which becomes a random variable when the input sequence is
random. For example, eight internal nodes are used in the trie in
Figure~\ref{fg-trie}. Under our Bernoulli model, we see that the
size $X_n$ satisfies \eqref{Xn-rr} with $X_0=X_1=0$ and $T_n=1$ for
$n\ge2$, where $n$ denotes the total number of input keys (or
external nodes). Under different guises and different initial
conditions, this is the most studied random variable defined on
tries or related structures in the literature, most of them dealing
with the expected value and very few of them with the variance. See,
for example, \cite{bourdon01a,PF161,devroye05a,janson12a,
massey81a,myoupo03a,shiau05a} and the references therein for the
mean, and \cite{jacquet88a,janssen00a,kaplan85a,kirschenhofer88b,
kirschenhofer91a,mahmoud92a,regnier89a} for the variance.

Since $T_n=1$, we have
\[
    \tilde{g}_1(z)=\tilde{g}_2(z)=1-(1+z)e^{-z}.
\]
From Proposition \ref{prop-closure}, we see that both functions
$\tilde{g}_1,\tilde{g}_2\in \JS_{\!\!\!0,0}$. Also we have
\begin{align} \label{VT-size}
    \tilde{V}_T(z) := \tilde{g}_2(z)-\tilde{g}_1^2(z)
    -z\tilde{g}_1'(z)^2
    =e^{-z}(1+z-(1+2z+z^2+z^3)e^{-z}),
\end{align}
and (see \eqref{s1-s2})
\begin{align} \label{s1-size}
    \tilde{\phi}_1(z)= 2e^{-z}\left((1+z)
    \bigl(\tilde{f}_1(pz)+\tilde{f}_1(qz)\bigr)
    -z^2\bigl(p\tilde{f}'_1(pz)
    +q\tilde{f}'_1(qz)\bigr)\right).
\end{align}
Both functions are exponentially small for large $|z|$ with $\Re(z)>0$.

The application of both Theorems~\ref{mean} and \ref{var} is
straightforward. Since
\[
    G_1(s)=\mathscr{M}[\tilde{g}_1;s]
    =-(s+1)\Gamma(s),
\]
we thus obtain, when $X_0=X_1=0$,
\[
    \frac{\mathbb{E}(X_n)}{n}=
    \frac{1}{h}+\mathscr{F}[G_1](r\log_{1/p}n)
    +o(1),
\]
a well-known result. When $X_0=a$ and $X_1=b$, then a direct
modification of the same argument gives
\[
    \frac{\mathbb{E}(X_n)}{n}=
    b-a+(a+1)\left(\frac{1}{h}
    +\mathscr{F}[G_1](r\log_{1/p}n)\right)+o(1).
\]

As regards the variance, the functions involved become more
complicated. We state our results by distinguishing between
symmetric case $p=1/2$ and asymmetric case $p\not=q$.

\begin{thm}[Symmetric case: $p=1/2$] \label{thm-size-sym}
The variance of the size of random symmetric tries satisfies 
asymptotically ($\chi_k:=2k\pi i/\log 2$)
\[
    \frac{\mathbb{V}(X_n)}{n} 
    = \frac{1}{\log 2}\sum_{k\in{\mathbb Z}}G(-1+\chi_k)
    n^{-\chi_k}+o(1),
\]
where the mean value of the periodic function is given by
\begin{align}\label{v-ss}
\begin{split}
    \frac{G(-1)}{\log 2}=\frac{1}{\log 2}
    \left(\frac{1}{4}+2\sum_{j\ge 1}
    \frac{(-1)^j(j-1)}{2^j-1}\right)\approx
    0.845858623076001\cdots,
\end{split}
\end{align}
and for $k\not=0$
\begin{align}\label{v-ssk}
\begin{split}
    G(-1+\chi_k)=-\frac{\chi_k\Gamma(-1+\chi_k)(1+\chi_k)^2}{4}
    +2\sum_{j\geq 1}\frac{(-1)^jj(j(j+\chi_k)-1)
    \Gamma(j+\chi_k)}{(j+1)!(2^j-1)}.
\end{split}
\end{align}
\end{thm}
The numerical value in \eqref{v-ss} coincides with that given in
\cite{regnier89a}, where they derived the alternative expression
\begin{equation}\label{re}
    \frac{1}{2\log 2}-\frac{1}{\log^2 2}
    -\frac{2}{\log 2}\sum_{j\ge 1}\frac{(-1)^j}{2^j-1}-
    \frac{4\pi^2}{\log^3 2}\sum_{j\ge 1}
    \frac{j}{\sinh\frac{2j\pi^2}{\log 2}};
\end{equation}
see also \cite{kirschenhofer91a}. This expression can 
also be derived by the simplification procedure for deriving 
\eqref{Phi-1-chi} (see also Theorem \ref{thm-size-asym} below). 
Equating the above two expressions yields the identity
\begin{equation}\label{id}
    \sum_{j\ge1}\frac{(-1)^jj}{2^j-1}
    =\frac{1}{8}-\frac{1}{2\log 2}
    -\frac{2\pi^2}{\log^2 2}\sum_{j\ge
    1}\frac{j}{\sinh\frac{2j\pi^2}{\log 2}},
\end{equation}
which can be proved directly by the residue calculus similar to 
that used in deriving \eqref{Phi-1-chi}. Of special mention here
is that the series on the right-hand side is less than 
$1.1\times 10^{-10}$, meaning that the first two terms on the 
right-hand side already provide a very accurate approximation
to the series on the left-hand side. 
A third expression with the same numerical value is given in
\cite[Sec.\ 5.4]{mahmoud92a}
\[
    \frac1{\log 2}\left(\frac12+2\sum_{j\ge1}
    \frac1{2^j+1}\right)-\frac1{\log^22}-
    \frac{4\pi^2}{\log^3 2}\sum_{j\ge 1}
    \frac{j}{\sinh\frac{2j\pi^2}{\log 2}},
\]
which can be obtained from \eqref{re} by the identity
\[
    \sum_{j\geq 1}\frac{1}{2^j+1}
    =\sum_{j\geq 1}\frac{(-1)^{j-1}}{2^j-1}.
\]
Regarding the oscillating terms, Kirschenhofer
and Prodinger derived in \cite{kirschenhofer91a} (with terms slightly
simplified and with minor corrections)
\begin{align*}
    G(-1+\chi_k)&= -3\chi_k\Gamma(-1+\chi_k)
    -(1-\chi_k)(2-\chi_k)\Gamma(\chi_k)
    \left(\frac{1}{2}-\sum_{j\ge 1}\frac{(\chi_k+j)
    \binom{-\chi_k}{j-1}}{(j+1)(2^j-1)}\right)\\
    &\quad-\frac{\chi_k\Gamma(1+\chi_k)}{\log 2}
    -2\Gamma(1+\chi_k)\left(\frac{5-\chi_k}{4(1-\chi_k)}
    -\sum_{j\ge 1}\frac{(\chi_k+j+1)
    \binom{-\chi_k-1}{j-1}}{(j+1)(2^j-1)}\right)\\
    &\quad +\frac{1}{\log 2}
    \sum_{\substack{j+m=k\\j,m\ne 0}}
    \chi_j\Gamma(-1+\chi_j)\chi_m\Gamma(1+\chi_m),
\end{align*}
which is to be compared with our expression \eqref{v-ssk}.
Numerically, the amplitude of the oscillating part 
is bounded above by $\sum_{k\not=0}|G(-1+\chi_k)|/\log 2\le
1.7\times 10^{-6}$; see Figure~\ref{fg-var-size-1}. 

\begin{figure}[!ht]
\begin{center}
\includegraphics[width=5.2cm,height=3.6cm]{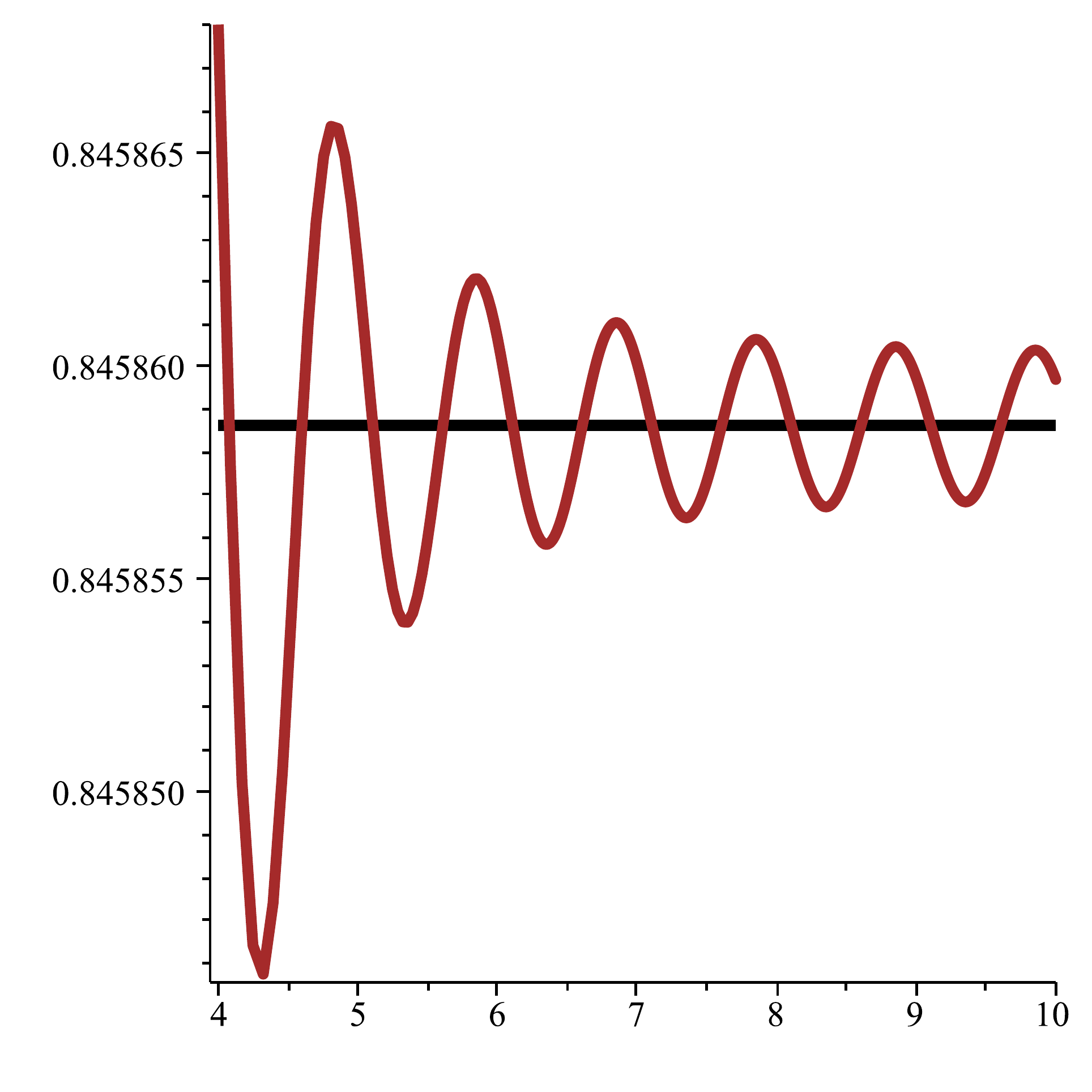}
\quad
\includegraphics[width=5cm,height=3.4cm]{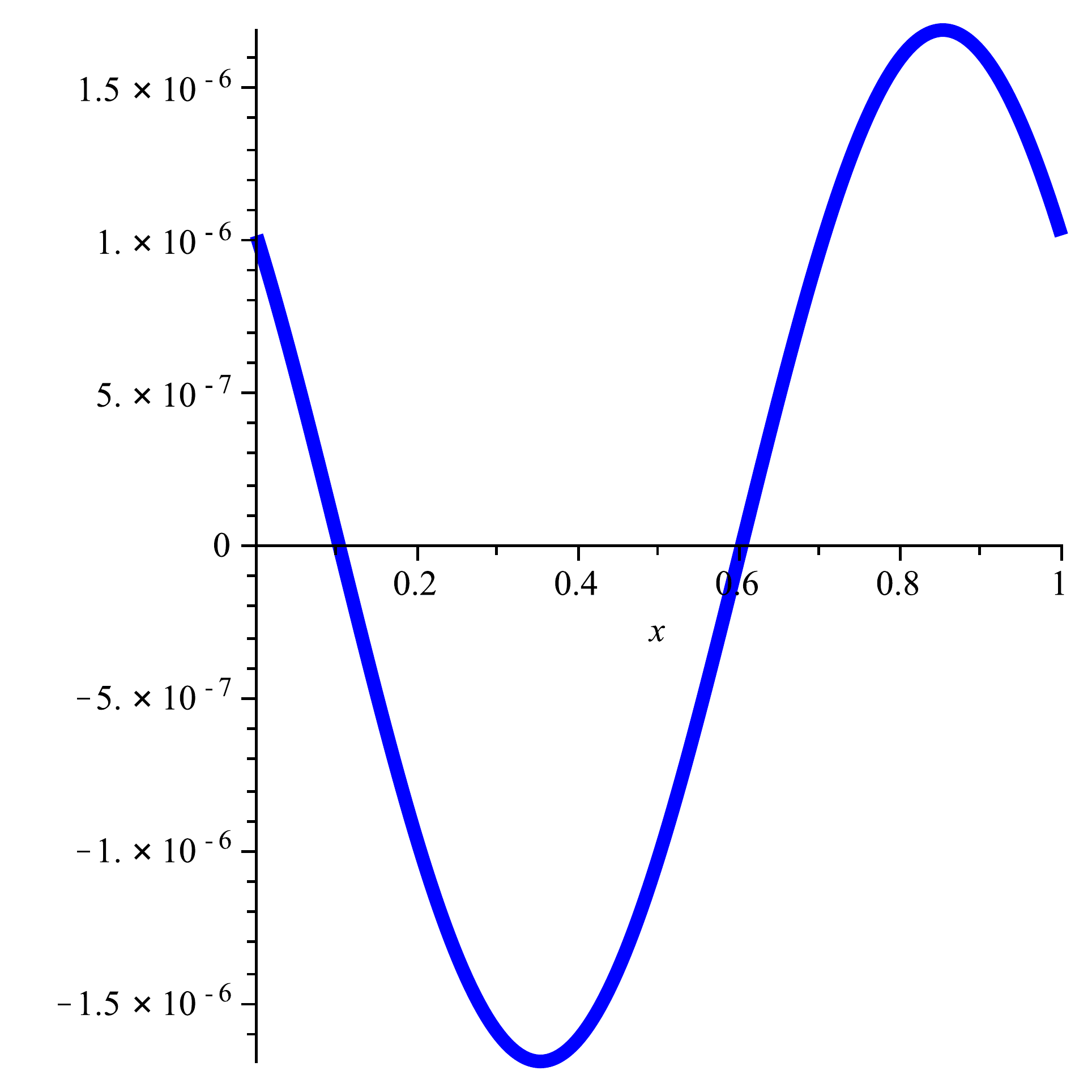}	
\end{center}
\caption{\emph{Periodic oscillations of the variance when $p=1/2$: 
$\mathbb{V}(X_n)/n$ in logarithmic scale (left) and the fluctuating 
part $\frac1{\log 2} \sum_{k\not=0} G(-1+\chi_k) e^{-2k\pi i x}$ 
(right).}} \label{fg-var-size-1}
\end{figure}	

We now state the result in the asymmetric case. 

\begin{thm}[Asymmetric case: $p\ne q$] \label{thm-size-asym}
The variance of the size of random asymmetric tries satisfies 
\[
    \frac{\mathbb{V}(X_n)}{n} = \frac{G(-1)}{h}+
    \mathscr{F}[G](r\log_{1/p}n)+o(1),
\]
where
\begin{equation} \label{var-size-G1}
\begin{split}
    G(-1) &= \frac12-\frac1h+ 2\sum_{j\ge2}
    \frac{(-1)^j(p^{j}+q^{j})}
    {1-p^{j}-q^{j}}\\
    &\qquad -\begin{cases}
        {\displaystyle\frac1{h\log p}\sum_{j\ge1}
        \frac{4rj\pi^2}{\sinh\frac{2rj\pi^2}{\log p}}},
        &\text{if } \frac{\log p}{\log q}\in{\mathbb Q};\\
        0,&\text{if } \frac{\log p}{\log q}\not\in{\mathbb Q},
    \end{cases}
\end{split}
\end{equation}
and for $k\not=0$ (only when 
$\frac{\log p}{\log q}\in\mathbb{Q}$)
\begin{align} 
\label{var-size}
\begin{split}
    G(-1+\chi_k)
    &=\chi_k\Gamma(-1+\chi_k)
    \left(1-\frac{\chi_k+3}{2^{1+\chi_k}}\right)
    -\frac{1}{h}\sum_{j\in\mathbb{Z}}
	\Gamma(\chi_j+1)\Gamma(\chi_{k-j}+1)\\
	&\qquad-2\sum_{j\ge 1}
	\frac{(-1)^j(j+1+\chi_k)\Gamma(j+\chi_k)
	\left(p^{j+1}+q^{j+1}\right)}{(j-1)!(j+1)
	\left(1-p^{j+1}-q^{j+1}\right)}.
\end{split}
\end{align}
\end{thm}
These expressions also hold in the symmetric case. However, 
the expressions for the Fourier coefficients in Theorem 
\ref{thm-size-sym} are simpler.

While the asymptotic pattern of the variance has long been known, 
the expressions for the Fourier coefficients have, as far as we
were aware, never been stated before in the above explicit forms.

Consider, for concreteness, the special rational case when $q=p^2$.
Then $p=(\sqrt{5}-1)/2$ is the golden ratio. From 
\eqref{var-size-G1}, we see that the non-periodic dominant term 
for the ratio between the variance and $n$ is given by 
\begin{align*}
    \frac{G(-1)}{h} &= \frac1h\left(
    \frac12-\frac1h+ 2\sum_{j\ge2}
    \frac{(-1)^j(p^{j}+p^{2j})}
    {1-p^{j}-p^{2j}}-\frac1{h\log p}\sum_{j\ge1}
    \frac{4j\pi^2}{\sinh\frac{2rj\pi^2}{\log p}}\right)\\
    &\approx 1.00834\,52644\,70994\dots,
\end{align*}
which is larger than the symmetric case \eqref{v-ss}. In general, 
$G(-1)=G(-1;p)$ is a symmetric bath-tub-shaped function of $p$
with its lowest value reached at $p=0.5$. 
The fluctuation of the periodic part is bounded above in modulus by 
$7.3\times 10^{-8}$; see Figure~\ref{fg-var-size-2}.

\begin{figure}[!ht]
\begin{center}
\includegraphics[width=5.2cm,height=3.6cm]{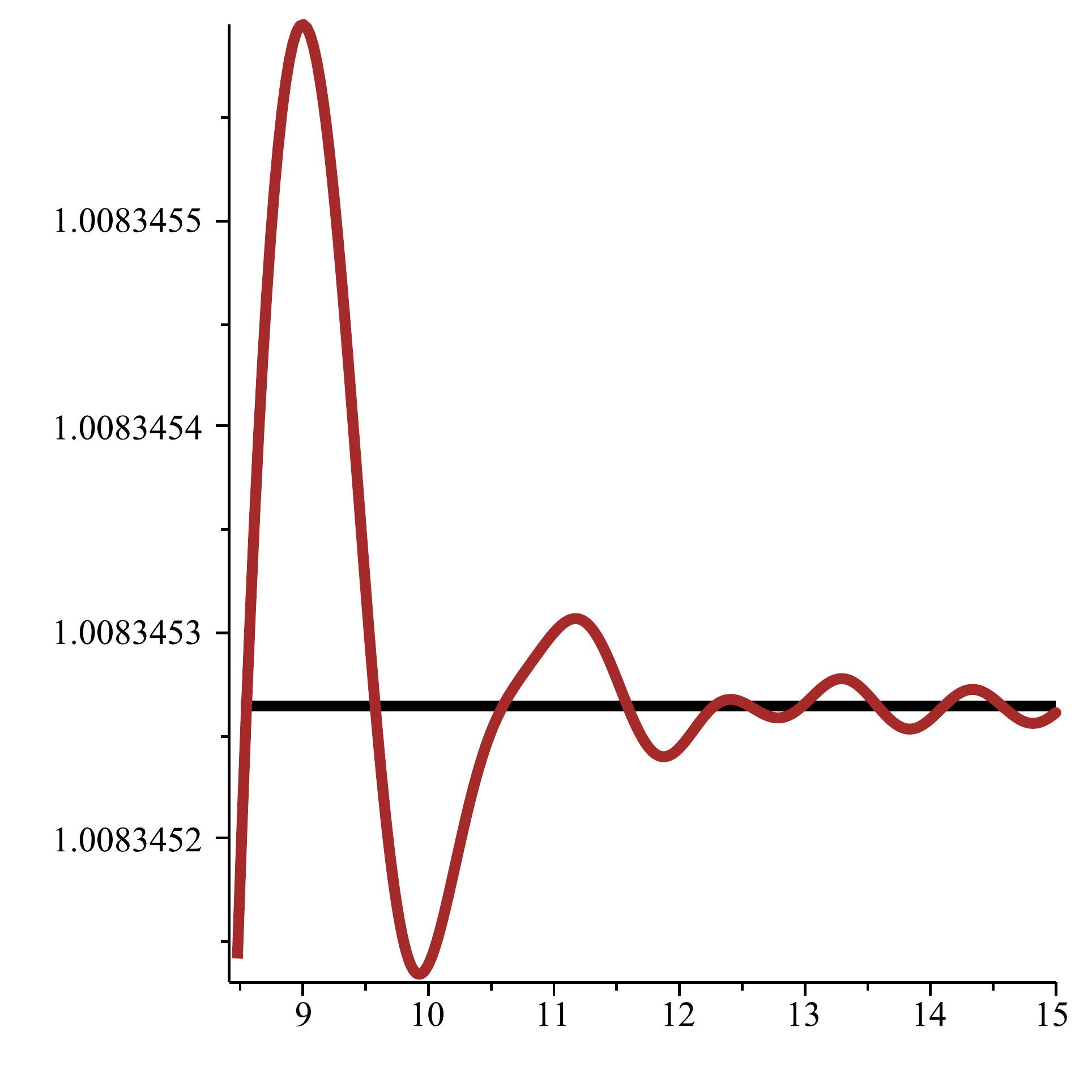}
\quad
\includegraphics[width=5cm,height=3.4cm]{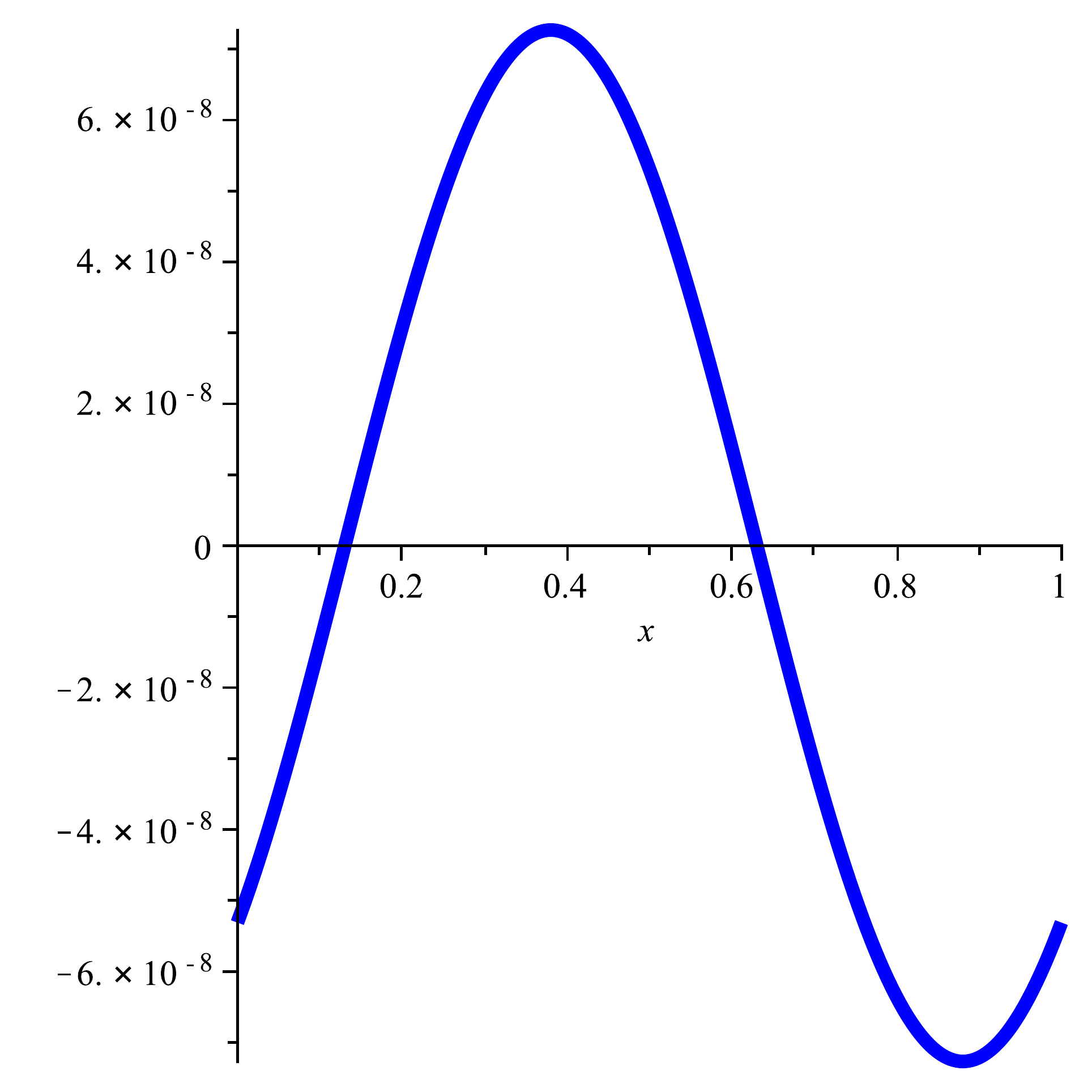}	
\end{center}
\caption{\emph{Periodic oscillations of the variance when 
$p=(\sqrt{5}-1)/2$: $\mathbb{V}(X_n)/n$ in logarithmic scale (left) 
and the fluctuating part $\mathscr{F}[G](x)$ 
(right).}} \label{fg-var-size-2}
\end{figure}	

\medskip

\noindent{\it Proof of Theorem \ref{thm-size-sym} and Theorem 
\ref{thm-size-asym}.} We look first at $\Phi_1(s)= \mathscr{M} 
[\tilde{V}_T+\tilde{\phi}_1; s]$. By \eqref{VT-size} and 
\eqref{s1-size}, we have
\[
    \Phi_1(s)=(s+1)\Gamma(s)
    \left(1-\frac{s^2+4s+8}
    {2^{s+3}}\right)+2Y_1(s),
\]
where
\begin{align}\label{Y1}
    Y_1(s):=\int_{0}^{\infty}z^{s-1}
    e^{-z}\left((1+z)\left(\tilde{f}_1(pz)+
    \tilde{f}_1(qz)\right)-z^2\left(p\tilde{f}'_1(pz)
    +q\tilde{f}'_1(qz)\right)\right) \dd z.
\end{align}
To simplify this integral, we use the inverse Mellin integral
\[
    \tilde{f}_1(z)=-\frac{1}{2\pi i}
    \int_{(-\frac32)}\frac{(w+1)
    \Gamma(w)}{1-p^{-w}-q^{-w}}\,z^{-w}\dd w,
\]
which, by taking derivative with respect to $z$,
\begin{equation}\label{sizedifmean}
    \tilde{f}'_1(z)=\frac{1}{2\pi i}
    \int_{(-\frac32)}\frac{\Gamma(w+2)}
    {1-p^{-w}-q^{-w}}\,z^{-w-1}\dd w.
\end{equation}
Substituting these into \eqref{Y1} yields
\begin{align*}
    Y_1(s)&=-\frac{1}{2\pi i}\int_{(-\frac32)}
    \frac{(w+1)((1+w)(s-w)+1)\Gamma(w)
    \Gamma(s-w)(p^{-w}+q^{-w})}
    {1-p^{-w}-q^{-w}} \dd w\\
    &=\sum_{j\ge 1}\frac{(-1)^jj(j(j+s+1)-1)
    \Gamma(j+s+1)(p^{j+1}+q^{j+1})}
    {(j+1)!(1-p^{j+1}-q^{j+1})},
\end{align*}
where the last expression is obtained by shifting the line of
integration to the left and by collecting all the residues
encountered.

In particular, when $p=1/2$,
\begin{align}\label{G1-size}
\begin{split}
    \Phi_1(s) &=(s+1)\Gamma(s)
    \left(1-\frac{s^2+4s+8}
    {2^{s+3}}\right) \\ &\qquad+2\sum_{j\ge1}
    \frac{(-1)^jj(j(j+s+1)-1)\Gamma(j+s+1)}
    {(j+1)!(2^j-1)},
\end{split}
\end{align}
which proves \eqref{v-ss} and \eqref{v-ssk}.

We now consider $\Phi_2(s)=\mathscr{M}[\tilde{\phi}_2;s]$
(see \eqref{s1-s2}). By (\ref{Phi-2-s}),
\begin{equation}\label{Phi-2-s-size}
    \Phi_2(s)=\frac{pq}{2\pi i}\int_{(0)}
    \frac{\Gamma(w+1)\left(p^{-w}-q^{-w}\right)}
    {1-p^{1-w}-q^{1-w}}\cdot
    \frac{\Gamma(s-w+2)
    \left(p^{w-s-1}-q^{w-s-1}\right)}
    {1-p^{w-s}-q^{w-s}}
    \dd w,
\end{equation}
where we shifted the line of integration to the imaginary axis. 
Since the above function is also analytic on $\Re(s)= -1$, 
we can substitute $s=-1+\chi_k$. The expressions 
\eqref{var-size-G1} and \eqref{var-size} are then obtained by
the simplification procedure that we used to derive \eqref{Phi-1-chi}
(for $\Phi_2(-1+\chi_k)$) and \eqref{J} with $\tilde{g}_1(z) 
= 1-(1+z)e^{-z}$ and $G_1(s)= -(s+1)\Gamma(s)$. \qed

\medskip

An alternative way of simplifying $\Phi_2(-1+\chi_k)$ is to 
shift the line of integration of (\ref{Phi-2-s-size}) to the left
and collect all residues encountered. 
This then yields the somehow more complicated expression
\begin{align*}
    \Phi_2(-1+\chi_k)&=pq\sum_{j\ge 1}
    \frac{(-1)^{j-1}\Gamma(j+\chi_k+1)
    \left(p^j-q^j\right)\left(p^{-j}-q^{-j}\right)}
    {(j-1)!(1-p^{1+j}-q^{1+j})(1-p^{1-j}-q^{1-j})}\\
    &\qquad\;-\sum_{\omega_j}
    \frac{\Gamma(\omega_j+1)\Gamma(-\omega_j+\chi_k+1)}
    {p^{1+\omega_j}\log p+q^{1+\omega_j}\log q},
\end{align*}
where $\omega_j$ runs over all zeros of $1-p^{1+w}-q^{1+w}=0$
with $\Re(\omega_j)<0$, and we used the relation 
\[
    \frac{pq(p^{\omega_j}-q^{\omega_j})
    (p^{-\omega_j}-q^{-\omega_j})}
    {1-p^{1-\omega_j}-q^{1-\omega_j}}=1.
\]
Here the convergence of the second series follows from the 
exponential decay of Gamma function at $c\pm i\infty$
and the property that the zeros $\omega_j$ are isolated in nature
(and equally spaced along vertical lines when $\frac{\log p}{\log
q}\in\mathbb{Q}$); see \cite{PF055} for details. Then we apply the
same procedure for deriving \eqref{Phi-1-chi} to further 
simplify the second series. 

\subsection{External path length}

The cost of constructing tries is directly proportional to the
external path length, which is the sum of all the distances between
each external node (where keys are stored) to the root. For example,
the external path length of the trie showed in Figure~\ref{fg-trie}
equals $2 + 3 + 4\times 3 + 5\times 2 = 27$. Under the same
Bernoulli model, the external path length is a random variable,
still denoted by $X_n$, satisfying \eqref{Xn-rr} with $T_n=n$. This
implies that the Poisson generating functions of the first two
moments of $T_n$ are given by
\[
    \tilde{g}_1(z)=z(1-e^{-z})\qquad\text{and}
    \qquad\tilde{g}_2(z)=z(1+z-e^{-z}).
\]
Thus JS-admissibility of these two functions follows directly from
Proposition~\ref{prop-closure}. Also $\tilde{g}_1(z) =z+O\left(\vert
z\vert^{-\delta}\right)$ uniformly as $|z|\to\infty$ and
$\vert\arg(z)\vert<\pi/2-\ve$ for all $\ve,\delta>0$. Moreover,
$\tilde{V}_T(z) =ze^{-z}(1-e^{-z}(1 -z+z^2))$ and
\[
    \tilde{\phi}_1(z)=e^{-z}\left(2z\tilde{f}_1(pz)
    +2z\tilde{f}_1(qz)+2pz(1-z)\tilde{f}'_1(pz)
    +2qz(1-z)\tilde{f}'_1(qz)\right),
\]
both being again exponential small.

For the expected value, we have $G_1(s) := \mathscr{M}[\tilde{g}_1
;s] =-\Gamma(s+1)$ and thus
\[
    \lim_{s\rightarrow-1}
    \left(G_1(s)+\frac1{s+1}\right)=\gamma,
\]
where $\gamma$ is Euler's constant. Applying Theorem \ref{mean}, we
obtain
\begin{align*}
    \frac{\mathbb{E}(X_n)}{n} &=
    \frac{1}{h}\log n+\frac{\gamma}{h}
    +\frac{p\log^2p+q\log^2 q}{2h^2}
    +\mathscr{G_1}[G_1](r\log_{1/p}n)
    +o(1).
\end{align*}
While this result has been widely known and discussed (see, for
example, \cite{knuth98a,jacquet86a,mahmoud92a}), the variance
is rarely addressed (see \cite{jacquet86a,jacquet88b,kirschenhofer89b})
due partly to its complexity and partly to methodological limitations.

\begin{thm} The variance of the total external path length satisfies
\[
    \frac{\mathbb{V}(X_n)}{n}
    =\frac{1}{\log 2}\sum_{k\in\mathbb{Z}}
    \Phi_1(-1+\chi_k)n^{-\chi_k} +o(1),
\]
in the symmetric case (when $p=1/2$), and
\begin{align*}
    \frac{\mathbb{V}(X_n)}{n} &=\frac{pq\log^2(p/q)}{h^3}\log n
    +\frac{d}{h}+\frac{pq\log^2(p/q)(p\log^2p+q\log^2 q)}{2h^4}
    \\ & \qquad +\mathscr{F}[G](r\log_{1/p}n)
    +o(1),
\end{align*}
in the asymmetric case, where $G=\Phi_1+\Phi_2$ with $\Phi_1$,
$\Phi_2$ and $d$ given below in \eqref{EPL-G1}, \eqref{EPL-H1}, and
\eqref{EPL-d}, respectively.
\end{thm}
The proof follows the same pattern as that used for the size,
details being omitted here. In particular, we have
\begin{align} \label{EPL-G1}
\begin{split}
    \Phi_1(s) &= \Gamma(s+1)
    \left(1-\frac{s^2+s+4}{2^{s+3}}\right)\\
    &\qquad +2\sum_{j\ge 1}\frac{(-1)^j
    (j(s+j)-1)(p^{j+1}+q^{j+1})
    \Gamma(s+j+1)}{j!(1-p^{j+1}-q^{j+1})},
\end{split}
\end{align}
where
\begin{align}\label{EPL-d}
    d=\Phi_1(-1)+pq\frac{\log^2(p/q)}{h^2}
    \left(\gamma+1+\frac{p\log^2p+q\log^2q}{2h}
    +\frac{\log p+\log q}{2}\right)+I_1(-1),
\end{align}
and for $k\ne 0$,
\begin{align} \label{EPL-H1}
    \Phi_2(-1+\chi_k)=pq\frac{\log^2(p/q)}{h^2}
    (\chi_k-1)\Gamma(\chi_k)+I_1(-1+\chi_k),
\end{align}
where
\begin{align*}
	I_1(-1)&=\frac{1}{4}-\log 2+\frac{\pi^2}{6h}-\frac{1}{h}
	+\frac{p\log^3p+q\log^3q}{6h^2}
	+\frac{(p\log^2p+q\log^2q)^2}{4h^3}\\
	&\qquad-2\sum_{j\geq 1}\frac{(-1)^j(j^2-1)(p^{j+1}+q^{j+1})}
	{j(1-p^{j+1}-q^{j+1})}\\
	&\qquad+\begin{cases}
	{\displaystyle\frac{1}{h}\sum_{j\ne 0}
	(\chi_j^2-1)\Gamma(\chi_j)\Gamma(-\chi_j)},
	&\text{if } \frac{\log p}{\log q}\in{\mathbb Q}\\0,
	&\text{if } \frac{\log p}{\log q}\not\in{\mathbb Q}
	\end{cases}.
\end{align*}
and for $k\ne 0$,
\begin{align*}
	I_1(-1+\chi_k)
	&=\Gamma(\chi_k)\left(\chi_k-1+
	\frac{\chi_k^2-3\chi_k+4}{2^{2+\chi_k}}\right)
	+\frac{2\Gamma(\chi_k)}{h}
	\left((1-\chi_k)
	(\psi(\chi_k+\gamma)-\chi_k\right)\\
	&\qquad-\frac{1}{h}
	\sum_{j\ne 0,k}(\chi_j-1)\Gamma(\chi_j)
	(\chi_{k-j}-1)\Gamma(\chi_{k-j})\\
	&\qquad +2\sum_{j\geq 1}\frac{(-1)^{j-1}(j+1)(\chi_k+j-1)
	\Gamma(\chi_k+j)(p^{j+1}+q^{j+1})}{j!(1-p^{j+1}-q^{j+1})}.
\end{align*}
The Fourier series is new.

In particular, in the symmetric case the Fourier coefficients of the
periodic function are given by
\[
    \frac{\Phi_1(-1)}{\log 2}
    =\frac{1}{\log 2}\left(\frac{1}{4}+\log 2 +2\sum_{j\ge
    1}\frac{(-1)^j(j^2-j-1)}{j(2^j-1)}\right) \approx
    4.352906698945400\cdots,
\]
and for $k\ne0$
\[
\begin{split}
    \frac{\Phi_1(-1+\chi_k)}{\log 2} &= \frac{(1-\chi_k)
    \Gamma(\chi_k+1)}4
    +2\sum_{j\ge 1}\frac{(-1)^j
    (j(j-1+\chi_k)-1)\Gamma(j+\chi_k)}
    {j!(2^j-1)}.
\end{split}
\]
The above numerical value for $\Phi_1(-1)/\log 2$ is in accordance
with that obtained in \cite{kirschenhofer89b} where the authors
derived the alternative expression
\[
    1+\frac{1}{2\log 2}-\frac{1}{\log^2 2}
    -\frac{2}{\log 2}\sum_{j\ge 1}
    \frac{(-1)^{j}(j+1)}{j(2^j-1)}
    -\frac{4\pi^2}{\log^3 2}\sum_{j\ge 1}
    \frac{j}{\sinh\frac{2j\pi^2}{\log 2}}.
\]
Equating them gives the same identity (\ref{id}) as we encountered 
in the size of tries.

\subsection{Radix sort}

Bucketing is a common design paradigm used for sorting or selecting
elements with specified properties; see \cite{devroye86a}. For
sorting purposes, a simple procedure, called \emph{radix sort}, is
to distribute elements into $b$ buckets according to their values
and then sort within each bucket recursively; see \cite{knuth98a,PF159}.
Since we can always
normalize elements into the unit interval, splitting into $b$
buckets amounts to using $b$-ary digit expansion of each element and
then distribute according to the leading digits. Thus the radix
sorting process induces a trie with up to $b$ branches at each node.

If we assume that the $n$ elements to be sorted are independent and
identically distributed uniform random variables (from the unit
interval), then the cost $X_n$ of radix sort (number of digit
extractions needed to sort) satisfies (see \cite{PF159})
\[
    \tilde{P}(z,y) = (1-e^y)ze^{-z}
    +e^{-(1-e^y)z} \tilde{P}\left(\frac{e^y z}b,y\right)^b,
\]
where $\tilde{P}(z,y) := e^{-z}\sum_{n\ge0} \mathbb{E}
(e^{X_ny})z^n/n!$. This is nothing but the Poisson generating
function for the external path length of random bucket tries with
branching factor $b$ (using $b$-ary expansion). All analysis above
carries through and we have
\[
    \tilde{f}_1(z) = b\tilde{f}_1(z/b) + z(1-e^{-z}),
\]
and the corresponding $\tilde{V}(z):= \tilde{f}_2(z)- \tilde{f}_1(z)
- z\tilde{f}_1'(z)$ satisfies
\[
    \tilde{V}(z) = b\tilde{V}(z/b) + \tilde{g}(z),
\]
where
\[
    \tilde{g}(z) := e^{-z} \left( 2bz\tilde{f}_1(z/b)+
    2z(1-z)\tilde{f}_1'(z/b)+z(1-e^{-z}) + z^2e^{-z}(1-z)
    \right).
\]
Then the Mellin transform of $\tilde{g}$ is given by
\begin{align} \label{radix-Gs}
\begin{split}
    G(s) &= \Gamma(s+1)\left(1-2^{-s-1} - s2^{-s-3}
    -s^22^{-s-3}\right)\\ &\qquad +2
    \sum_{k\ge1}  \frac{(-1)^k\Gamma(s+k+1)}
    {k!(b^k-1)}(k(s+k)-1)\qquad(\Re(s)>-2).
\end{split}
\end{align}
It follows, by the same Mellin analysis and JS-admissibility, that
($\chi_k := 2k\pi i/\log b$)
\begin{align*}
    \frac{\mathbb{E}(X_n)}{n} &=
    \log_bn +\frac{\gamma}{\log b} +\frac12 +
    \frac1{\log b}\sum_{k\in\mathbb{Z}\setminus\{0\}}
    \Gamma(\chi_k) n^{-\chi_k}+o(1)\\
    \frac{\mathbb{V}(X_n)}{n} &= \frac1{\log b}
    \sum_{k\in\mathbb{Z}}
    G\left(-1+\chi_k\right) n^{-\chi_k} + o(1).
\end{align*}
An expression for $G(s)$ was derived in \cite[p.\ 755]{PF159},
which is more messy than \eqref{radix-Gs}. Indeed, one can
simplify that expression and obtain
\begin{align*}
    \frac{G(s)}{\Gamma(s+1)} &= 1-2^{-s-1} - s2^{-s-3}
    -s^22^{-s-3} \\ &\qquad
    - 2(s+1)(s+2)U(s+3)+2(s+1)U(s+2)+2V(s+1),
\end{align*}
where ($U(s) = V(s-1)-V(s)$)
\[
    U(s) := \sum_{k\ge1} b^{-k}(1+b^{-k})^{-s},
    \quad V(s) := \sum_{k\ge1}\left(1-(1+b^{-k})^{-s}\right).
\]
Such an expression for $G(s)$ is on the other hand also easily
obtained from \eqref{radix-Gs} by binomial theorem.

In particular, by \eqref{radix-Gs},
\[
    G(-1) = \frac14+\log 2 + 2\sum_{k\ge1}
    \left(\left(b^k+1\right)^{-2}
    + \log\left(1+b^{-k}\right)\right).
\]
From this we obtain the following numerical table.
\begin{center}
\begin{tabular}{|c|c|}\hline
    $b$ & $G(-1)/\log b\approx$ \\ \hline\hline
    $2$ & $4.35290\,66989\,45400\,60374\,$ \\ \hline
    $3$ & $1.80839\,11899\,92781\,96720\,$ \\ \hline
    $4$ & $1.18266\,25542\,39848\,25415\,$ \\ \hline
    $5$ & $0.91013\,81377\,49170\,45524\,$ \\ \hline
    $6$ & $0.75883\,87760\,90906\,35697\,$ \\ \hline
    $7$ & $0.66265\,99366\,11117\,50882\,$ \\ \hline
    $8$ & $0.59600\,35264\,60033\,23615\,$ \\ \hline
    $9$ & $0.54696\,00912\,93530\,34188\,$ \\ \hline
    $10$ & $0.50926\,08387\,26247\,61651\,$ \\ \hline
\end{tabular}
\end{center}
Thus \emph{increasing the number of buckets in radix sort reduces
the variance of the cost}, with the most drastic change from $2$ to $3$.

\subsection{Peripheral path length}
We define the \emph{peripheral path length} of a tree as the sum of
the fringe-sizes of all leaf-nodes, where the fringe-size of a leaf
is defined to be the number of external nodes of the
subtree rooted at its parent-node. This parameter was investigated
in \cite{drmota09a} where it was called the $w$-parameter. It was
also studied in phylogenetics in the context of sum of all minimal
clade sizes (see \cite{blum05a}).

If we define $T_n$
\[
    (T_n\vert I_n=k)=
    \left\{\begin{array}{ll} n-1,
    &\text{if}\ k=1\ \text{or}\ k=n-1
    \\ 0,&\text{otherwise},
    \end{array}\right.
\]
for $n\ge 3$ and with $n-1$ replaced by $2$ for $n=2$, then the
peripheral path length $X_n$ of random tries of $n$ keys
satisfies \eqref{Xn-rr} with the initial conditions $X_0=0$ and
$X_1=1$.

Since $T_n$ depends on $I_n$, such a parameter does not directly fit
in our schemes. However, the same approach applies. The moment generating
function of $X_n$ then has the recursive form
\[
    M_n(y) = \sum_{k\not=1,n-1}\pi_{n,k}M_k(y)M_{n-k}(y)
    + e^{(n-1)y}n\left(pq^{n-1}+qp^{n-1}\right)M_{n-1}(y),
\]
for $n\ge2$ with $M_0(y)=1$ and $M_1(y)=e^y$. It follows that
\begin{align*}
    \tilde{g}_1(z)&=pqz^2(e^{-pz}+e^{-qz})\\
    \tilde{g}_2(z)&=pqz^2(2e^{-z}+(1+qz)e^{-pz}+(1+pz)e^{-qz}),
\end{align*}
which are both exponentially small and JS-admissible by Proposition
\ref{prop-closure}. The function $\tilde{h}_2(z)$ (see \eqref{h2z})
is now given by
\begin{align*}
    \tilde{h}_2(z) &=2e^{-z}\sum_{n\ge 2}\sum_{0\le j\le
    n}\pi_{n,k}\left(\mu_k+ \mu_{n-k}\right)
    \mathbb{E}(T_n\vert I_n=k)\frac{z^n}{n!}\\
    &=pqz^2\Bigl(2e^{-pz}\tilde{f}_1(qz)
    +2e^{-qz}\tilde{f}_1(pz)
    +2e^{-pz}\tilde{f}'_1(qz)+2e^{-qz}\tilde{f}'_1(pz)
    \\ &\qquad\qquad  +2e^{-pz}+2e^{-qz}\Bigr),
\end{align*}
which is also JS-admissible by Proposition \ref{prop-closure}. This
gives rise to the following expression for $\tilde{\phi}_1(z)$ (see
\eqref{s1-s2})
\begin{align*}
    \tilde{\phi}_1(z)&=2pqz^2\Big(-e^{-pz}
    \tilde{f}_1(pz)-e^{-qz}\tilde{f}_1(qz)
    +e^{-pz}+e^{-qz} \\
    &\qquad\qquad +(1-2p+pqz)e^{-qz}
    \tilde{f}'_1(pz)+(1-2q+pqz)e^{-pz}\tilde{f}'_1(qz)\\
    &\qquad\qquad-(2p-p^2z)e^{-pz}\tilde{f}'_1(pz)
    -(2q-q^2z)e^{-qz}\tilde{f}'(qz)\Big).
\end{align*}
Finally,
\begin{align*}
    \tilde{V}_T(z)&=pqz^2
    \Big(2(1-4pqz+pqz^2-p^2q^2z^3)e^{-z}+(1+qz)e^{-pz}
    +(1+pz)e^{-qz}\\
    &\qquad-pqz(4+z-4pz+p^2z^2)e^{-2pz}
    -pqz(4+z-4qz+q^2z^2)e^{-2qz}\Big).
\end{align*}
All these functions are exponentially small for large $|z|$ with
$\Re(z)>0$.

Observe that $G_1(s):= \mathscr{M}[\tilde{g}_1;s] =pq(p^{-s-2}
+q^{-s-2})\Gamma(s+2)$. An application of Theorem \ref{mean} then
gives
\[
    \frac{\mathbb{E}(X_n)}{n}=
    1+\frac{1}{h} + \mathscr{F}[G_1](r\log_{1/p}n)
    +o(1),
\]
where the additional term $1$ on the right-hand side arises from the
initial condition.

Although Theorem \ref{var} does not apply directly to the variance
of $X_n$, the same method of proof works well as in Theorem
\ref{var} part (\emph{b})-(\emph{i}), and we obtain
\[
    \frac{\mathbb{V}(X_n)}{n}=
    \frac{G(-1)}{h} + \mathscr{F}[G](r\log_{1/p}n)
    +o(1),
\]
where a series-form for $G(s)$ can be derived as the discussions
above. For simplicity, we have, in the symmetric case,
\begin{align*}
    G(s)&=s(s+1)\Gamma(s)
    \left(2^{s+1}(s+3)-
    \frac{s^3+5s^2+22s+24}{16}\right)\\
    &\qquad-2^{s+2}\sum_{j\ge 1}\frac{(-1)^j
    \Gamma(s+j+2)}{(j-1)!(2^j-1)}(j(s+j+2)-j-1).
\end{align*}
In particular, the average value of the periodic function is given
by
\[
    \frac{13}{8}-2\sum_{j\ge 1}\frac{(-1)^j
    j(j^2-1)}{2^j-1}
    = \frac{13}8-12\sum_{j\ge1}
    \frac{1}{4^j(1+2^{-j})^4}
    \approx 0.55730\,49532\,49505\cdots.
\]
Note that we can also derive the identity  
\[
    2\sum_{j\ge 1}\frac{(-1)^j
    j(j^2-1)}{2^j-1} = \frac1{\log 2}-\frac38
    + \frac{4\pi^2}{(\log 2)^4}\sum_{k\ge1}
    \frac{k((2k\pi)^2+(\log 2)^2)}{\sinh\frac{2k\pi^2}{\log 2}},
\]
the series on the right-hand side being smaller than $6\times 10^{-9}$.

\subsection{Leader election (or loser selection)}

The coin-flipping process is applicable to single out a leader in
real life or in abstract models: every individual involved throws a
coin and those who get head continue until only one is left; see
\cite{prodinger93a}. In this case, the approach we use so far leads
to extremely simple forms for the number of coin-flippings; this
example thus has a more instructional value. Let $X_n$ denote the
total number of coin flippings used in the leader election procedure
of $n$ people. Then $X_0=X_1=0$ and the exponential generating
function $P(z,y) := \sum_{n\ge0} \mathbb{E}(e^{X_ny})z^n/n!$
satisfies
\[
    P(z,y) = \left(e^{yz/2}+1\right) P
    \left(\frac {yz}2,y\right) -e^{yz/2} +(1-y)z.
\]

Instead of the usual Poisson generating function, we consider,
as in \cite{prodinger93a}, the Bernoulli generating function
\[
    \tilde{f}_m(z) := \frac{1}{e^z-1}\sum_{n\ge0}
    \frac{\mathbb{E}(X_n^m)}{n!}\,z^n.
\]
Then $\tilde{f}_1(0)=0$ and
\[
    \tilde{f}_1(z) = \tilde{f}_1(z/2) + z,
\]
which gives the identity $\tilde{f}_1(z)=2z$. Thus $\mathbb{E}(X_n)
\equiv 2n$ for $n\ge2$. Also the normalized function $\tilde{V} :=
\tilde{f}_2-\tilde{f}_1^2 - z(\tilde{f}_1')^2$ satisfies
\[
    \tilde{V}(z) = \tilde{V}(z/2)+z+\frac{3z^2}{e^z-1}.
\]
Standard Mellin analysis yields
\[
    \tilde{V}(z) = 2z+\frac{\pi^2}{2\log 2}
    +\frac{3}{\log 2}\sum_{k\in\mathbb{Z}\setminus\{0\}}
    \zeta(2+\chi_k)\Gamma(2+\chi_k)z^{-\chi_k}
    +O(|z|^{-1}),
\]
as $|z|\to\infty$ in the half-plane $\Re(z)>0$, where $\zeta(s)$
denotes Riemann's zeta function. Consequently, a similar
de-Poissonization argument leads to
\[
    \sigma_n^2 = 2n +\frac{\pi^2}{2\log 2}
    +\frac{3}{\log 2}\sum_{k\in\mathbb{Z}\setminus\{0\}}
    \zeta(2+\chi_k)\Gamma(2+\chi_k)n^{-\chi_k} + O(n^{-1}).
\]
Thus, with an average of $2n$ coin-tossings and a $\sqrt{n}$-order
of standard deviation, selecting a leader or a loser by such a naive
splitting process is a very efficient procedure.

\section{Further extensions}
\label{sec-more}

Since BSPs appear in a large number of diverse contexts, many
extensions of our frameworks are possible. We briefly discuss some
examples in this section.

\subsection{Internal path length of random tries}

If, instead of summing over all the distances between the root and
each external node (where records are stored), we add up all the
distance between the root and each internal node, then we have the
system of recurrences for the number of internal nodes $N_n$
(already discussed in Section \ref{sec-size}) and the internal path
length $X_n$ in a random trie of $n$ elements under the Bernoulli
model
\begin{align*}
    \left\{\begin{array}{l}
    N_n\stackrel{d}{=}N_{I_n}+N^*_{n-I_n}+1\\
    X_n\stackrel{d}{=}X_{I_n}+X^*_{n-I_n}+N_{I_n}+N^{*}_{n-I_n},
    \end{array}\right.
\end{align*}
for $n\ge2$, with initial conditions $N_0=N_1=I_0=I_1=0$, where
$N_n^*$ and $X_n^*$ are independent copies of $N_n$ and $X_n$,
respectively.

We will see that the variance changes completely its asymptotic
behavior and is asymptotic to $n(\log n)^2$ weighted by a periodic
function. This estimate is independent of the rationality of
$\frac{\log p}{\log q}$. This was previously observed in
\cite{nguyen-the04a} but with incomplete proof; see also the
recent paper \cite{fuchs12c} for a study of Wiener index.

The asymptotics of the variance can be addressed by the same
approach we used for the node-wise path length of random digital
search trees in \cite{hwang10a}. We begin with the moment generating
function $M_n(u,v)=\mathbb{E}(e^{N_nu+X_nv})$, which satisfies the
recurrence
\[
    M_n(u,v)=e^{u}\sum_{0\le k\le
    n}\pi_{n,k}M_k(u+v,v)M_{n-k}(u+v,v)\qquad (n\ge 2).
\]
We then deduce that the Poisson generating functions of
$\mathbb{E}(N_n)$ and $\mathbb{E}(X_n)$, denoted by
$\tilde{f}_{1,0}(z)$ and $\tilde{f}_{0,1}(z)$, respectively, satisfy
the functional equations
\begin{align*}
    \tilde{f}_{1,0}(z)&=\tilde{f}_{1,0}(pz)
    +\tilde{f}_{1,0}(qz)+1-(1+z)e^{-z}\\
    \tilde{f}_{0,1}(z)&=\tilde{f}_{0,1}(pz)
    +\tilde{f}_{0,1}(qz)+\tilde{f}_{1,0}(pz)+\tilde{f}_{1,0}(qz).
\end{align*}
Let $\tilde{f}_{2,0}(z),\tilde{f}_{1,1}(z)$ and $\tilde{f}_{0,2}(z)$
denote the Poisson generating function of $\mathbb{E}(N_n^2),
\mathbb{E}(N_nX_n)$ and $\mathbb{E}(X_n^2)$, respectively. Then we
define the Poissonized versions of the variance and the covariance
as
\begin{align*}
    \tilde{V}(z)&:=\tilde{f}_{2,0}(z)
    -\tilde{f}_{1,0}(z)^2-z\tilde{f}'_{1,0}(z)^2\\
    \tilde{C}(z)&:=\tilde{f}_{1,1}(z)
    -\tilde{f}_{1,0}(z)\tilde{f}_{0,1}(z)
    -z\tilde{f}'_{1,0}(z)\tilde{f}'_{0,1}(z)\\
    \tilde{W}(z)&:=\tilde{f}_{0,2}(z)
    -\tilde{f}_{0,1}(z)^2-z\tilde{f}'_{0,1}(z)^2.
\end{align*}
A lengthy calculation then gives
\begin{align*}
    \tilde{V}(z)&=\tilde{V}(pz)+\tilde{V}(qz)
    +\tilde{g}_{2,0}(z)\\
    \tilde{C}(z)&=\tilde{C}(pz)+\tilde{C}(qz)
    +\tilde{V}(pz)+\tilde{V}(qz)+\tilde{g}_{1,1}(z)\\
    \tilde{W}(z)&=\tilde{W}(pz)+\tilde{W}(qz)
    +2\tilde{C}(pz)+2\tilde{C}(qz)+\tilde{V}(pz)
    +\tilde{V}(qz)+\tilde{g}_{0,2}(z),
\end{align*}
where
\begin{align*}
    \tilde{g}_{2,0}(z)&:=e^{-z}\Big\{2(1+z)\left(
    \tilde{f}_{1,0}(pz)+\tilde{f}_{1,0}(qz)\right)
    -2z^2\left(p\tilde{f}'_{1,0}(pz)
    +q\tilde{f}'_{1,0}(qz)\right)\\
    &\qquad\quad+1+z-(1+2z+z^2+z^3)e^{-z}\Big\}+pqz
    \left(\tilde{f}'_{1,0}(pz)-\tilde{f}'_{1,0}(qz)\right)^2
\end{align*}
and
\begin{align*}
    \tilde{g}_{1,1}(z)&:=e^{-z}\Big\{(1+z)\left(\tilde{f}_{1,0}(pz)
    +\tilde{f}_{1,0}(qz)+\tilde{f}_{0,1}(pz)
    +\tilde{f}_{0,1}(qz)\right)\\
    &\qquad-z^2\left(p\tilde{f}'_{1,0}(pz)+q\tilde{f}'_{1,0}(qz)
    +p\tilde{f}'_{0,1}(pz)+q\tilde{f}'_{0,1}(qz)\right)\Big\}\\
    &\qquad+pqz\left(\tilde{f}'_{1,0}(pz)
    -\tilde{f}'_{1,0}(qz)\right)\left(\tilde{f}'_{1,0}(pz)
    -\tilde{f}'_{1,0}(qz)+\tilde{f}'_{0,1}(pz)
    -\tilde{f}'_{0,1}(qz)\right)
\end{align*}
and
\[
    \tilde{g}_{0,2}(z)
    :=pqz\left(\tilde{f}'_{1,0}(pz)
    -\tilde{f}'_{1,0}(qz)+\tilde{f}'_{0,1}(pz)
    -\tilde{f}'_{0,1}(qz)\right)^2.
\]
Then we have
\begin{align*}
    \mathscr{M}[\tilde{f}_{1,0};s]
    &=-\frac{(s+1)\Gamma(s)}
    {1-p^{-s}-q^{-s}}\\
    \mathscr{M}[\tilde{f}_{0,1};s]
    &=-\frac{(s+1)\Gamma(s)(p^{-s}+q^{-s})}
    {(1-p^{-s}-q^{-s})^2}.
\end{align*}
It follows that (already derived in Section~\ref{sec-size})
\[
    \frac{\mathbb{E}(N_n)}{n} = \frac{1}{h}
    + \mathscr{F}[G_{1,0}](r\log_{1/p}n)
    +o(1),
\]
where $G_{1,0}(s) = -(s+1)\Gamma(s)$. Similarly,
\begin{align*}
    \frac{\mathbb{E}(X_n)}{n} &= \left(\frac{1}{h}
    +\mathscr{F}[G_{1,0}](r\log_{1/p}n)\right)
    \frac{\log n}h \\ &\qquad+
    \frac{p\log^2 p+q\log^2 q}{h^3}+\frac{\gamma-1}{h^2}-
    \frac1h+
    \frac1h\mathscr{F}[G_{0,1}](r\log_{1/p}n) +o(1),
\end{align*}
where ($\psi$ being the derivative of $\log\Gamma$)
\[
    G_{0,1}(s)= \Gamma(s)\left(\left(\psi(s)
    +h-\frac{p\log^2p+q\log^2 q}{h}\right)(1+s)
    +1\right).
\]

By the same Mellin analysis, we obtain
\begin{align*}
    \mathscr{M}[\tilde{V};s]
    &=\frac{\Phi_1(s)+\Phi_2(s)}{1-p^{-s}-q^{-s}}\\
    \mathscr{M}[\tilde{C};s]
    &=\frac{1}{(1-p^{-s}-q^{-s})^2}
    \Big((p^{-s}+q^{-s})
    (\Phi_1(s)+\Phi_2(s))\\
    &\qquad+(1-p^{-s}-q^{-s})(G_2(s)+H_2(s))\Big)\\
    \mathscr{M}[\tilde{W};s]
    &=\frac{1}{(1-p^{-s}-q^{-s})^3}
    \Big((p^{-s}+q^{-s})(1+p^{-s}+q^{-s})(\Phi_1(s)+\Phi_2(s))\\
    &\qquad+2(p^{-s}+q^{-s})(1-p^{-s}-q^{-s})(G_2(s)+H_2(s))
    \\&\qquad+(1-p^{-s}-q^{-s})^2H_3(s)\Big),
\end{align*}
where
\begin{align*}
    \Phi_1(s)&=\mathscr{M}\left[\tilde{g}_{2,0}(z)
    -pqz\left(\tilde{f}'_{1,0}(pz)-\tilde{f}'_{1,0}(qz)\right)^2;
    s\right]\\
    G_2(s)&=\mathscr{M}\left[\tilde{g}_{1,1}(z)
    -pqz\left(\tilde{f}'_{1,0}(pz)-\tilde{f}'_{1,0}(qz)
    +\tilde{f}'_{0,1}(pz)-\tilde{f}'_{0,1}(qz)\right);
    s\right].
\end{align*}
and
\begin{align*}
    \Phi_2(s)&=\mathscr{M}\left[pqz\left(\tilde{f}'_{1,0}(pz)
    -\tilde{f}'_{1,0}(qz)\right)^2;s\right]\\
    H_2(s)&=\mathscr{M}\left[pqz\left(\tilde{f}'_{1,0}(pz)
    -\tilde{f}'_{1,0}(qz)\right)\left(
    \tilde{f}'_{0,1}(pz)-\tilde{f}'_{0,1}(qz)\right);s\right]\\
    H_3(s)&=\mathscr{M}\left[pqz\left(\tilde{f}'_{0,1}(pz)
    -\tilde{f}'_{0,1}(qz)\right)^2;s\right].
\end{align*}
From these functions, we can derive, by the same arguments we used
above, asymptotic approximations to the covariance of $N_n$ and
$X_n$, and the variance of $X_n$.

\begin{thm} The variance of the internal path length of random tries
satisfies
\begin{align*}
    \frac{\mathbb{V}(X_n)}{n} &= F_{0,2}(r\log_{1/p}n)
    \frac{(\log n)^2}{h^2}+
    + F_{0,2}^{[2]}(r\log_{1/p}n)
    \frac{\log n}h + F_{0,2}^{[3]}(r\log_{1/p}n)
    +o(1),
\end{align*}
and the covariance of $N_n$ and $X_n$ satisfies
\begin{align*}
    \frac{\mathrm{Cov}(N_n,X_n)}{n} &=
    F_{0,2}(r\log_{1/p}n)
    \frac{\log n}h +F_{1,1}^{[2]}(r\log_{1/p}n)
    +o(1),
\end{align*}
where $F_{0,2}(x)=G(-1)/h+\mathscr{F}[G](x)$ with $G$ given in
Section \ref{sec-size}, and the other $F_{\cdot,\cdot}^{[\cdot]}$'s
are either constants when $\frac{\log p}{\log q}\not\in\mathbb{Q}$
or periodic functions with computable Fourier series when
$\frac{\log p}{\log q}\in\mathbb{Q}$.
\end{thm}
For simplicity, we give only the expressions in the symmetric case
\begin{align*}
    F_{1,1}^{[2]}(x)&=-\frac{1}{(\log 2)^2}
    \sum_{k\in\mathbb{Z}}(G'_1(-1+\chi_k)
    -G_2(-1+\chi_k)\log 2)e^{2k\pi i x}\\
    F_{0,2}^{[2]}(x)&=-\frac{2}{(\log 2)^2}
    \sum_{k\in\mathbb{Z}}(G'_1(-1+\chi_k)
    -G_2(-1+\chi_k)\log 2)e^{2k\pi i x}\\
    F_{0,2}^{[3]}(x)&=\frac{1}{(\log 2)^3}
    \sum_{k\in\mathbb{Z}}(G''_1(-1+\chi_k)
    -2G'_2(-1+\chi_k)\log 2)e^{2k\pi x},
\end{align*}
where $G_1(s)$ is given in \eqref{G1-size} and
\[
    G_2(s)=\sum_{j\ge 1}\frac{(-1)^jj
    \Gamma(s+j+1)}{(j+1)!(2^j-1)^2}{(2^j+2)}(j(j+1+s)-1).
\]

An intuitive interpretation of why the variance is of order $n(\log
n)^2$ is as follows. Any path from the root of length $k$ to an
internal node contributes $1+2+\cdots+k=O(k^2)$ to the internal path
length. Since the expected values of both internal and external path
lengths are of order $n\log n$, we see that most nodes lie at levels
of order $\log n$, and these nodes thus contribute an order $n(\log
n)^2$ to the variance.

In a completely similar manner, if $Y_n$ denotes the peripheral path
length where we change subtree-size to the sum of all internal nodes
(instead of all external nodes), then we can derive the asymptotic
approximations to the variance of $Y_n$ and the covariance of $Y_n$
and $N_n$, which are both linear
\begin{align*}
    \frac{\mathbb{V}(Y_n)}{n}
    &= F_{0,2}^{[Y]}(r\log_{1/p}n)+o(1)\\
    \frac{\mathrm{Cov}(Y_n,N_n)}{n}
    &= F_{1,1}^{[Y]}(r\log_{1/p}n)+o(1),
\end{align*}
where the $F_{\cdot,\cdot}^{[Y]}$'s are either constants when
$\frac{\log p}{\log q}\not\in\mathbb{Q}$ or computable periodic
functions when $\frac{\log p}{\log q}\in\mathbb{Q}$.

\subsection{Contention resolution in multi-access channel
using tree algorithms}

There is an abundant literature on the subject and we are specially
interested in the complexity of tree algorithms used in resolving
the contention before either transmitting information to the common
shared channel or performing certain tasks in a distributed
computing environment. The tree algorithm (originally due to
Capetanakis, Tsybakov and Mikhailov in the late 1970s) resolves the
conflict (when more than one user is sending simultaneously his
message to the common channel) by the outcome of a coin-flipping at
each contender's site, similar to the splitting rule used for
constructing a trie; see
\cite{biglieri07a,massey81a,molle93a,PF054,wagner09a} for details. The
analysis of the time needed for such algorithms to resolve the
conflict of $n$ contenders often leads to recurrences of the form
\eqref{Xn-rr} or its extensions. The expected value of the time to
resolve all conflicts, which corresponds essentially to the size of
random tries, has been widely addressed in the information-theoretic
and communication literature, but there are very few papers on the
variance; see \cite{janssen00a,kaplan85a}.

Consider the extended environment where each ``coin'' has $r$
distinct outcomes with respective probabilities $p_1,\dots, p_r$,
where $\sum_{1\le m\le r}p_m=1$ and none of them is zero. Then the
time $X_n$ to resolve the collision of $n$ contenders satisfies (see
\cite{PF054})
\begin{align}\label{tPzy}
    \tilde{P}(z,y)
    =e^{y}\prod_{1\le m\le r}\tilde{P}(p_mz,y)
    +(1-e^y)(1+z)e^{-z},
\end{align}
where $\tilde{P}(z,y) := e^{-z}\sum_{n\ge0}
\mathbb{E}(e^{X_ny})z^n/n!$. For simplicity, we consider a version
with $X_0=X_1=0$; the situation of nonzero initial conditions can be
manipulated by extending the same arguments we use (only the mean
will be altered, the variance remains the same). From \eqref{tPzy},
we obtain the functional equation for the Poisson generating
function of $\mathbb{E}(X_n)$
\[
    \tilde{f}_1(z)=\sum_{1\le m\le r}
    \tilde{f}_1(p_mz)+1-(1+z)e^{-z},
\]
with $\tilde{f}_1(0)=0$, and, similarly, for $\tilde{V}:=
\tilde{f}_2-\tilde{f}_1^2-z(\tilde{f}_1')^2$,
\[
    \tilde{V}(z)=\sum_{1\le m\le r}\tilde{V}(p_mz)+\tilde{g}(z),
\]
with $\tilde{V}(0)=0$, where
\begin{align*}
    \tilde{g}(z)&=e^{-z}\left(1+z-(1+2z+z^2+z^3)e^{-z}\right)
    \\ &\qquad+2e^{-z}
    \Biggl((1+z)\sum_{1\le m\le r}\tilde{f}_1(p_mz)
    -z^2\sum_{1\le m\le r}p_m\tilde{f}'_1(p_mz)\Biggr)\\
    &\qquad+z\sum_{1\le m<l\le r}p_mp_l
    \left(\tilde{f}'_1(p_mz)-\tilde{f}'_1(p_l z)\right)^2.
\end{align*}

Then all our analysis extends \emph{mutatis mutandis} to these
functional equations, and we have the following asymptotics for
$\mathbb{E}(X_n)$ and $\mathbb{V}(X_n)$.

Let
\[
    P(s):=\sum_{1\le m\le r}p_m^s.
\]
Then the entropy is
\[
    h:=-P'(1) = -\sum_{1\le m\le r}p_m\log p_m.
\]
As in the Bernoulli case, we need to distinguish between rational
(periodic) and irrational (aperiodic) cases. The former is
characterized either by the existence of a $\rho\in\mathbb{R}$ such
that $p_m=\rho^{e_m}$, $e_m\in{\mathbb N}$ for $1\le m\le r$, or by
the ratios $\frac{\log p_m}{\log p_l}\in\mathbb{Q}$ for all pairs
$(m,l)$.
\begin{thm}
The expected value and the variance of $X_n$ (defined in
\eqref{tPzy}) can asymptotically be approximated by
\begin{align*}
    \frac{\mathbb{E}(X_n)}{n}
    &= \frac1h + F_1(\log_{1/\rho}n) + o(1),\\
    \frac{\mathbb{V}{X_n}}{n} &= \frac{G(-1)}{h}
    +F_2(\log_{1/\rho}n)+  o(1),
\end{align*}
where both $F_1=F_2=0$ in the irrational case and ($\chi_k=
\frac{2k\pi i}{\log \rho}$)
\begin{align*}
    F_1(x) &= \frac1h\sum_{k\in\mathbb{Z}\setminus\{0\}}
    \chi_k\Gamma(-1+\chi_k) e^{2k\pi i x},\\
    F_2(x) &= \frac1h\sum_{k\in\mathbb{Z}\setminus\{0\}}
    G(-1+\chi_k) n^{-\chi_k},
\end{align*}
in the rational case, where $G=\mathscr{M}[\tilde{g};s]$ is given in
\eqref{gGs} below.
\end{thm}
While the dominant term involving the entropy for the expected value
is well-known (see \cite{bourdon01a,PF161,janson12a}), the
corresponding term $G(-1)/h$ for the variance is far from being
intuitive. On the other hand, if we start with $X_0=a$ and $X_1=b$, 
then
\[
    \frac{\mathbb{E}(X_n)}{n}
    = (b-a)+ ((r-1)a+1)\left(
    \frac1h + F_1(\log_{1/\rho}n)\right) + o(1).
\]

The function $G$ in the Theorem is described as follows.
For $\Re(s)>-2$,
\begin{align}\label{gGs}
\begin{split}
    G(s)&=(s+1)\Gamma(s)
    \left(1-\frac{s^2+4s+8}{2^{s+3}}\right)\\
    &\qquad+2\sum_{j\ge 1}
    \frac{(-1)^jj(j(j+s+1)-1)\Gamma(j+s+1)P(j+1)}
    {(j+1)!(1-P(j+1))}+\Phi_2(s),
\end{split}
\end{align}
where $\Phi_2(s)\equiv 0$ if $p_m=1/r$ for $1\le m\le r$ (the
symmetric case), and
\begin{align*}
	\Phi_2(-1+\chi_k)&=\frac{\Gamma(2+\chi_{k})}
	{2^{2+\chi_k}}-2\sum_{j\ge 1}
	\frac{(-1)^j\Gamma(j+\chi_k+1)P(j+1)}{(j-1)!(1-P(j+1))}\\
	&\qquad-\begin{cases}{\displaystyle\frac{1}{h}
	\sum_{j\in{\mathbb Z}}\Gamma(\chi_j+1)
	\Gamma(\chi_{k-j}+1)},
	&\text{in the rational case}\\
	0,&\text{in the irrational case,}
	\end{cases}
\end{align*}
in the asymmetric case. Consequently,
\begin{align*}
	G(-1+\chi_k)&=\chi_k\Gamma(-1+\chi_k)
	\left(1-\frac{\chi_k+3}{2^{1+\chi_k}}\right)\\
	&\qquad -2\sum_{j\ge 1}\frac{(-1)^j(j+1+\chi_k)
	\Gamma(j+\chi_k)P(j+1)}{(j-1)!(j+1)
	\left(1-P(j+1)\right)}\\
	&\qquad -\begin{cases}{\displaystyle \frac{1}{h}
	\sum_{j\in{\mathbb Z}}\Gamma(\chi_j+1)\Gamma(\chi_{k-j}+1)},
	&\text{in the rational case};\\
	0,&\text{in the irrational case.}
	\end{cases}
\end{align*}
which reduces to \eqref{var-size-G1} and \eqref{var-size} in the 
Bernoulli model ($r=2$).

When $p_m=1/r$ for $1\le m\le r$
\begin{align*}
    G(s) &= (s+1)\Gamma(s) \left(1-\frac{
    s^2+4s+8}{2^{s+3}}\right) \\
    &\qquad+2\sum_{j\ge1}\frac{(-1)^j j\Gamma(j+1+s)}
    {(j+1)!(b^j-1)}\left(j(j+1+s)-1\right);
\end{align*}
compare \eqref{v-ss} and \eqref{v-ssk}. Thus the 
average value of the periodic function is given by 
\[
    \frac{G(-1)}{\log b} = \frac1{4\log b}
    +\frac2{\log b} \sum_{k\ge1} \frac{(-1)^k(k-1)}{b^k-1}
    = \frac1{4\log b}
    +\frac2{\log b} \sum_{k\ge1} \frac1{(b^k+1)^2}.
\]
This is consistent with the expression derived in \cite{janssen00a}
\[
    \frac1{2\log b} - \frac1{(\log b)^2}+
    \frac2{\log b}\sum_{k\ge1} \frac1{b^k+1}
    - \frac{4\pi^2}{(\log b)^3}\sum_{k\ge1}
    \frac{k}{\sinh\frac{2k\pi^2}{\log b}}.
\]
Equating the two expressions leads to the identity
\[
    \frac1{2} - \frac1{\log b}+2\sum_{k\ge1} \frac1{b^k+1}
    = \frac14+2\sum_{k\ge1} \frac{1}{(b^k+1)^2}
    +\frac{4\pi^2}{(\log b)^2}\sum_{k\ge1}
    \frac{k}{\sinh\frac{2k\pi^2}{\log 2}},
\]
which generalizes \eqref{id}. Our Fourier series for $F_2$ is new
even in this simple case.

For many other concrete examples, see \cite{biglieri07a,jacquet85a,
massey81a,PF054,molle93a,wagner09a} and the references therein.

\section{PATRICIA Tries}
\label{sec-patricia}

In typical random tries, internal nodes at successive levels may
have only one descendant (corresponding to the extreme probabilities
when binomial distribution assumes $0$ and $n$), resulting in an
increase in storage. Indeed, the expected number $\mu_n$ of internal
nodes under the initial condition $\mu_1=0$ is asymptotic to
$(h^{-1} +\mathscr{F}[G](r\log_{1/p}n))n$ (see
Section \ref{sec-size}). Thus the expected number of internal nodes
with only one child is asymptotic to
$(h^{-1}-1+\mathscr{F}[G](r\log_{1/p}n))n$. In the symmetric case, the
leading constant (neglecting the fluctuation term) is about $1/\log
2-1\approx .4427$, about $44\%$ extra space being needed, and this is
the minimum when $p$ varies between $0$ and $1$. The idea of
PATRICIA\footnote{PATRICIA is the acronym of ``practical algorithm
to retrieve information coded in alphanumeric".} tries arose when
there was a need to compress such a one-child-in-one-generation
pattern; see \cite{knuth98a,morrison68a}. When removing all such
nodes, the resulting tree has $n-1$ internal nodes (for $n$ external
nodes). See \cite{wagner10a} for an analysis connected to unary
nodes of random tries, and \cite{bourdon01a,devroye05a,janson12a,
kirschenhofer89a,schachinger95b} for other linear shape measures.

Under the same Bernoulli model, we can construct random PATRICIA
tries by using the same rule for constructing an ordinary trie but
compress all internal nodes with only one descendant. If $X_n$
represents an additive shape parameter in a random PATRICIA trie of
size $n$, then, for $n\ge 2$,
\begin{align}\label{Xn-PT}
    X_n\stackrel{d}{=}
    X_{I_n'}+X_{n-I_n'}^{*}+T_n,
\end{align}
where
\[
    \mathbb{P}(I_n'=k)=\pi_{n,k}'
    :=\frac{\binom{n}{k}p^kq^{n-k}}{1-p^n-q^n},
    \qquad(k=1,\dots,n-1),
\]
and the $X_n^*$'s are independent copies of $X_n$. Since we are
mainly interested in the variance, we may assume that $X_0=X_1=0$.
This then translates into the recurrence for the moment-generating
functions (assuming $T_n$ independent of $X_n$)
\[
    M_n(y)= \mathbb{E}(e^{T_ny})\sum_{1\le k< n}
    \pi_{n,k}'M_k(y)M_{n-k}(y)
    \qquad (n\ge 2),
\]
with $M_0(y)=M_1(y)=1$. It follows that the Poisson generating
function $\tilde{f}_1$ of $\mathbb{E}(X_n)$ satisfies the functional
equation
\begin{align}\label{fe-pat-trie}
    \tilde{f}_1(z)=\tilde{f}_1(pz)+\tilde{f}_1(qz)
    +\tilde{g}_1(z)-e^{-qz}\tilde{g}_1(pz)
    -e^{-pz}\tilde{g}_1(qz),
\end{align}
with $\tilde{f}_1(0)=\tilde{f}_1'(0)=0$, where $\tilde{g}_1$
represents the Poisson generating function of $\mathbb{E}(T_n)$. For
convenience, we also assume $\tilde{g}_1(0)=\tilde{g}_1'(0)=0$.

The same tools we developed for tries readily apply to
\eqref{fe-pat-trie} and the same asymptotic pattern holds.

\begin{thm}\label{pmean}
Let $0<\theta<\pi/2, \alpha<1$ and $\beta\in\mathbb{R}$.
\begin{itemize}

\item[(a)] If more precisely $\tilde{g}_1\in
\JS_{\!\!\!\alpha,\beta}$, then
\[
    \frac{\mathbb{E}(X_n)}{n} =
    \frac{G_1(-1)}{h} + \mathscr{F}[G](r\log_{1/p}n)
    +o(1),
\]
where $G_1(s)=\mathscr{M}[\tilde{g}_1(z)-e^{-qz}
\tilde{g}_1(pz)-e^{-pz}\tilde{g}_1(qz);s]$.

\item[(b)] If $\tilde{g}_1 \in \JS$ and $\tilde{g}_1(z)=cz+O(\vert
z\vert^\alpha\left(\log_{+}\vert z\vert\right)^{\beta})$ uniformly
for $\vert\arg(z)\vert\le\theta$, then
\begin{align*}
    \frac{\mathbb{E}(X_n)}{n} &= \frac{c}{h}\log n
    +\frac{d}{h}+\frac{p\log^2p+q\log^2 q}
    {2h^2} + \mathscr{F}[G_1](r\log_{1/p}n)
    +o(1),
\end{align*}
where $G_1(s)$ is the meromorphic continuation of
$\mathscr{M}[\tilde{g}_1(z)-e^{-qz}\tilde{g}_1(pz)
-e^{-pz}\tilde{g}_1(qz);s]$ and $d=\lim_{s\rightarrow
-1}(G_1(s)+c/(s+1))$.
\end{itemize}
\end{thm}
Since the method of proof is the same as that of Theorem~\ref{mean},
we omit the details.

For the variance of $X_n$, we have, using the same notations,
\[
    \tilde{V}_X(z)
    =\tilde{V}_X(pz)+\tilde{V}_X(qz)
    +\tilde{V}_T(z)+\tilde{\phi}_0(z)
    +\tilde{\phi}_1(z)+\tilde{\phi}_2(z),
\]
where
\begin{align*}
    \tilde{\phi}_0(z)&=-e^{-qz}\tilde{g}_2(pz)
    -e^{-pz}\tilde{g}_2(qz)+2\tilde{g}_1(z)
    \left(e^{-qz}\tilde{g}_1(pz)+e^{-pz}\tilde{g}_1(qz)\right)\\
    &\quad-2z\tilde{g}'_1(z)\left(qe^{-qz}\tilde{g}_1(pz)
    +pe^{-pz}\tilde{g}_1(qz)-pe^{-qz}\tilde{g}'_1(pz)
    -qe^{-pz}\tilde{g}'_1(qz)\right)\\
    &\quad-z\left(qe^{-qz}\tilde{g}_1(pz)
    +pe^{-pz}\tilde{g}_1(qz)-pe^{-qz}\tilde{g}'_1(pz)
    -qe^{-pz}\tilde{g}'_1(qz)\right)^2\\
    &\quad-\left(e^{-qz}\tilde{g}_1(pz)
    +e^{-pz}\tilde{g}_1(qz)\right)^2,
\end{align*}
and
\begin{align*}
    \tilde{\phi}_1(z)&=\tilde{h}_2(z)-2\tilde{g}_1(z)
    \left(\tilde{f}_1(pz)+\tilde{f}_1(qz)\right)-2z\tilde{g}_1'(z)
    \left(p\tilde{f}_1'(pz)+q\tilde{f}_1'(qz)\right)\\
    &\quad+2\left(e^{-qz}\tilde{g}_1(pz)
    +e^{-pz}\tilde{g}_1(qz)\right)\left(\tilde{f}_1(pz)
    +\tilde{f}_1(qz)\right)\\
    &\quad-2z\left(qe^{-qz}\tilde{g}_1(pz)
    +pe^{-pz}\tilde{g}_1(qz)-pe^{-qz}\tilde{g}'_1(pz)
    -qe^{-pz}\tilde{g}'_1(qz)\right)\\
    &\qquad\qquad \times
    \left(p\tilde{f}'_1(pz)+q\tilde{f}'_1(qz)\right)\\
    \tilde{\phi}_2(z)&=pqz\left(\tilde{f}_1'(pz)
    -\tilde{f}_1'(qz)\right)^2.
\end{align*}
Here $\tilde{h}_2$ is given by
\begin{align*}
    \tilde{h}_2(z) &=2e^{-z}
    \sum_{n\ge 0}\mathbb{E}(T_n)\sum_{0\le j\le
    n}\pi_{n,j}(\mathbb{E}(X_j)+ {\mathbb
    E}(X_{n-j}))\frac{z^n}{n!}\\
    &\qquad -2e^{-z}\sum_{n\ge 0}(p^n+q^n){\mathbb
    E}(T_n)\mathbb{E}(X_n)\frac{z^n}{n!}.
\end{align*}
Note that, by Propositions~\ref{prop-closure} and \ref{prop-at}, if
$\tilde{g}_1\in\JS$, then $\tilde{f}_1\in\JS$, which in turn
implies, by Proposition~\ref{prop-hada}, that $\tilde{h}_2\in \JS$.
Consequently, if $\tilde{g}_1 \in\JS$ and $\tilde{g}_2\in\JS$, then
both $\tilde{f}_1\in\JS$ and $\tilde{f}_2\in\JS$. Thus our approach
applies to $\mathbb{V}(X_n)$.

\begin{thm}\label{pvariance}
Let $0<\theta<\pi/2, \alpha<1$ and $\beta\in\mathbb{R}$. Assume
$\tilde{g}_1, \tilde{g}_2\in\JS$ and $\tilde{V}_T(z)
=O\left(|z|^{\alpha}(\log_{+}|z|)^{\beta}\right)$ for
$|\arg(z)|\le\theta$.
\begin{itemize}
\item[(a)] If $p=q=1/2$, and
$\tilde{g}_1\in\JS_{\!\!\!\alpha,\beta}$ or
$\tilde{g}_1\in\JS_{\!\!\!1,0}$. Then
\[
    \frac{\mathbb{V}(X_n)}{n}
    =\frac{1}{\log 2}\sum_{k\in\mathbb{Z}}
    G(-1+\chi_k)n^{-\chi_k} +o(1),
\]
where $G(s)=\mathscr{M}[\tilde{V}_T(z)+\tilde{\phi}_0(z)
+\tilde{\phi}_1(z);s]$.

\item[(b)] Assume $p\ne q$.
\begin{itemize}
\item[(i)] If $\tilde{g}_1\in\JS_{\!\!\!\alpha,\beta}$, then
\[
    \frac{\mathbb{V}(X_n)}{n}=
    \frac{G(-1)}{h} + \mathscr{F}[G](r\log_{1/p}n)
    +o(1),
\]
where $G(s)=\Phi_1(s)+\Phi_2(s)$ with
$\Phi_1(s)=\mathscr{M}[\tilde{V}_T(z) +\tilde{\phi}_0(z)
+\tilde{\phi}_1(z)]$ and $\Phi_2(s)$ is an analytic continuation of
$\mathscr{M} [\tilde{\phi}_2;s]$.

\item[(ii)] If $\tilde{g}_1(z)=z+O(\vert z\vert^\alpha
\left(\log_{+}\vert z\vert\right)^{\beta})$ uniformly for
$\vert\arg(z)\vert\le\theta$. Then
\begin{align*}
    \frac{\mathbb{V}(X_n)}{n}
    &= \frac{pq\log^2(p/q)}{h^3}\,\log n
    +\frac{d}{h}+\frac{p\log^2p+q\log^2q}{2h^2}
    \\&\qquad  + \mathscr{F}[G](r\log_{1/p}n)
    +o(1).
\end{align*}
Here $G(s)=\Phi_1(s)+\Phi_2(s)$ with $\Phi_1(s)$ as above,
$\Phi_2(s)$ is a meromorphic continuation of $\mathscr{M}
[\tilde{\phi}_2;s]$ and $d=\lim_{s\rightarrow-1}
(G(s)+pq\log^2(p/q)/(h^2(s+1)))$.
\end{itemize}
\end{itemize}
\end{thm}
The proof follows the same arguments as that of Theorem~\ref{var}
and is omitted.

Consider the external path length, which satisfies \eqref{Xn-PT}
with $T_n=n$. In this case, we have
\[
    \tilde{g}_1(z)=z(1-e^{-z}),
    \qquad \tilde{g}_2(z)=z(1-e^{-z})+z^2,
\]
and
\[
    \tilde{V}_T(z)=e^{-z}(z(1-e^{-z})+z^2(1-z)e^{-z}).
\]
Also
\begin{align*}
    \tilde{\phi}_1(z)&=-2zpq\left((zp-1)e^{-pz}
    \tilde{f}'_1(pz)+(zq-1)e^{-qz}\tilde{f}'_1(qz)
    \right. \\ &\qquad\qquad \left.
    +(zp+1)e^{-qz}\tilde{f}'_1(pz)
    +(zq+1)e^{-pz}\tilde{f}'_1(qz)\right)\\
    &\qquad\qquad +2qze^{-pz}\tilde{f}_1(pz)
    +2pze^{-qz}\tilde{f}_1(qz).
\end{align*}

Observe that
\[
    G_1(s) :=
    \mathscr{M}[\tilde{g}_1(z)-e^{-qz}\tilde{g}_1(pz)
    -e^{-pz}\tilde{g}_1(qz);s]=-\Gamma(s+1)
    \left(qp^{-s-1}+pq^{-s-1}\right).
\]
Thus, by Theorem \ref{pmean},
\begin{align*}
    \frac{\mathbb{E}(X_n)}{n}&= \frac1h\log n
    +\frac{\gamma}{h}
    +\frac{p\log^2 p+q\log^2 q}{2h^2}-1
    + \mathscr{F}[G_1](r\log_{1/p}n)
    +o(1).
\end{align*}

Now by Theorem \ref{pvariance}, the variance satisfies
\[
    \frac{\mathbb{V}(X_n)}{n} =
    \frac{G(-1)}{h} + \mathscr{G}[G](r\log_{1/p}n)+o(1),
\]
where $G=\Phi_1+\Phi_2$, as described in Theorem~\ref{pvariance}.
Expressions can be derived for $G$. For brevity, consider only the
symmetric case for which we have
\begin{align*}
    G(s)=\Phi_1(s)&=\Gamma(s+1)\left(2^{s+1}
    (s+2)-\frac{s^2+3s+6}{4}\right)
    \\ &\qquad +2^{s+2}
    \sum_{j\ge 1}\frac{(-1)^j\Gamma(s+j+2)}{(j-1)!(2^j-1)}.
\end{align*}
Note that the last series has the alternative form
\[
    \sum_{j\ge 1}\frac{(-1)^j\Gamma(s+j+2)}{(j-1)!(2^j-1)}
    = -\Gamma(s+3)\sum_{j\ge1}\frac1{2^j(1+2^{-j})^{3+s}}.
\]
Hence, the mean value of the periodic function is given by
\[
    1+\frac{3}{4\log 2}+\frac{2}{\log 2}
    \sum_{j\ge 1}\frac{1}{2^j(1+2^{-j})^2}
    \approx 0.36132\,60597\,81678\cdots
\]
which is the same as that obtained in \cite{kirschenhofer89a} with a
different expression (equating our expression with theirs gives the
same identity \eqref{id}).

\section{Conclusions}
\label{sec-conclusion}

The prevalent appearance in diverse modeling contexts and high
concentration of the binomial distribution make BSPs a distinctive
subject full of featured properties and numerous extensions.
Periodic oscillation is among the phenomena for which analytic tools
proved to be a successful bridge between theory and practical
observations. The analytic methodology developed in this paper,
based largely on earlier works founded by Flajolet and his coauthors
and aiming at clarifying the periodic oscillation of the variance,
is itself easily amended for other circumstances, including
particularly the case of quadratic shape measures such as the Wiener
index (see \cite{fuchs12c}) or the analysis of partial-match queries
(see \cite{fuchs10a}). The combination of Mellin analysis and
analytic de-Poissonization (operated at the more abstract level of
admissible functions) proves once again to be powerful tools for
unriddling the intrinsic complexity of the asymptotic variance, and
provides an efficient mechanical \emph{art of conjecturing} and
proving in more general contexts the structure of the variance. More
developments will be discussed in a subsequent paper.

\bibliographystyle{acm}
\bibliography{tries,PF-MSN}

\begin{thebibliography}{100}

\bibitem{aldous96a}
{\sc Aldous, D.}
\newblock Probability distributions on cladograms.
\newblock In {\em Random discrete structures ({M}inneapolis, {MN}, 1993)},
  vol.~76 of {\em IMA Vol. Math. Appl.} Springer, New York, 1996, pp.~1--18.

\bibitem{banderier13a}
{\sc Banderier, C., Hwang, H.-K., Ravelomanana, V., and Zacharovas, V.}
\newblock Analysis of an exhaustive search algorithm in random graphs and the
  {$n^{c \log n}$}-asymptotics.
\newblock Submitted for publication.

\bibitem{biglieri07a}
{\sc Biglieri, E., and Gy{\"o}rfi, L.}
\newblock {\em Multiple Access Channels: Theory and Practice}.
\newblock NATO Security Through Science Series. D: Information and Com. Ios
  PressInc, 2007.

\bibitem{blum05a}
{\sc Blum, M. G.~B., and Fran{\c{c}}ois, O.}
\newblock Minimal clade size and external branch length under the neutral
  coalescent.
\newblock {\em Adv. in Appl. Probab. 37}, 3 (2005), 647--662.

\bibitem{bourdon01a}
{\sc Bourdon, J., Nebel, M., and Vall{\'e}e, B.}
\newblock On the stack-size of general tries.
\newblock {\em Theor. Inform. Appl. 35}, 2 (2001), 163--185.

\bibitem{bradley85a}
{\sc Bradley, R., and Strenski, P.}
\newblock Directed aggregation on the bethe lattice: Scaling, mappings, and
  universality.
\newblock {\em Physical Review B 31}, 7 (1985), 4319.

\bibitem{chen03a}
{\sc Chen, W.-M., and Hwang, H.-K.}
\newblock Analysis in distribution of two randomized algorithms for finding the
  maximum in a broadcast communication model.
\newblock {\em J. Algorithms 46}, 2 (2003), 140--177.

\bibitem{christophi01a}
{\sc Christophi, C.~A., and Mahmoud, H.~M.}
\newblock Distribution of the size of random hash trees, pebbled hash trees and
  {$N$}-trees.
\newblock {\em Statist. Probab. Lett. 53}, 3 (2001), 277--282.

\bibitem{PF161}
{\sc Cl{{\'e}}ment, J., Flajolet, P., and Vall{{\'e}}e, B.}
\newblock Dynamical sources in information theory: a general analysis of trie
  structures.
\newblock {\em Algorithmica 29}, 1-2 (2001), 307--369.

\bibitem{cristea07a}
{\sc Cristea, L.-L., and Prodinger, H.}
\newblock Order statistics for the {C}antor-{F}ibonacci distribution.
\newblock {\em Aequationes Math. 73}, 1-2 (2007), 78--91.

\bibitem{de-bruijn72a}
{\sc de~Bruijn, N.~G., Knuth, D.~E., and Rice, S.~O.}
\newblock The average height of planted plane trees.
\newblock In {\em Graph theory and Computing}. Academic Press, New York, 1972,
  pp.~15--22.

\bibitem{de-la-briandais59a}
{\sc De~La~Briandais, R.}
\newblock File searching using variable length keys.
\newblock In {\em Papers presented at the Western Joint Computer Conference
  (March 3-5, 1959,)\/} (1959), ACM, pp.~295--298.

\bibitem{dean06a}
{\sc Dean, D.~S., and Majumdar, S.~N.}
\newblock Phase transition in a generalized {E}den growth model on a tree.
\newblock {\em J. Stat. Phys. 124}, 6 (2006), 1351--1376.

\bibitem{devroye86a}
{\sc Devroye, L.}
\newblock {\em Lecture Notes on Bucket Algorithms}, vol.~6 of {\em Progress in
  Computer Science}.
\newblock Birkh\"auser Boston Inc., Boston, MA, 1986.

\bibitem{devroye05a}
{\sc Devroye, L.}
\newblock Universal asymptotics for random tries and {PATRICIA} trees.
\newblock {\em Algorithmica 42}, 1 (2005), 11--29.

\bibitem{drmota09a}
{\sc Drmota, M., Gittenberger, B., Panholzer, A., Prodinger, H., and Ward,
  M.~D.}
\newblock On the shape of the fringe of various types of random trees.
\newblock {\em Math. Methods Appl. Sci. 32}, 10 (2009), 1207--1245.

\bibitem{drmota11a}
{\sc Drmota, M., and Szpankowski, W.}
\newblock A master theorem for discrete divide and conquer recurrences.
\newblock In {\em Proceedings of the {T}wenty-{S}econd {A}nnual {ACM}-{SIAM}
  {S}ymposium on {D}iscrete {A}lgorithms\/} (Philadelphia, PA, 2011), SIAM,
  pp.~342--361.

\bibitem{eisenberg08a}
{\sc Eisenberg, B.}
\newblock On the expectation of the maximum of {IID} geometric random
  variables.
\newblock {\em Statist. Probab. Lett. 78}, 2 (2008), 135--143.

\bibitem{erdos87a}
{\sc Erd{\H{o}}s, P., Hildebrand, A., Odlyzko, A., Pudaite, P., and Reznick,
  B.}
\newblock The asymptotic behavior of a family of sequences.
\newblock {\em Pacific J. Math. 126}, 2 (1987), 227--241.

\bibitem{fagin79a}
{\sc Fagin, R., Nievergelt, J., Pippenger, N., and Strong, H.~R.}
\newblock Extendible hashing - a fast access method for dynamic files.
\newblock {\em ACM Trans. Database Syst. 4}, 3 (1979), 315--344.

\bibitem{PF055}
{\sc Fayolle, G., Flajolet, P., and Hofri, M.}
\newblock On a functional equation arising in the analysis of a protocol for a
  multi-access broadcast channel.
\newblock {\em Adv. in Appl. Probab. 18}, 2 (1986), 441--472.

\bibitem{PF049}
{\sc Fayolle, G., Flajolet, P., Hofri, M., and Jacquet, P.}
\newblock Analysis of a stack algorithm for random multiple-access
  communication.
\newblock {\em IEEE Trans. Inform. Theory 31}, 2 (1985), 244--254.

\bibitem{fill96a}
{\sc Fill, J.~A., Mahmoud, H.~M., and Szpankowski, W.}
\newblock On the distribution for the duration of a randomized leader election
  algorithm.
\newblock {\em Ann. Appl. Probab. 6}, 4 (1996), 1260--1283.

\bibitem{PF037}
{\sc Flajolet, P.}
\newblock On the performance evaluation of extendible hashing and trie
  searching.
\newblock {\em Acta Inform. 20}, 4 (1983), 345--369.

\bibitem{PF048}
{\sc Flajolet, P.}
\newblock Approximate counting: a detailed analysis.
\newblock {\em BIT 25}, 1 (1985), 113--134.

\bibitem{PF071}
{\sc Flajolet, P.}
\newblock {\'E}valuation de protocoles de communication~: aspects
  math{\'e}matiques.
\newblock In {\em Le Codage et la Transmission de l'Information~: Journ{\'e}e
  annuelle de la Soci{\'e}t{\'e} Math{\'e}matique de France}. Soci{\'e}t{\'e}
  Math{\'e}matique de France, 1988, pp.~1--22.

\bibitem{PF180}
{\sc Flajolet, P.}
\newblock Counting by coin tossings.
\newblock In {\em Advances in Computer Science - ASIAN 2004. Higher-Level
  Decision Making; Proceedings of the 9th Asian Computing Science Conference;
  Dedicated to Jean-Louis Lassez on the Occasion of His 5th Cycle Birthday},
  M.~J. Maher, Ed., vol.~3321. Springer, Berlin/Heidelberg, 2004, pp.~1--12.

\bibitem{PF193}
{\sc Flajolet, P., Fusy, {\'e}., Gandouet, O., and Meunier, F.}
\newblock Hyper{L}og{L}og: the analysis of a near-optimal cardinality
  estimation algorithm.
\newblock In {\em 2007 {C}onference on {A}nalysis of {A}lgorithms, {A}of{A}
  07}, Discrete Math. Theor. Comput. Sci. Proc., AH. Assoc. Discrete Math.
  Theor. Comput. Sci., Nancy, 2007, pp.~127--145.

\bibitem{PF120}
{\sc Flajolet, P., Gourdon, X., and Dumas, P.}
\newblock Mellin transforms and asymptotics: harmonic sums.
\newblock {\em Theoret. Comput. Sci. 144}, 1-2 (1995), 3--58.

\bibitem{PF050}
{\sc Flajolet, P., and Martin, G.~N.}
\newblock Probabilistic counting algorithms for data base applications.
\newblock {\em J. Comput. System Sci. 31}, 2 (1985), 182--209.

\bibitem{PF200}
{\sc Flajolet, P., Pelletier, M., and Soria, M.}
\newblock On {B}uffon machines and numbers.
\newblock In {\em Proceedings of the {T}wenty-{S}econd {A}nnual {ACM}-{SIAM}
  {S}ymposium on {D}iscrete {A}lgorithms\/} (Philadelphia, PA, 2011), SIAM,
  pp.~172--183.

\bibitem{PF058}
{\sc Flajolet, P., and Puech, C.}
\newblock Partial match retrieval of multidimensional data.
\newblock {\em J. Assoc. Comput. Mach. 33}, 2 (1986), 371--407.

\bibitem{PF052}
{\sc Flajolet, P., R{{\'e}}gnier, M., and Sedgewick, R.}
\newblock Some uses of the {M}ellin integral transform in the analysis of
  algorithms.
\newblock In {\em Combinatorial algorithms on words ({M}aratea, 1984)}, vol.~12
  of {\em NATO Adv. Sci. Inst. Ser. F Comput. Systems Sci.} Springer, Berlin,
  1985, pp.~241--254.

\bibitem{PF208}
{\sc Flajolet, P., Roux, M., and Vall{{\'e}}e, B.}
\newblock Digital trees and memoryless sources: from arithmetics to analysis.
\newblock In {\em 21st {I}nternational {M}eeting on {P}robabilistic,
  {C}ombinatorial, and {A}symptotic {M}ethods in the {A}nalysis of {A}lgorithms
  ({A}of{A}'10)}, Discrete Math. Theor. Comput. Sci. Proc., AM. Assoc. Discrete
  Math. Theor. Comput. Sci., Nancy, 2010, pp.~233--260.

\bibitem{PF060}
{\sc Flajolet, P., and Saheb, N.}
\newblock The complexity of generating an exponentially distributed variate.
\newblock {\em J. Algorithms 7}, 4 (1986), 463--488.

\bibitem{PF124}
{\sc Flajolet, P., and Sedgewick, R.}
\newblock Mellin transforms and asymptotics: finite differences and {R}ice's
  integrals.
\newblock {\em Theoret. Comput. Sci. 144}, 1-2 (1995), 101--124.

\bibitem{PF201}
{\sc Flajolet, P., and Sedgewick, R.}
\newblock {\em Analytic combinatorics}.
\newblock Cambridge University Press, Cambridge, 2009.

\bibitem{PF034}
{\sc Flajolet, P., and Sotteau, D.}
\newblock A recursive partitioning process of computer science.
\newblock In {\em Proceedings of the Second World Conference on Mathematics at
  the Service of Man\/} (Las Palmas, Canary Islands, Spain, 1982),
  A.~Ballester, D.~Card{\'u}s, and E.~Trillas, Eds., Universidad
  Polit{\'e}cnica de Las Palmas, pp.~25--30.

\bibitem{PF036}
{\sc Flajolet, P., and Steyaert, J.-M.}
\newblock A branching process arising in dynamic hashing, trie searching and
  polynomial factorization.
\newblock In {\em Automata, languages and programming ({A}arhus, 1982)},
  vol.~140 of {\em Lecture Notes in Comput. Sci.} Springer, Berlin, 1982,
  pp.~239--251.

\bibitem{fredkin60a}
{\sc Fredkin, E.}
\newblock Trie memory.
\newblock {\em Commun. ACM 3\/} (September 1960), 490--499.

\bibitem{fredman74a}
{\sc Fredman, M.~L., and Knuth, D.~E.}
\newblock Recurrence relations based on minimization.
\newblock {\em J. Math. Anal. Appl. 48\/} (1974), 534--559.

\bibitem{fuchs10a}
{\sc Fuchs, M.}
\newblock The variance for partial match retrievals in {$k$}-dimensional bucket
  digital trees.
\newblock In {\em 21st {I}nternational {M}eeting on {P}robabilistic,
  {C}ombinatorial, and {A}symptotic {M}ethods in the {A}nalysis of {A}lgorithms
  ({A}of{A}'10)}, Discrete Math. Theor. Comput. Sci. Proc., AM. Assoc. Discrete
  Math. Theor. Comput. Sci., Nancy, 2010, pp.~261--275.

\bibitem{fuchs11a}
{\sc Fuchs, M.}
\newblock The subtree size profile of plane-oriented recursive trees.
\newblock In {\em A{NALCO}11---{W}orkshop on {A}nalytic {A}lgorithmics and
  {C}ombinatorics}. SIAM, Philadelphia, PA, 2011, pp.~85--92.

\bibitem{fuchs12c}
{\sc Fuchs, M., and Lee, C.-K.}
\newblock The {W}iener index of random digital tries.
\newblock submitted for publication (2012).

\bibitem{fuchs12b}
{\sc Fuchs, M., and Prodinger, H.}
\newblock Words with a generalized restricted growth property.
\newblock {\em Indagationes Mathematicae\/} (2013), to appear.

\bibitem{gelenbe86a}
{\sc Gelenbe, E., Nelson, R., Philips, T., and Tantawi, A.}
\newblock An approximation of the processing time for a random graph model of
  parallel computation.
\newblock In {\em Proceedings of 1986 ACM Fall Joint Computer Conference\/}
  (Los Alamitos, CA, USA, 1986), ACM '86, IEEE Computer Society Press,
  pp.~691--697.

\bibitem{goodrich08a}
{\sc Goodrich, M.~T., and Hirschberg, D.~S.}
\newblock Improved adaptive group testing algorithms with applications to
  multiple access channels and dead sensor diagnosis.
\newblock {\em J. Comb. Optim. 15}, 1 (2008), 95--121.

\bibitem{grabner96a}
{\sc Grabner, P.~J., and Prodinger, H.}
\newblock Asymptotic analysis of the moments of the {C}antor distribution.
\newblock {\em Statist. Probab. Lett. 26}, 3 (1996), 243--248.

\bibitem{grabner02a}
{\sc Grabner, P.~J., and Prodinger, H.}
\newblock Sorting algorithms for broadcast communications: mathematical
  analysis.
\newblock {\em Theoret. Comput. Sci. 289}, 1 (2002), 51--67.

\bibitem{hayman56a}
{\sc Hayman, W.~K.}
\newblock A generalisation of {S}tirling's formula.
\newblock {\em J. Reine Angew. Math. 196\/} (1956), 67--95.

\bibitem{hildebrandt59a}
{\sc Hildebrandt, P., and Isbitz, H.}
\newblock Radix exchange---an internal sorting method for digital computers.
\newblock {\em J. ACM 6}, 2 (Apr. 1959), 156--163.

\bibitem{hubalek02a}
{\sc Hubalek, F., Hwang, H.-K., Lew, W., Mahmoud, H., and Prodinger, H.}
\newblock A multivariate view of random bucket digital search trees.
\newblock {\em J. Algorithms 44}, 1 (2002), 121--158.

\bibitem{hush98a}
{\sc Hush, D.~R., and Wood, C.}
\newblock Analysis of tree algorithms for rfid arbitration.
\newblock In {\em 1998 IEEE International Symposium on Information Theory\/}
  (1998), IEEE, p.~107.

\bibitem{hwang10a}
{\sc Hwang, H.-K., Fuchs, M., and Zacharovas, V.}
\newblock Asymptotic variance of random symmetric digital search trees.
\newblock {\em Discrete Math. Theor. Comput. Sci. 12}, 2 (2010), 103--165.

\bibitem{jacquet85a}
{\sc Jacquet, P.}
\newblock {The part and try algorithm adapted to free channel access}.
\newblock Tech. Rep. RR-0436, INRIA, Rocquencourt, Aug. 1985.

\bibitem{jacquet90b}
{\sc Jacquet, P., and Muhlethaler, P.}
\newblock {Marginal throughtput of a stack algorithm for CSMA/CD random length
  packet communication when the load is over the channel efficiency}.
\newblock Tech. Rep. RR-1275, INRIA, Rocquencourt, Aug. 1990.

\bibitem{jacquet86a}
{\sc Jacquet, P., and R{\'e}gnier, M.}
\newblock Trie partitioning process: limiting distributions.
\newblock In {\em C{AAP} '86 ({N}ice, 1986)}, vol.~214 of {\em Lecture Notes in
  Comput. Sci.} Springer, Berlin, 1986, pp.~196--210.

\bibitem{jacquet88b}
{\sc Jacquet, P., and Regnier, M.}
\newblock {Normal limiting distribution for the size and the external path
  length of tries}.
\newblock Tech. Rep. RR-0827, INRIA, Rocquencourt, Apr. 1988.

\bibitem{jacquet88a}
{\sc Jacquet, P., and R{\'e}gnier, M.}
\newblock Normal limiting distribution of the size of tries.
\newblock In {\em Performance '87 ({B}russels, 1987)}. North-Holland,
  Amsterdam, 1988, pp.~209--223.

\bibitem{jacquet95a}
{\sc Jacquet, P., and Szpankowski, W.}
\newblock Asymptotic behavior of the {L}empel-{Z}iv parsing scheme and [in]
  digital search trees.
\newblock {\em Theoret. Comput. Sci. 144}, 1-2 (1995), 161--197.

\bibitem{jacquet98a}
{\sc Jacquet, P., and Szpankowski, W.}
\newblock Analytical de-{P}oissonization and its applications.
\newblock {\em Theoret. Comput. Sci. 201}, 1-2 (1998), 1--62.

\bibitem{janson12a}
{\sc Janson, S.}
\newblock Renewal theory in the analysis of tries and strings.
\newblock {\em Theoret. Comput. Sci. 416\/} (2012), 33--54.

\bibitem{janson97a}
{\sc Janson, S., and Szpankowski, W.}
\newblock Analysis of an asymmetric leader election algorithm.
\newblock {\em Electron. J. Combin. 4}, 1 (1997), Research Paper 17, 16 pp.
  (electronic).

\bibitem{janssen00a}
{\sc Janssen, A. J. E.~M., and de~Jong, M.~J.}
\newblock Analysis of contention tree algorithms.
\newblock {\em IEEE Trans. Inform. Theory 46}, 6 (2000), 2163--2172.

\bibitem{kaplan85a}
{\sc Kaplan, M.~A., and Gulko, E.}
\newblock Analytic properties of multiple-access trees.
\newblock {\em IEEE Trans. Inform. Theory 31}, 2 (1985), 255--263.

\bibitem{kemp78a}
{\sc Kemp, R.}
\newblock The average number of registers needed to evaluate a binary tree
  optimally.
\newblock {\em Acta Inform. 11}, 4 (1978/79), 363--372.

\bibitem{kirschenhofer88b}
{\sc Kirschenhofer, P., and Prodinger, H.}
\newblock {$b$}-tries: a paradigm for the use of number-theoretic methods in
  the analysis of algorithms.
\newblock In {\em Contributions to General Algebra, 6}.
  H\"older-Pichler-Tempsky, Vienna, 1988, pp.~141--154.

\bibitem{kirschenhofer91a}
{\sc Kirschenhofer, P., and Prodinger, H.}
\newblock On some applications of formulae of {R}amanujan in the analysis of
  algorithms.
\newblock {\em Mathematika 38}, 1 (1991), 14--33.

\bibitem{kirschenhofer89a}
{\sc Kirschenhofer, P., Prodinger, H., and Szpankowski, W.}
\newblock On the balance property of {P}atricia tries: external path length
  viewpoint.
\newblock {\em Theoret. Comput. Sci. 68}, 1 (1989), 1--17.

\bibitem{kirschenhofer89b}
{\sc Kirschenhofer, P., Prodinger, H., and Szpankowski, W.}
\newblock On the variance of the external path length in a symmetric digital
  trie.
\newblock {\em Discrete Appl. Math. 25}, 1-2 (1989), 129--143.

\bibitem{kirschenhofer93a}
{\sc Kirschenhofer, P., Prodinger, H., and Szpankowski, W.}
\newblock Multidimensional digital searching and some new parameters in tries.
\newblock {\em Internat. J. Found. Comput. Sci. 4}, 1 (1993), 69--84.

\bibitem{knuth78a}
{\sc Knuth, D.~E.}
\newblock The average time for carry propagation.
\newblock {\em Nederl. Akad. Wetensch. Indag. Math. 40}, 2 (1978), 238--242.

\bibitem{knuth98a}
{\sc Knuth, D.~E.}
\newblock {\em The Art of Computer Programming, Volume III, Sorting and
  Searching}, 2nd~ed.
\newblock Addison-Wesley, 1998.

\bibitem{louchard12a}
{\sc Louchard, G., Prodinger, H., and Ward, M.~D.}
\newblock Number of survivors in the presence of a demon.
\newblock {\em Period. Math. Hungar. 64}, 1 (2012), 101--117.

\bibitem{maddison91a}
{\sc Maddison, W.~P., and Slatkin, M.}
\newblock Null models for the number of evolutionary steps in a character on a
  phylogenetic tree.
\newblock {\em Evolution\/} (1991), 1184--1197.

\bibitem{PF159}
{\sc Mahmoud, H., Flajolet, P., Jacquet, P., and R{{\'e}}gnier, M.}
\newblock Analytic variations on bucket selection and sorting.
\newblock {\em Acta Inform. 36}, 9-10 (2000), 735--760.

\bibitem{mahmoud92a}
{\sc Mahmoud, H.~M.}
\newblock {\em Evolution of Random Search Trees}.
\newblock Wiley-Interscience Series in Discrete Mathematics and Optimization.
  John Wiley \& Sons Inc., New York, 1992.

\bibitem{majumdar03a}
{\sc Majumdar, S.~N.}
\newblock Traveling front solutions to directed diffusion-limited aggregation,
  digital search trees, and the lempel-ziv data compression algorithm.
\newblock {\em Physical Review E 68}, 2 (2003), 026103.

\bibitem{massey81a}
{\sc Massey, J.~L.}
\newblock Collision-resolution algorithms and random-access communications.
\newblock In {\em Multi-User Communication Systems}, G.~Longo, Ed., vol.~CISM
  Courses and Lectures. Springer, 1981, pp.~73--137.

\bibitem{PF054}
{\sc Mathys, P., and Flajolet, P.}
\newblock {$Q$}-ary collision resolution algorithms in random-access systems
  with free or blocked channel access.
\newblock {\em IEEE Trans. Inform. Theory 31}, 2 (1985), 217--243.

\bibitem{mellier04a}
{\sc Mellier, R., Myoupo, J.~F., and Ravelomanana, V.}
\newblock A non-token-based-distributed mutual exclusion algorithm for
  single-hop mobile ad hoc networks.
\newblock In {\em MWCN 2004\/} (2004), pp.~287--298.

\bibitem{mendelson82a}
{\sc Mendelson, H.}
\newblock Analysis of extendible hashing.
\newblock {\em IEEE Trans. Software Engrg.}, 6 (1982), 611--619.

\bibitem{molle93a}
{\sc Molle, M.~L., and Polyzos, G.~C.}
\newblock Conflict resolution algorithms and their performance analysis.
\newblock {\em Tech. Rep., Dept. CSE, UCSD\/} (1993).

\bibitem{morrison68a}
{\sc Morrison, D.~R.}
\newblock Patricia---practical algorithm to retrieve information coded in
  alphanumeric.
\newblock {\em J. ACM 15}, 4 (1968), 514--534.

\bibitem{myoupo03a}
{\sc Myoupo, J.~F., Thimonier, L., and Ravelomanana, V.}
\newblock Average case analysis-based protocols to initialize packet radio
  networks.
\newblock {\em Wireless Commun. Mobile Comput. 3}, 4 (2003), 539--548.

\bibitem{namboodiri10a}
{\sc Namboodiri, V., and Gao, L.}
\newblock Energy-aware tag anticollision protocols for rfid systems.
\newblock {\em IEEE Trans. Mobile Comput. 9}, 1 (2010), 44--59.

\bibitem{neininger04a}
{\sc Neininger, R., and R{\"u}schendorf, L.}
\newblock A general limit theorem for recursive algorithms and combinatorial
  structures.
\newblock {\em Ann. Appl. Probab. 14}, 1 (2004), 378--418.

\bibitem{nguyen-the04a}
{\sc Nguyen-The, M.}
\newblock {\em Distribution de valuations sur les arbres}.
\newblock PhD thesis, LIX, Ecole Polytechnique, 2004.

\bibitem{PF174}
{\sc Nicod{{\`e}}me, P., Salvy, B., and Flajolet, P.}
\newblock Motif statistics.
\newblock {\em Theoret. Comput. Sci. 287}, 2 (2002), 593--617.

\bibitem{park08a}
{\sc Park, G., Hwang, H.-K., Nicod{\`e}me, P., and Szpankowski, W.}
\newblock Profiles of tries.
\newblock {\em SIAM J. Comput. 38}, 5 (2008/09), 1821--1880.

\bibitem{polasek00a}
{\sc Polasek, W.}
\newblock The {B}ernoullis and the origin of probability theory: Looking back
  after 300 years.
\newblock {\em Resonance 5}, 8 (2000), 26--42.

\bibitem{prodinger93a}
{\sc Prodinger, H.}
\newblock How to select a loser.
\newblock {\em Discrete Math. 120}, 1-3 (1993), 149--159.

\bibitem{regnier89a}
{\sc R{\'e}gnier, M., and Jacquet, P.}
\newblock New results on the size of tries.
\newblock {\em IEEE Trans. Inform. Theory 35}, 1 (1989), 203--205.

\bibitem{ressler92a}
{\sc Ressler, E.~K.}
\newblock Random list permutations in place.
\newblock {\em Inform. Process. Lett. 43}, 5 (1992), 271--275.

\bibitem{schachinger95b}
{\sc Schachinger, W.}
\newblock On the variance of a class of inductive valuations of data structures
  for digital search.
\newblock {\em Theoret. Comput. Sci. 144}, 1-2 (1995), 251--275.

\bibitem{schachinger00a}
{\sc Schachinger, W.}
\newblock Limiting distributions for the costs of partial match retrievals in
  multidimensional tries.
\newblock {\em Random Structures Algorithms 17}, 3-4 (2000), 428--459.

\bibitem{sedgewick78a}
{\sc Sedgewick, R.}
\newblock Data movement in odd-even merging.
\newblock {\em SIAM J. Comput. 7}, 3 (1978), 239--272.

\bibitem{shafer96a}
{\sc Shafer, G.}
\newblock The significance of {Jacob Bernoulli's} \emph{Ars Conjectandi} for
  the philosophy of probability today.
\newblock {\em J. Econometrics 75}, 1 (1996), 15--32.

\bibitem{shiau05a}
{\sc Shiau, S.-H., and Yang, C.-B.}
\newblock A fast initialization algorithm for single-hop wireless networks.
\newblock {\em IEICE Transactions 88-B}, 11 (2005), 4285--4292.

\bibitem{simon88a}
{\sc Simon, K.}
\newblock An improved algorithm for transitive closure on acyclic digraphs.
\newblock {\em Theoret. Comput. Sci. 58}, 1-3 (1988), 325--346.

\bibitem{szpankowski90a}
{\sc Szpankowski, W.}
\newblock Patricia tries again revisited.
\newblock {\em J. Assoc. Comput. Mach. 37}, 4 (1990), 691--711.

\bibitem{szpankowski01a}
{\sc Szpankowski, W.}
\newblock {\em Average Case Analysis of Algorithms on Sequences}.
\newblock Wiley-Interscience, New York, 2001.
\newblock With a foreword by Philippe Flajolet.

\bibitem{wagner09a}
{\sc Wagner, S.}
\newblock On tries, contention trees and their analysis.
\newblock {\em Ann. Comb. 12}, 4 (2009), 493--507.

\bibitem{wagner10a}
{\sc Wagner, S.}
\newblock On unary nodes in tries.
\newblock In {\em 21st {I}nternational {M}eeting on {P}robabilistic,
  {C}ombinatorial, and {A}symptotic {M}ethods in the {A}nalysis of {A}lgorithms
  ({A}of{A}'10)}, Discrete Math. Theor. Comput. Sci. Proc., AM. Assoc. Discrete
  Math. Theor. Comput. Sci., Nancy, 2010, pp.~577--589.

\bibitem{yang91a}
{\sc Yang, C.-B.}
\newblock Reducing conflict resolution time for solving graph problems in
  broadcast communications.
\newblock {\em Inf. Process. Lett. 40}, 6 (1991), 295--302.

\end{thebibliography}

\end{document}